\documentclass[a4paper,10pt]{article}
\usepackage{graphicx}

\usepackage{amsmath}
\usepackage{amssymb}
\usepackage{amsfonts}
\usepackage{amsthm}
\usepackage{mathtools}
\usepackage[binary-units=true]{siunitx}
\usepackage{booktabs}
\usepackage[affil-sl]{authblk}

\usepackage{dsfont}
\usepackage{url}
\usepackage{enumerate}
\usepackage[inline]{enumitem}
\usepackage{cite}

\usepackage{xspace}
\usepackage{tikz}
\usepackage{subfigure}
\usepackage{longtable}

\usepackage[margin=3cm]{geometry}

\definecolor{darkred}{RGB}{150, 0, 0}
\definecolor{darkgreen}{RGB}{0, 150, 0}
\definecolor{darkblue}{RGB}{0, 0, 150}
\usepackage[breaklinks=true,colorlinks,citecolor=darkgreen,linkcolor=darkred,urlcolor=darkblue,bookmarks=false]{hyperref}

\usepackage{array}
\newcolumntype{L}[1]{>{\raggedright\let\newline\\\arraybackslash\hspace{0pt}}m{#1}}
\newcolumntype{R}[1]{>{\raggedleft\let\newline\\\arraybackslash\hspace{0pt}}m{#1}}

\newcommand{\R}{\mathds{R}}
\newcommand{\Z}{\mathds{Z}}
\newcommand{\B}[1]{\{0,1\}^{#1}}

\newcommand{\cI}{\mathcal{I}}
\newcommand{\cC}{\mathcal{C}}

\newcommand{\sym}[1]{\mathcal{S}_{#1}}
\newcommand{\ssym}[1]{\sym{n}^{\pm}}
\newcommand{\perm}{\pi}
\newcommand{\sperm}{\gamma}
\newcommand{\inv}[1]{{#1}^{-1}}
\newcommand{\invperm}{\inv{\perm}}
\newcommand{\invsperm}{\inv{\sperm}}
\newcommand{\dcenter}{\xi}
\newcommand{\refl}{\rho}
\newcommand{\sgroup}{\Gamma}
\newcommand{\id}{\mathrm{id}}

\newcommand{\type}{t}
\newcommand{\colors}{\mathfrak{C}}
\newcommand{\values}{\mathfrak{N}}
\newcommand{\variables}{\mathfrak{V}}
\newcommand{\operators}{\mathfrak{O}}
\newcommand{\anchor}{\mathfrak{a}}
\newcommand{\sqdiff}{\boxminus}

\newcommand{\lexgeq}{\geq_{\mathrm{lex}}}
\newcommand{\lexgt}{>_{\mathrm{lex}}}
\newcommand{\nlexgeq}{\ngeq_{\mathrm{lex}}}

\newcommand{\define}{\coloneqq}
\newcommand{\T}{^{\top}}
\newcommand{\sprod}[2]{{#1}\T{#2}}
\DeclareMathOperator{\sign}{sgn}
\newcommand{\abs}[1]{\lvert #1 \rvert}

\newcommand{\SDG}{SDG\xspace}
\newcommand{\SDGs}{SDGs\xspace}

\newcommand{\solver}[1]{\textsc{#1}\xspace}
\newcommand{\scip}{\solver{SCIP}}
\newcommand{\code}[1]{\texttt{#1}}

\theoremstyle{plain}
\newtheorem{theorem}{Theorem}[section]
\newtheorem{lemma}[theorem]{Lemma}
\newtheorem{proposition}[theorem]{Proposition}

\newtheorem*{claim*}{Claim}

\theoremstyle{definition}
\newtheorem{remark}[theorem]{Remark}
\newtheorem{definition}[theorem]{Definition}
\newtheorem{example}[theorem]{Example}

\begin{document}

\title{Detecting and Handling Reflection Symmetries in Mixed-Integer
  (Nonlinear) Programming}

\author[1]{Christopher Hojny}
\affil[1]{%
  Eindhoven University of Technology\\
  Combinatorial Optimization Group\\
  PO~Box~513\\
  5600 MB Eindhoven, The Netherlands\\
  \emph{email} c.hojny@tue.nl
}

\maketitle

\begin{abstract}
  Symmetries in mixed-integer (nonlinear) programs (MINLP),
  if not handled appropriately, are known to negatively impact the
  performance of (spatial) branch-and-bound algorithms.
  Usually one thus tries to remove symmetries from the problem formulation
  or is relying on a solver that automatically detects and handles
  symmetries.
  While modelers of a problem can handle various kinds of symmetries,
  automatic symmetry detection and handling is mostly restricted to
  permutation symmetries.
  This article therefore develops techniques such that also
  black-box solvers can automatically detect and handle a broader class
  of symmetries.

  Inspired from geometric packing problems such as the kissing number
  problem, we focus on reflection symmetries of MINLPs.
  We develop a generic and easily applicable framework that allows to automatically detect reflection
  symmetries for MINLPs.
  To handle this broader class of symmetries, we discuss
  generalizations of state-of-the-art methods for permutation symmetries,
  and develop dedicated symmetry handling methods for special reflection
  symmetry groups.
  Our symmetry detection framework has been implemented in the open-source
  solver \scip and we provide a comprehensive discussion of the
  implementation.
  The article concludes with a detailed numerical evaluation of our
  symmetry handling methods when solving MINLPs.

  \textbf{Keywords}
  mixed-integer nonlinear programming $\bullet$ permutation symmetry
    $\bullet$ reflection symmetry $\bullet$ automatic symmetry detection $\bullet$
    automatic symmetry handling
\end{abstract}

\section{Introduction}

We consider mixed-integer nonlinear programs (MINLP) of the form
\begin{align}
  \tag{MINLP}\label{eq:MINLP}
  \begin{aligned}
    \min \sprod{c}{x} &&&\\
    g_k(x) &\leq 0, && k \in \{1,\dots,m\},\\
    \ell_i \leq x_i &\leq u_i, && i \in \{1,\dots,n\},\\
    x_i &\in \Z, && i \in \cI,    
  \end{aligned}
\end{align}
where~$m$ and~$n$ are positive integers, $c \in \R^n$, $\cI \subseteq
\{1,\dots,n\}$, and we have, for every~$k \in \{1,\dots,m\}$,
that~$g_k\colon\R^n \to \R$  as well as, for each~$i \in \{1,\dots,n\}$,
that~$\ell_i \in \R \cup \{-\infty\}$ and~$u_i \in \R \cup \{+\infty\}$.
Note that a linear objective~$\sprod{c}{x}$ is without loss of generality as a
nonlinear objective~$f\colon \R^n \to \R$ can be replaced by a single
auxiliary variable~$\alpha$ and an additional constraint~$f(x) - \alpha
\leq 0$.

A major technique for solving MINLPs is spatial
branch-and-bound~\cite{BelottiEtAl2013}.
In a nutshell, the idea is to define a relaxation of~\eqref{eq:MINLP}, which
is easier to solve than the original problem.
This relaxation is then iteratively split into smaller subproblems until an optimal solution
of~\eqref{eq:MINLP} is found or the problem is proven to be infeasible.
But if~\eqref{eq:MINLP} admits symmetries (which will be defined
properly below), also the list of subproblems will arguably
contain symmetric problems.
Branch-and-bound can thus be accelerated by removing
symmetric copies from the list.
This observation has led to various symmetry handling techniques~\cite{BendottiEtAl2021,Bodi2013,DoornmalenHojny2023,DoornmalenHojny2024,Friedman2007,Hojny2020,HojnyPfetsch2019,KaibelEtAl2011,KaibelPfetsch2008,Liberti2008,Liberti2012,Liberti2012a,LibertiOstrowski2014,LinderothEtAl2021,Margot2002,Margot2003,Ostrowski2009,OstrowskiEtAl2011,Salvagnin2018}.
Most of these approaches consider only permutation symmetries,
which roughly speaking reorder entries of a solution vector and seem to be
most relevant for linear problems.
In many MINLPs, however, richer classes of symmetries can arise.

This article's aim is to extend the literature by tools for
reflection symmetries.
We believe that this is an important class of symmetries arising in many
MINLPs, as illustrated in Example~\ref{ex:geompack} below.
Such symmetries have, to the best of our knowledge, mainly been
handled by reformulating models for specific applications.
Our goal is to devise algorithms such that a black-box solver can
automatically detect and handle reflection symmetries.
We particularly aim for methods that handle more symmetries than
model reformulations.
\begin{example}
  \label{ex:geompack}
  Let~$D$ be a positive integer and let be~$B = [-\frac{W}{2},\frac{W}{2}] \times [-\frac{H}{2},\frac{H}{2}]$ be a
  rectangular box of width~$W \geq 0$ and height~$H \geq 0$.
  Consider the problem of finding the largest value~$r$ such
  that~$D$ non-overlapping disks of radius~$r$ can be packed in~$B$.
  This problem can be modeled as an MINLP, cf.~\cite{Khajavirad2017,Szabo2005}:
  \begin{align*}
    \max r &&&\\
    (x_i - x_j)^2 + (y_i - y_j)^2 &\geq 4r^2, && 1 \leq i < j \leq D,\\
    -\tfrac{W}{2} + r \leq x_i &\leq \tfrac{W}{2} - r, && i \in \{1,\dots,D\},\\
    -\tfrac{H}{2} + r \leq y_i &\leq \tfrac{H}{2} - r, && i \in \{1,\dots,D\},\\
  \end{align*}
  where~$(x_i,y_i)$ models the center point of disk~$i \in \{1,\dots,D\}$.

  This problem admits two kinds of symmetries, see
  Figure~\ref{fig:geompack} for an illustration.
  First, all disks are equivalent, i.e., for any permutation~$\pi$
  of~$\{1,\dots,D\}$, replacing~$(x_i,y_i)$ by~$(x_{\pi(i)}, y_{\pi(i)})$
  in a solution of the MINLP leads to an equivalent solution.
  Second, every solution can be reflected along the reflection symmetry
  axes of the box~$B$.
  In formulae, this means one can replace~$(x_i,y_i)$ by~$(-x_i, y_i)$
  for all~$i \in \{1,\dots,D\}$ (and similarly for horizontal reflections) to
  obtain an equivalent solution.
\end{example}
\begin{figure}[t]
  \centering
  \subfigure[An initial solution.]{
    \begin{tikzpicture}[scale=0.7]
      \draw[-,thick] (0,0) -- (5,0) -- (5,3) -- (0,3) -- (0,0);
      \draw[-,fill=red,draw=red] (.95,.95) circle [radius=.9cm];
      \draw[-,fill=blue,draw=blue] (2.4,2.05) circle [radius=.9cm];
      \draw[-,fill=purple,draw=purple] (4.05,.95) circle [radius=.9cm];
    \end{tikzpicture}
  }
  \subfigure[Permutation of colors.\label{fig:geompackB}]{
    \begin{tikzpicture}[scale=0.7]
      \draw[-,thick] (0,0) -- (5,0) -- (5,3) -- (0,3) -- (0,0);
      \draw[-,fill=blue,draw=blue] (.95,.95) circle [radius=.9cm];
      \draw[-,fill=purple,draw=purple] (2.4,2.05) circle [radius=.9cm];
      \draw[-,fill=red,draw=red] (4.05,.95) circle [radius=.9cm];
    \end{tikzpicture}
  }
  \subfigure[Reflection along a symmetry axis (dashed line).\label{fig:geompackC}]{
    \begin{tikzpicture}[scale=0.7]
      \draw[-,thick] (0,0) -- (5,0) -- (5,3) -- (0,3) -- (0,0);
      \draw[-,fill=purple,draw=purple] (.95,.95) circle [radius=.9cm];
      \draw[-,fill=blue,draw=blue] (2.6,2.05) circle [radius=.9cm];
      \draw[-,fill=red,draw=red] (4.05,.95) circle [radius=.9cm];
      \draw[-,dashed] (2.5,0) -- (2.5,1.2);
      \draw[-,dashed,draw=white] (2.5,1.2) -- (2.5,3);
    \end{tikzpicture}
  }
  \caption{Illustration of permutation and reflection symmetries.
    The different disks are indicated by colors.
  }
  \label{fig:geompack}
\end{figure}
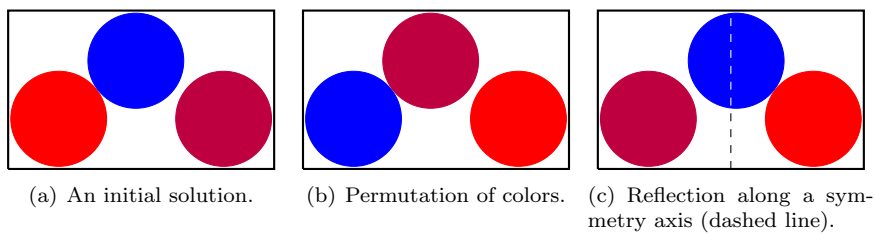
The first class of symmetries corresponds to permutation symmetries, for
which detection mechanisms are described in~\cite{Liberti2012a,Salvagnin2005}.
Except for linear problems~\cite{Bodi2013}, however, to the best of our knowledge
no general mechanism for detecting reflection symmetries in MINLP has been
described.
As a consequence, also the literature on methods for handling reflection
symmetries in general purpose MINLP solvers is limited.
The goals of this article therefore are to
\begin{enumerate}[label={{(G\arabic*)}},ref={G\arabic*},leftmargin={1cm}]
\item\label{goal1} provide a framework for detecting both permutation and
  reflection symmetries in MINLP;
\item\label{goal3} develop an easily extendable open-source tool for
  detecting permutation and reflection symmetries;
\item\label{goal2} derive symmetry handling methods for reflection
  symmetries;
\item\label{goal4} evaluate how frequently reflection symmetries arise in
  MINLP and by how much our methods accelerate the solving process.
\end{enumerate}
The main focus will be on achieving Goal~\eqref{goal3}.
In contrast to existing symmetry detection frameworks, \eqref{goal3}
requires a flexible mechanism to encode symmetries of an MINLP that can easily
incorporate custom constraints whose logic is not known to an MINLP solver.
We therefore need to develop an abstract symmetry detection framework for
achieving both~\eqref{goal1} and~\eqref{goal3}.

To achieve our goals, we proceed as follows.
A detailed description of the problem and underlying mathematical concepts
is provided in Section~\ref{sec:problem}.
To achieve~\eqref{goal1}, we provide a detailed description of
existing symmetry detection schemes in Section~\ref{sec:existingdetection}.
A common abstraction of these schemes is provided in
Section~\ref{sec:detection}, and
Sections~\ref{sec:basicdetection} and~\ref{sec:enhanceddetection}
describe two instantiations of our abstract framework to detect reflection
symmetries in MINLP.
Turning to Goal~\eqref{goal3}, Section~\ref{sec:implementation} describes
an implementation of our abstract framework in the open-source MINLP solver
\scip~\cite{bolusani2024scip}.
Our implementation is included in~\scip since version~9.0, and thus,
publicly available.
Section~\ref{sec:handling} provides an overview of existing symmetry
handling methods.
We also describe how these methods can be adapted for reflection symmetries,
cf.~\eqref{goal2}.
We conclude the article in Section~\ref{sec:numerics} with a detailed
analysis of numerical experiments to evaluate the impact of handling
reflection symmetries in MINLP, cf.~\eqref{goal4}.
For a literature review of symmetry detection and handling, we refer the
reader to Section~\ref{sec:existingdetection} and~\ref{sec:handling},
respectively.

\section{Problem Statement and Basic Definitions}
\label{sec:problem}

The arguably most general symmetry of~\eqref{eq:MINLP} is a
map~$\sigma\colon \R^n \to \R^n$ such that, for all~$x \in \R^n$, one has
\begin{enumerate*}[label={{(S\arabic*)}},ref={S\arabic*}]
\item\label{symdef1} $\sigma(x)$ is feasible for~\eqref{eq:MINLP} if and
  only if~$x$ is feasible, and
\item\label{symdef2} $\sprod{c}{\sigma(x)} = \sprod{c}{x}$.
\end{enumerate*}
Towards the detection of symmetries, however, this definition might be too
general as it is unclear how a general detection mechanism could look like.
One therefore usually restricts to subclasses of symmetries, the most
popular class being permutation symmetries.

Let~$n$ be a positive integer and define~$[n] \define \{1,\dots,n\}$.
Any bijective map~$\perm\colon [n] \to [n]$ is called a \emph{permutation}
of~$[n]$.
The set of all permutations of~$[n]$ is denoted~$\sym{n}$ and forms a
group w.r.t.\ composition, the \emph{symmetric group}.
A permutation~$\perm \in \sym{n}$ naturally acts on~$\R^n$ as~$\perm(x) =
(x_{\invperm(1)}, \dots, x_{\invperm(n)})$, i.e., $\perm$ permutes the
coordinates of a vector~$x \in \R^n$.
Note that the inverses are necessary to ensure that this indeed defines a
group action on~$\R^n$.
We say that~$\pi \in \sym{n}$ is a \emph{permutation symmetry}
of~\eqref{eq:MINLP} if it satisfies~\eqref{symdef1} and~\eqref{symdef2}.
The group consisting of all permutation symmetries of~\eqref{eq:MINLP} is
called its \emph{permutation symmetry group}.

Despite being the most popular class of symmetries in the mixed-integer
(nonlinear) programming literature, computing all permutation symmetries
is NP-hard already for binary linear programs~\cite{Margot2010}.
For general MINLP, it is even undecidable if~$\pi$ defines a permutation
symmetry, cf.~\cite{Liberti2012a}.
For integer programs, the reason is that symmetries are defined based on
the feasible region of~\eqref{eq:MINLP}, which is not known explicitly, and
for MINLP the difficulty arises that determining feasibility of a set of
linear equations is undecidable~\cite{Zhu2006}.
A possible remedy is thus to restrict to permutation symmetries that keep a
particular formulation of~\eqref{eq:MINLP} invariant.
We discuss this approach in more detail in
Section~\ref{sec:existingdetection} and turn to the definition of
reflection symmetries next.

The basis for our definition of reflection symmetries are signed
permutations.
We call a map $\sperm\colon \{\pm 1,\dots,\pm n\} \to \{\pm 1,\dots,\pm
n\}$ \emph{signed permutation} if it is bijective and
satisfies~$\sperm(-i) = -\sperm(i)$ for every~$i \in \{\pm 1,\dots, \pm
n\}$.
Signed permutations form a group under composition, the \emph{signed symmetric
  group}, denoted~$\ssym{n}$.
Analogously to permutations, we define a group action on~$\R^n$ via
\[
  \sperm(x)
  =
  \left(\sign(\invsperm(1))x_{\abs{\invsperm(1)}},\dots,\sign(\invsperm(n))x_{\abs{\invsperm(n)}}\right),
\]
where~$\sign\colon \R \to \{0, \pm 1\}$ denotes the sign operator.
That is, due to taking the absolute values in the indices, $\sperm$ reorders
the entries of~$x$ like a permutation, but can also change the sign of some
entries.
A signed permutation thus reorders the entries of a vector and reflects it
along some of the standard hyperplanes with normal vectors being the
standard unit vectors.
\begin{example}[Example~\ref{ex:geompack} continued]
  \label{ex:geompack2}
  Consider the situation of Example~\ref{ex:geompack} with three disks as
  illustrated in Figure~\ref{fig:geompack}.
  Collect all variables in a common vector~$z =
  (x_1,x_2,x_3,y_1,y_2,y_3,r)$, where~$(x_1,y_1)$, $(x_2,y_2)$,
  and~$(x_3,y_3)$ corresponds to the red, blue, and purple disk,
  respectively.
  The permutation symmetry of Figure~\ref{fig:geompackB} corresponds
  to~$\perm = (1,3,2)(4,6,5)$, where the cycle notation denotes that~$z_1$
  is mapped to~$z_3$, $z_3$ is mapped to~$z_2$, and~$z_2$ is mapped
  to~$z_1$ etc.
  Indeed, $\perm(z) = (z_{\invperm(1)},\dots,z_{\invperm(7)})$ is given by
  \[
    \perm(z)
    =
    (z_3, z_2, z_1, z_6,z_5,z_4,z_7)
    =
    (x_3,x_2,x_1,y_3,y_2,y_1,r),
  \]
  which corresponds to the exchange of colored disks.
  The reflection symmetry of Figure~\ref{fig:geompackC} corresponds to the
  signed permutation~$\sperm = (1,-1)(2,-2)(3,-3)$, where the cycle
  notation now also takes signs into account.
  Then, symmetric solution $\sperm(z) =
  (-x_1,-x_2,-x_3,y_1,y_2,y_3,r)$ indeed reflects all $x$-coordinates.
\end{example}
While signed permutations capture reflections along standard
hyperplanes, they do not allow for reflections along general affine
hyperplanes with standard normal vectors.
For example, if the box in Example~\ref{ex:geompack} is not centered at the
origin, but is given by~$[0,W] \times [0,H]$, no signed
permutation can express the reflection~$(x,y) \mapsto (W - x, y)$ along~$x
= \frac{W}{2}$.
To also capture such symmetries, we introduce another action
of~$\sperm \in \ssym{n}$ on~$\R^n$ that is adapted to an MINLP.

We still propose to encode our more general reflection symmetries via
signed permutations.
But instead of necessarily reflecting a variable~$x_i$ at the origin, we reflect a
variable at the center of its domain~$\ell_i \leq x_i \leq u_i$ as given by
the bounds of~\eqref{eq:MINLP}.
Note, however, that the center is not well-defined in
case~$[\ell_i,u_i]$ defines a half-open interval.
We therefore introduce the set
\[
  \cC
  \define
  \{i \in [n] : \abs{\ell_i} = u_i = \infty \text{ or both $\ell_i$ and $u_i$ are finite}\}
\]
containing the indices of variables whose domain has a well-defined
center.
For~$i \in \cC$, the \emph{center} of variable~$x_i$ is~$\dcenter_i
= \frac{\ell_i + u_i}{2}$, where we define~$-\infty + \infty = 0$.
For~$i \notin \cC$, we define~$\dcenter_i = 0$.
With this notion, the reflection of variable~$x_i$, $i \in [n]$, is
given by~$x_i \mapsto 2 \dcenter_i - x_i$.
If~$i \in \cC$, variable~$x_i$ is reflected at its domain center; in
particular, a variable can be reflection-symmetric to itself.
For~$i \in [n] \setminus \cC$, a variable is reflected at the origin.
\begin{definition}
  \label{def:admissable}
  Consider~\eqref{eq:MINLP} with lower and upper bounds~$\ell$ and~$u$,
  respectively, used to derive~$\cC$ and the corresponding center points.
  For~$\sperm \in \ssym{n}$, the
  \emph{reflection}~$\refl\colon\R^n \to \R^n$, for~$i \in [n]$,
  is
  \[
    \refl(x; \sperm)_i
    =
    \dcenter_i + \sign(\invsperm(i)) \cdot (x_{\abs{\invsperm(i)}} - \dcenter_{\abs{\invsperm(i)}}).
  \]
\end{definition}
That is, if~$\sign(\invsperm(i)) < 0$, then the pre-image of~$\refl(x;
\sperm)_i$ is reflected according to the map $x_{\abs{\invsperm(i)}} \mapsto 2
\dcenter_{\abs{\invsperm(i)}} - x_{\abs{\invsperm(i)}}$ and afterwards
centered around~$\dcenter_i$ via the translation~$\dcenter_i -
\dcenter_{\abs{\invsperm(i)}}$.
Otherwise, if~$\sign(\invsperm(i)) > 0$, $\refl(x; \sperm)_i$ arises by
an ordinary permutation and applying the same translation as before.
\begin{definition}
  \label{def:reflsym}
  A signed permutation~$\sperm \in \ssym{n}$ defines a \emph{reflection
    symmetry} of~\eqref{eq:MINLP} if~$\refl(\cdot; \sperm)$
  satisfies~\eqref{symdef1} and~\eqref{symdef2}.
\end{definition}
Our aim is to achieve Goals~\eqref{goal1}--\eqref{goal4}
for reflection symmetries of~\eqref{eq:MINLP}.

We close this section by showing that reflections
define an action on~$\R^n$, i.e., for~\mbox{$\sperm_1,\sperm_2 \in \ssym{n}$},
their composition~$\sperm_2 \circ \sperm_1$ satisfies~$\refl(x; \sperm_2
\circ \sperm_1) = \refl(\refl(x; \sperm_1); \sperm_2)$, and, if~$\id$ denotes the
identity, then~$\refl(x; \id) = x$.
Thus, our definition of reflections is independent from whether one first
composes signed permutations and applies a single reflection or
whether one applies a series of reflections for the corresponding signed
permutations.
\begin{lemma}
  Let~$\sgroup \leq \ssym{n}$ and let~$\dcenter_i$, $i \in [n]$, be the
  reflection center for the~$i$-th coordinate.
  Then, $\refl(\cdot; \sperm)$, $\sperm \in \sgroup$, defines a group
  action on~$\R^n$.
\end{lemma}
\begin{proof}
  Let~$x \in \R^n$.
  Then, $\refl(x; \id) = x$ follows immediately.
  It thus remains to show for~$\sperm_1,\sperm_2 \in \ssym{n}$
  that~$\refl(x; \sperm_2 \circ \sperm_1) = \refl(\refl(x; \sperm_1);
  \sperm_2)$.
  To this end, observe that, for any~$i \in [n]$, we have
  \begin{equation}
    \label{eq:auxAction}
    \abs{\invsperm_1(\abs{\invsperm_2(i)})}
    =
    \abs{(\invsperm_1 \circ \invsperm_2)(i)}
    =
    \abs{\inv{(\sperm_2 \circ \sperm_1)}(i)}
  \end{equation}
  as signed permutations~$\sperm$ satisfy~$\sperm(-j) = -\sperm(j)$
  for all~$j \in [n]$.
  The same relation also implies
  \begin{equation}
    \label{eq:auxAction2}
    \sign(\invsperm_2(i)) \cdot \sign(\invsperm_1(\abs{\invsperm_2(i)}))
    =
    \sign(\invsperm_1 \circ \invsperm_2(i))
    =
    \sign(\inv{(\sperm_2 \circ \sperm_1)}(i))
  \end{equation}
  since~$y = \sign(y) \cdot \abs{y}$ for every~$y \in \R$.

  Let~$\sperm = \sperm_2 \circ \sperm_1$.
  Then,
  \begin{align*}
    \refl(x; \sperm_1)_{\abs{\invsperm_2(i)}}
    &=
      \dcenter_{\abs{\invsperm_2(i)}}
      +
      \sign(\invsperm_1(\abs{\invsperm_2(i)}))
      \cdot
      (
      x_{\abs{\invsperm_1(\abs{\invsperm_2(i)})}}
      -
      \dcenter_{\abs{\invsperm_1(\abs{\invsperm_2(i)})}}
      )\\
    &\overset{\eqref{eq:auxAction}}{=}
      \dcenter_{\abs{\invsperm_2(i)}}
      +
      \sign(\invsperm_1(\abs{\invsperm_2(i)}))
      \cdot
      (
      x_{\abs{\invsperm(i)})}
      -
      \dcenter_{\abs{\invsperm(i)})}
      ),
  \end{align*}
  and thus,
  \begin{align*}
    \refl(\refl(x; \sperm_1); \sperm_2)_i
    &=
      \dcenter_{i}
      +
      \sign(\invsperm_2(i)) \cdot (\refl(x; \sperm_1)_{\abs{\invsperm_2(i))}} - \dcenter_{\abs{\invsperm_2(i))}})\\
    &=
      \dcenter_i
      +
      \sign(\invsperm_2(i)) \cdot ( \dcenter_{\abs{\invsperm_2(i)}} - \dcenter_{\abs{\invsperm_2(i)}})\\
    &\phantom{\;= \dcenter_i}
      +
      \sign(\invsperm_2(i)) \cdot \sign(\invsperm_1(\abs{\invsperm_2(i)}))
      \cdot
      (
      x_{\abs{\invsperm(i)})}
      -
      \dcenter_{\abs{\invsperm(i)})}
      )\\
    &\overset{\eqref{eq:auxAction2}}{=}
      \dcenter_i + \sign(\invsperm(i)) \cdot (x_{\abs{\invsperm(i)})} - \dcenter_{\abs{\invsperm(i)})})\\
    &=
      \refl(x; \sperm).
  \end{align*}
  Thus, $\refl(\cdot; \sperm)$ defines a group action.
  \end{proof}

\section{Existing Symmetry Detection Schemes}
\label{sec:existingdetection}

In preparation for our symmetry detection framework, this section reviews
four schemes discussed in the literature.
These schemes have been developed to detect permutation symmetries of
mixed-integer linear programs (MILP), signed permutation symmetries of
MILPs, reflection symmetries in satisfiability problems (SAT), and
permutation symmetries of MINLPs, respectively.

\paragraph{Permutation Symmetries of MILP.}

Since already deciding whether~$\sym{n}$ is the permutation symmetry group
of an MILP is NP-complete~\cite{Margot2010}, one usually only considers
symmetries that keep the formulation of an MILP invariant.
Given an MILP
\[
  \min \{ \sprod{c}{x} : Ax \leq b,\; x_i \in \Z \text{ for all } i \in
  \cI,\; x \in \R^n\},
\]
where~$A \in \R^{m \times n}$, $b \in \R^m$, $c \in \R^n$, and~$\cI \subseteq
[n]$, a \emph{formulation symmetry} is a permutation~$\perm \in \sym{n}$
for which there exists~$\perm' \in \sym{m}$ such that
\begin{itemize}
\item $\perm(c) = c$ and~$\perm'(b) = b$;
\item for all~$(i,j) \in [m] \times [n]$, one has~$A_{\perm'(i),\perm(j)} =
  A_{i,j}$;
\item $\perm$ keeps~$\cI$ invariant.
\end{itemize}

For detecting formulation symmetries, \cite{Salvagnin2005} observes that
the formulation symmetry group of an MILP is isomorphic to the
color-preserving automorphism group of the following graph.
This graph contains a node for every variable as well as each constraint
from~$Ax \leq b$.
Every variable node receives a color, which is determined based on the
variable's type (objective coefficient, an (non-) integrality); every
constraint node gets a color based on the constraint's right-hand side.
Moreover, there is an edge between the node of variable~$x_j$ and the node
of the~$i$-th constraint if and only if~$A_{i,j} \neq 0$.
The edge is colored based on the coefficient~$A_{i,j}$.
An illustration of the symmetry detection graph for the MILP
\begin{alignat}{11}
  \label{eq:exIllustr}
  \begin{aligned}
  \min &&   x_1&& -&& x_2&& + && 2x_3&& +&& 2x_4 &&&&\\
       &&      &&  &&    &&   && x_3&& + &&x_4   &&\leq&& 1\\
       &&- x_1 && +&& x_2&& + && 3x_3&&  &&      &&\leq&& 4\\
       &&- x_1 && +&& x_2&&   &&    &&+  &&3x_4  &&\leq&& 4
  \end{aligned}
\end{alignat}
is given in Figure~\ref{fig:exIllustrA}.

Automorphisms of such symmetry detection graphs can be found by graph
automorphism tools such as \solver{bliss}~\cite{bliss},
\solver{dejavu}~\cite{dejavu}, or \solver{nauty}~\cite{nauty}.
We remark that these tools can only handle node colors, but not edge
colors.
The graphs thus need to be manipulated by replacing an edge~$\{u,v\}$ by an
auxiliary node~$w$, which receives the color of~$\{u,v\}$, and the
edges~$\{u,w\}$ and~$\{v,w\}$.
To reduce the number of auxiliary nodes, grouping techniques can
identify auxiliary nodes of several edges with each other, see~\cite{PfetschRehn2019}
and Section~\ref{sec:technicalDetails}.

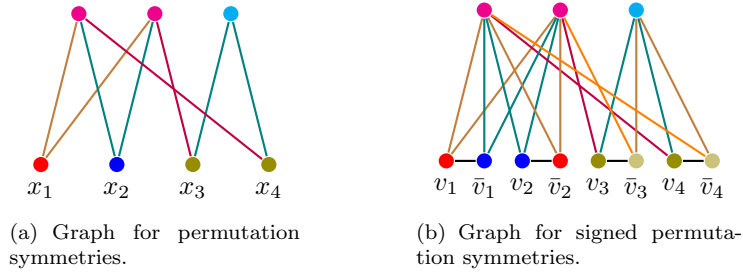
\begin{figure}[t]
  \centering
  \subfigure[Graph for permutation symmetries.\label{fig:exIllustrA}]{
    \begin{tikzpicture}[scale=1.0]
      \tikzstyle{v} += [circle,thick,inner sep=2pt,minimum size=2mm];

      \node (x1) at (0,0) [v,fill=red,label=below:$x_1$]{};
      \node (x2) at (1,0) [v,fill=blue,label=below:$x_2$]{};
      \node (x3) at (2,0) [v,fill=olive,label=below:$x_3$]{};
      \node (x4) at (3,0) [v,fill=olive,label=below:$x_4$]{};

      \node (c1) at (2.5,2) [v,fill=cyan] {};
      \node (c2) at (1.5,2) [v,fill=magenta] {};
      \node (c3) at (0.5,2) [v,fill=magenta] {};

      \draw[-,thick,draw=teal] (c1) -- (x3);
      \draw[-,thick,draw=teal] (c1) -- (x4);
      \draw[-,thick,draw=brown] (c2) -- (x1); 
      \draw[-,thick,draw=teal] (c2) -- (x2);
      \draw[-,thick,draw=purple] (c2) -- (x3);
      \draw[-,thick,draw=brown] (c3) -- (x1); 
      \draw[-,thick,draw=teal] (c3) -- (x2);
      \draw[-,thick,draw=purple] (c3) -- (x4);
   \end{tikzpicture}
 }
 \qquad\qquad
  \subfigure[Graph for signed permutation
  symmetries.\label{fig:exIllustrB}]{
    \begin{tikzpicture}
      \tikzstyle{v} += [circle,thick,inner sep=2pt,minimum size=2mm];

      \node (x1) at (0.0,0) [v,fill=red,label=below:$v_1$]{};
      \node (x2) at (1.0,0) [v,fill=blue,label=below:$v_2$]{};
      \node (x3) at (2.0,0) [v,fill=olive,label=below:$v_3$]{};
      \node (x4) at (3.0,0) [v,fill=olive,label=below:$v_4$]{};
      \node (v1) at (0.5,0) [v,fill=blue,label=below:$\bar{v}_1$]{};
      \node (v2) at (1.5,0) [v,fill=red,label=below:$\bar{v}_2$]{};
      \node (v3) at (2.5,0) [v,fill=olive!50,label=below:$\bar{v}_3$]{};
      \node (v4) at (3.5,0) [v,fill=olive!50,label=below:$\bar{v}_4$]{};

      \draw[-,thick] (x1) -- (v1);
      \draw[-,thick] (x2) -- (v2);
      \draw[-,thick] (x3) -- (v3);
      \draw[-,thick] (x4) -- (v4);

      \node (c1) at (2.5,2) [v,fill=cyan] {};
      \node (c2) at (1.5,2) [v,fill=magenta] {};
      \node (c3) at (0.5,2) [v,fill=magenta] {};

      \draw[-,thick,draw=teal] (c1) -- (x3);
      \draw[-,thick,draw=teal] (c1) -- (x4);
      \draw[-,thick,draw=brown] (c2) -- (x1); 
      \draw[-,thick,draw=teal] (c2) -- (x2);
      \draw[-,thick,draw=purple] (c2) -- (x3);
      \draw[-,thick,draw=brown] (c3) -- (x1); 
      \draw[-,thick,draw=teal] (c3) -- (x2);
      \draw[-,thick,draw=purple] (c3) -- (x4);      

      \draw[-,thick,draw=brown] (c1) -- (v3);
      \draw[-,thick,draw=brown] (c1) -- (v4);
      \draw[-,thick,draw=teal] (c2) -- (v1); 
      \draw[-,thick,draw=brown] (c2) -- (v2);
      \draw[-,thick,draw=orange] (c2) -- (v3);
      \draw[-,thick,draw=teal] (c3) -- (v1); 
      \draw[-,thick,draw=brown] (c3) -- (v2);
      \draw[-,thick,draw=orange] (c3) -- (v4);      
    \end{tikzpicture}
  }
  \caption{Illustration of symmetry detection graphs for
    Problem~\eqref{eq:exIllustr}.}
  \label{fig:exIllustr}
\end{figure}

\paragraph{Signed Permutation Symmetries of MILP.}

For detecting formulation symmetries of an MILP corresponding
to signed permutations, \cite{Bodi2013} also suggests to find automorphisms
of a colored graph.
For every variable~$x_j$, their graph contains two nodes~$v_j$
and~$\bar{v}_j$, where~$v_j$ will represent the original variable~$x_j$
and~$\bar{v}_j$ the negation~$-x_j$.
Moreover, the~$i$-th constraint is represented by a node~$w_i$, and every
entry~$A_{i,j}$ of the constraint~$A$ gets two nodes~$a_{i,j}$
and~$\bar{a}_{i,j}$.
Finally, there are three groups of nodes representing the numerical
coefficients of the MILP: one group consisting of the right-hand side
values~$\{b_1,\dots,b_m\}$, another group consisting of the original and
negated objective coefficients~$\{\pm c_1, \dots, \pm c_n\}$, and the last
group consisting of the original and negated matrix entries~$\{\pm A_{i,j}
: (i,j) \in [m] \times [n]\}$.
The edge set is given by
\begin{align*}
  &\left\{\{v_j, c_j\}, \{\bar{v}_j, -c_j\} : j \in [n] \right\}\\
  \cup & \left\{ \{w_i, b_i\} : i \in [m] \right\}\\
  \cup &
         \left\{\{a_{i,j}, A_{i,j}\},
         \{v_j, a_{i,j}\},
         \{w_i, a_{i,j}\}
         :
         (i,j) \in [m] \times [n] \right\}\\
  \cup &
         \left\{\{\bar{a}_{i,j}, -A_{i,j}\},
         \{\bar{v}_j, \bar{a}_{i,j}\},
         \{w_i, \bar{a}_{i,j}\}
         :
         (i,j) \in [m] \times [n] \right\}\\
  \cup &
         \left\{ \{v_j, \bar{v}_j\} : j \in [n] \right\}
\end{align*}
In contrast to~\cite{Salvagnin2005}, edges remain uncolored.
Instead, there is a unique color for each of the variable, constraint, and
coefficients groups~$\{v_j, \bar{v}_j : j \in [n]\}$,
${\{w_i : i \in [n]\}}$, and~$\{a_{i,j}, \bar{a}_{i,j} : (i,j) \in [m] \times [n]\}$,
respectively; nodes representing a numerical value, receive a color
corresponding to their numerical value.

The idea of this graph is to not only encode the original constraint
system as in~\cite{Salvagnin2005}, but also the negated constraint system.
The last class of edges~$\{v_j, \bar{v}_j\}$, $j \in [n]$, ensures the
property that~$\sperm(-j) = - \sperm(j)$ for a signed permutation~$\sperm$.

Note that the graph can be represented more compactly:
On the one hand, numerical value nodes can be removed and instead nodes
adjacent with them can receive their color.
On the other hand, nodes~$a_{i,j}$ and~$\bar{a}_{i,j}$ can be removed, and
instead one connects~$v_j$ and~$w_i$ by an edge colored according
to~$A_{i,j}$ (analogously for negated variable nodes~$\bar{v}_j$).
One can thus interpret nodes~$a_{i,j}$ and~$\bar{a}_{i,j}$ as the auxiliary
nodes that arise when substituting colored edges in the graph
of~\cite{Salvagnin2005}.
As such, the grouping techniques described in~\cite{PfetschRehn2019} can
also be used to reduce the number of needed nodes~$a_{i,j}$ and~$\bar{a}_{i,j}$.
Figure~\ref{fig:exIllustrB} illustrates the more compact graph,
better revealing relations to the graph by~\cite{Salvagnin2005}.

\paragraph{Reflection Symmetries for SAT}

Consider variables~$x_1,\dots,x_n \in \B{}$.
The reflection of~$x_j$ is given by~$\bar{x}_j = 1 - x_j$,
cf. Definition~\ref{def:admissable}.
Using MINLP notation, a SAT formula can be represented by a set
of~$m$ linear constraints
\[
  \sum_{j \in J_i^+} x_j + \sum_{j \in J_i^-} \bar{x}_j \geq 1,
\]
where~$J_i^+, J_i^- \subseteq [n]$ are disjoint for every~$i \in [m]$, and
the objective being constant~0.
That is, SAT problems are pure feasibility problems.

A common way to detect reflection symmetries for SAT is described
in~\cite{Sakallah2021}.
As in~\cite{Bodi2013}, there are nodes~$v_j$ and~$\bar{v}_j$ to
represent variable~$x_j$ and its reflection~$\bar{x}_j$ for~$j \in [n]$,
and there is a node~$w_i$ for every~$i \in [m]$; the edge set
is
\[
  \left\{ \{w_i,v_j\} : i \in [m], j \in J_i^+\right\}
  \cup
  \left\{ \{w_i,\bar{v}_j\} : i \in [m], j \in J_i^-\right\}
  \cup
  \left\{ \{v_j, \bar{v}_j : j \in [n]\} \right\}.
\]
The role of the last group of edges is again to make sure that if a
symmetry maps~$v_j$ onto~$v_{j'}$ or~$\bar{v}_{j'}$, then~$\bar{v}_{j'}$ is
mapped onto~$\bar{v}_{j}$ or~$v_{j}$, respectively.

In contrast to the previous graphs, no colors are needed since there is
only one type of coefficients.
Furthermore, this graph allows to detect reflection symmetries, because all
variables have the same domain.
Despite the similarities to the previous graphs for MILP, the latter
cannot immediately be used for detecting reflection symmetries in
MILP as variables might have different domains.
We resolve this issue in our symmetry detection scheme in
Section~\ref{sec:detection}.

\paragraph{Permutation Symmetries in MINLP}

As discussed above, one of the difficulties for detecting symmetries
of~\eqref{eq:MINLP} is that checking whether two nonlinear equations are
equivalent is undecidable.
The symmetry detection scheme of~\cite{Liberti2012a} resolves this issue by
restricting to nonlinear constraints~$g_k(x)$, $k \in [m]$, that admit a
representation via expression trees~\cite{Cohen2003}.
An \emph{expression tree} for~$g_k$ is an arborescence~$T_k$ in
which all arcs point away from the root node.
Each leaf node belongs to exactly one of two classes, it is either a variable
or value node, and the non-leaf nodes are operator nodes.
Each variable node corresponds to a variable present in~$g_k$ and each
value node holds a numerical value.
Every operator node~$v$ corresponds to an~$d$-ary mathematical operator,
where~$d$ is the number of children of~$v$.
The nonlinear function~$g_k$ can then be recovered from~$T_k$ by
iteratively evaluating the nodes of~$T_k$, where evaluating a node means to
assign it a mathematical function.
Leaf nodes evaluate to the corresponding variables or numerical values.
Operator nodes~$v$ are evaluated by applying the operator of~$v$ to the
functions assigned to the children of~$v$.
By evaluating the nodes in a bottom-up fashion, the last processed node is
the root node, whose assigned function is~$g_k$, see
Figure~\ref{fig:exprTree} for an illustration.

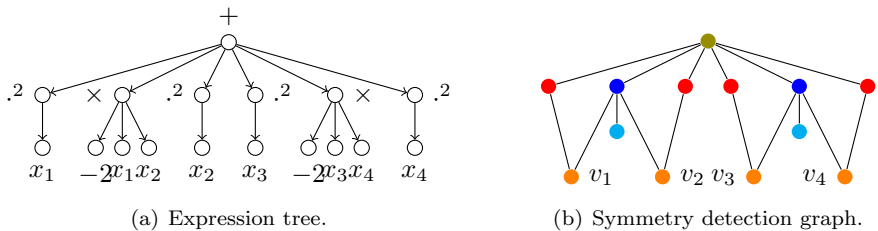
\begin{figure}[t]
  \centering
  \subfigure[Expression tree.\label{fig:exprTree}]{
    \begin{tikzpicture}[scale=0.7]
      \tikzstyle{v} += [circle,inner sep=2pt,minimum size=2mm,draw=black];

      \node (r) at (0,0) [v,label=above:$+$] {};
      \node (l11) at (-3.5,-1) [v,label=left:$\cdot^2$] {};
      \node (l12) at (-2.0,-1) [v,label=left:$\times$] {};
      \node (l13) at (-0.5,-1) [v,label=left:$\cdot^2$] {};
      \node (l14) at (0.5,-1) [v,label=right:$\cdot^2$] {};
      \node (l15) at (2.0,-1) [v,label=right:$\times$] {};
      \node (l16) at (3.5,-1) [v,label=right:$\cdot^2$] {};

      \draw[->] (r) -- (l11);
      \draw[->] (r) -- (l12);
      \draw[->] (r) -- (l13);
      \draw[->] (r) -- (l14);
      \draw[->] (r) -- (l15);
      \draw[->] (r) -- (l16);

      \node (l21) at (-3.5,-2) [v,label=below:$x_1$] {};
      \node (l22) at (-2.5,-2) [v,label=below:$-2$] {};
      \node (l23) at (-2.0,-2) [v,label=below:$x_1$] {};
      \node (l24) at (-1.5,-2) [v,label=below:$x_2$] {};
      \node (l25) at (-0.5,-2) [v,label=below:$x_2$] {};
      \node (l26) at (0.5,-2) [v,label=below:$x_3$] {};
      \node (l27) at (1.5,-2) [v,label=below:$-2$] {};
      \node (l28) at (2.0,-2) [v,label=below:$x_3$] {};
      \node (l29) at (2.5,-2) [v,label=below:$x_4$] {};
      \node (l210) at (3.5,-2) [v,label=below:$x_4$] {};

      \draw[->] (l11) -- (l21);
      \draw[->] (l12) -- (l22);
      \draw[->] (l12) -- (l23);
      \draw[->] (l12) -- (l24);
      \draw[->] (l13) -- (l25);

      \draw[->] (l14) -- (l26);
      \draw[->] (l15) -- (l27);
      \draw[->] (l15) -- (l28);
      \draw[->] (l15) -- (l29);
      \draw[->] (l16) -- (l210);
    \end{tikzpicture}
  }
  \qquad
  \subfigure[Symmetry detection graph.\label{fig:exprTreeDetect}]{
    \begin{tikzpicture}[scale=0.6]
      \tikzstyle{v} += [circle,inner sep=2pt,minimum size=2mm];

      \node (r) at (0,0) [v,fill=olive] {};
      \node (l11) at (-3.5,-1) [v,fill=red] {};
      \node (l12) at (-2.0,-1) [v,fill=blue] {};
      \node (l13) at (-0.5,-1) [v,fill=red] {};
      \node (l14) at (0.5,-1) [v,fill=red] {};
      \node (l15) at (2.0,-1) [v,fill=blue] {};
      \node (l16) at (3.5,-1) [v,fill=red] {};

      \draw[-] (r) -- (l11);
      \draw[-] (r) -- (l12);
      \draw[-] (r) -- (l13);
      \draw[-] (r) -- (l14);
      \draw[-] (r) -- (l15);
      \draw[-] (r) -- (l16);

      \node (val1) at (-2,-2) [v,fill=cyan] {};
      \node (val2) at (2,-2) [v,fill=cyan] {};
      \node (x1) at (-3.0,-3) [v,fill=orange,label=right:{$v_1$}] {};
      \node (x2) at (-1.0,-3) [v,fill=orange,label=right:{$v_2$}] {};
      \node (x3) at (1.0,-3) [v,fill=orange,label=left:{$v_3$}] {};
      \node (x4) at (3.0,-3) [v,fill=orange,label=left:{$v_4$}] {};

      \draw[-] (l12) -- (val1);
      \draw[-] (l15) -- (val2);
      \draw[-] (l11) -- (x1);
      \draw[-] (l12) -- (x1);
      \draw[-] (l12) -- (x2);
      \draw[-] (l13) -- (x2);
      \draw[-] (l14) -- (x3);
      \draw[-] (l15) -- (x3);
      \draw[-] (l15) -- (x4);
      \draw[-] (l16) -- (x4);
    \end{tikzpicture}
  }  
  \caption{Expression tree and symmetry detection graph
    for~$x_1^2 - 2x_1x_2 + x_2^2 + x_3^2 - 2x_3x_4 + x_4^2$.}
\end{figure}

Expression trees can also make use of non-commutative operators,
e.g., the minus-operator.
In this case, arcs receive labels to indicate the order of the
input.
For the ease of exposition, we assume all operators to be commutative, but
all results can be easily refined for non-commutative operators.

To detect permutation symmetries for MINLPs whose constraints are
represented via expression trees, \cite{Liberti2012a} suggests to derive a
symmetry detection graph from the expression trees, see
Figure~\ref{fig:exprTreeDetect} for an illustration.
For each variable~$x_j$, $j \in [n]$, construct a node~$v_j$, and, for each
constraint~$g_k$, $k \in [m]$, of~\eqref{eq:MINLP}, construct an
undirected version of~$T_k$.
Remove the leaves corresponding to variables~$x_j$ and
connect their parents to~$v_j$ instead.
That is, the entire graph contains only a single node corresponding
to~$x_j$.
Then, the operators and numerical values are interpreted as distinct
colors.
Finally, color the nodes~$v_j$, $j \in [n]$, according to their type
(same objective coefficient, and lower and upper
bounds).\footnote{Formally, \cite{Liberti2012a} introduces a separate
  expression tree for the objective, but due to our assumption of a linear
  objective, the graph we describe is a bit simpler.}
One can show that every color preserving automorphism of this graph
corresponds to a symmetry of~\eqref{eq:MINLP}.

\section{A Framework for Symmetry Detection in MINLP}
\label{sec:detection}

The core of our symmetry detection framework is an abstract notion of a
symmetry detection graph, generalizing the existing frameworks.
This will be useful, on the one hand, since different classes of
constraints (e.g., linear, nonlinear) might allow for more compact
representations to detect symmetries. On the other hand, this notion is
flexible towards achieving Goal~\eqref{goal3}, the development of
an easily extendable open-source tool.
After defining symmetry detection graphs, Section~\ref{sec:basicdetection}
provides a basic instantiation of the abstract framework to detect reflection
symmetries in MINLPs.
Section~\ref{sec:enhanceddetection} then refines the previous concepts to
detect more symmetries.

Before we define symmetry detection graphs, we introduce some notation.
Consider a list of variables~$x_i$, $i \in [n]$, with lower bounds~$\ell_i
\in \R \cup \{-\infty\}$ and upper bounds~$u_i\in \R \cup \{\infty\}$, as well as
objective coefficient~$c_i \in \R$.
For~$i \in [n]$, let~$\dcenter_i$ be the reflection center of
variable~$x_i$ as defined in Section~\ref{sec:problem}, and denote
by~$x_{-i}$ the reflection of~$x_i$ at~$\dcenter_i$.
That is, variable~$x_{-i}$ has lower bound~$\ell_{-i} =
2\dcenter_i - u_i$, upper bound~$u_{-i} = 2\dcenter_i - \ell_i$, and
objective coefficient~$c_{-i} = -c_i$.
Moreover, for a positive integer~$n$, let~$[\pm n] \define \{\pm 1, \dots,
\pm n\}$.
\begin{definition}
  \label{def:symdetectgraph}
  Let~$P$ be an instance of~\eqref{eq:MINLP}.
  Let~$G$ be a node and edge colored connected graph that contains
  pairwise distinct distinguished nodes~$v_i$, $i \in [\pm n]$.
  We call~$G$ a \emph{symmetry detection graph (\SDG)} for~$P$ if
  each color-preserving automorphism~$\perm$ of~$G$ satisfies:
  \begin{enumerate}
  \item $\perm$ keeps~$\{v_i : i \in [\pm n]\}$ invariant;
  \item if~$\perm(v_i) = v_j$ for some~$i,j \in [\pm n]$,
    then~$\perm(v_{-i}) = v_{-j}$;
  \item the signed permutation~$\sperm \in \ssym{n}$ defines a reflection
    symmetry of~$P$, where, for each~$i \in [\pm n]$,
    $\sperm(i) = j$ for the unique~$j$ with~$\perm(v_i) = v_j$.
  \end{enumerate}
\end{definition}
Note that \SDGs are well-defined:
Due to the first property, the second property is well-defined,
and~$\sperm$ in the third property is indeed a signed permutation due to
the first and second property.

Definition~\ref{def:symdetectgraph} seems to be the right notion for
detecting reflection symmetries as it provides an abstract generalization
of the symmetry detection frameworks of Section~\ref{sec:existingdetection}.
But to use it, one needs a concrete mechanism to build an \SDG.
While black-box solvers can use fixed sets of rules for building \SDGs,
an easily extendable tool needs to be flexible to also
support unknown custom constraints, e.g., nonlinear
constraints not admitting a (simple) representation via expression trees.
We therefore suggest to use an abstract mechanism to build an \SDG from
\SDGs of single constraints as formalized next.
This mechanism is the core of our implementation, which we discuss in
Section~\ref{sec:implementation} in detail.

Let~$P$ be an instance of~\eqref{eq:MINLP} with~$m$
constraints~${g_1,\dots,g_m\colon \R^n \to \R}$.
For~$k \in [m]$, denote by~$P_k$ the MINLP arising from~$P$ by
removing all constraints except for~$g_k$.
To build an \SDG for~$P$ from the \SDGs~$G_1,\dots,G_m$
of~$P_1,\dots,P_m$, our idea is to take the disjoint union
of~$G_1,\dots,G_m$, and to identify the distinguished nodes for
variable~$x_i$, $i \in [\pm n]$, with each other.
To apply this idea, we need to make sure that the colors of the
different \SDGs match.
That is, equivalent objects can only be mapped to equivalent objects.

To model equivalence of variables, we introduce variable types.
For a variable~$x_i$, $i \in [n]$, of~\eqref{eq:MINLP} with lower
bound~$\ell_i$, upper bound~$u_i$, objective coefficient~$c_i$, and a
Boolean encoding if~$i \in \cI$, the \emph{type} of~$x_i$ is
\[
  \type(x_i) \define (\ell_i - \dcenter_i, u_i - \dcenter_i, c_i, i \in \cI).
\]
That is, the type is defined by the lower and upper bounds relative to the
reflection center, the objective coefficient, and integrality status.
Using relative lower and upper bounds instead of absolute bounds will allow
us to detect symmetries also of variables with different domain centers.
The type of~$x_{-i}$ is defined based on the reflected variable, i.e.,
${\type(x_{-i}) = (\dcenter_i - u_i, \dcenter - \ell_i, -c_i, i \in \cI)}$.
We then define the set of variable colors as~$\variables = \{\type(x_i) : i
\in [\pm n]\}$, i.e., each type is associated with a unique color.
An \SDG is called \emph{variable color compatible} if node~$v_i$
associated with variable~$x_i$, $i \in [\pm n]$, is colored by
color~$\type(x_i)$ and none of the non-distinguished nodes receives a color
from~$\variables$.

Besides mapping variables of the same type to each other, we need to
ensure that the set of constraints remains invariant after applying a
reflection symmetry.
This is achieved via the concepts of anchors and constraint
compatibility.
An \SDG~$G = (V,E)$ is called \emph{anchored} if there is a non-distinguished~$a \in V$
such that, for each~$v \in V \setminus \{a\}$, there is a path from~$a$
to~$v$ in~$G$ such that neither of the interior nodes along the path is a
distinguished variable node.
We call~$a$ an \emph{anchor} of~$G$.
For the second concept,
let~$k,k' \in [m]$.
Let~$G_k = (V_k,E_k)$ and~$G_{k'} = (V_{k'}, E_{k'})$ be \SDGs for~$P_k$
and~$P_{k'}$, respectively.
We call~$G_k$ and~$G_{k'}$ \emph{constraint compatible} if they are
non-isomorphic or~$g_k \equiv g_{k'}$,
where $g_k \equiv g_{k'}$ means that there is a permutation~$\perm \in
\sym{n}$ such that~$g_k(x) = g_{k'}(\perm(x))$.
\begin{theorem}
  \label{thm:buildSDG}
  Let~$P$ be an MINLP with~$m$ constraints. Let~$G_k = (V_k,E_k)$, $k
  \in [m]$, be an anchored \SDG for~$P_k$ with distinguished nodes~$v^k_i$,
  $i \in [\pm n]$.
  Suppose~$G_1,\dots,G_m$ are variable color and
  constraint compatible.
  Let~$v_i$, ${i \in [\pm n]}$, be pairwise distinct with color~$\type(x_i)$
  and let
  \begin{align*}
    V =& \{v_i : i \in [\pm n] \} \cup \bigcup_{k = 1}^m V_k \setminus
        \{v^k_i : i \in [\pm n]\},\\
    E =& \left\{ \{v_i, v_{-i}\} : i \in [n] \right\}
        \cup \bigcup_{k = 1}^m \left\{ \{v_i,v\} : \{v_i^k, v\} \in E_k
        \text{ and } v \neq v_{-i}^k
        \right\}\\
       &\cup \bigcup_{k = 1}^m \{ e \in E_k : e \text{ contains no
        distinguished node of } G_k\}.
  \end{align*}
  Then, $G = (V,E)$ is an \SDG for~$P$ with distinguished
  nodes~$v_i$, $i \in [\pm n]$.
\end{theorem}
\begin{proof}
  Let~$\perm$ be a color-preserving automorphism of~$G$.
  We need to show that~$\perm$ satisfies the three properties of
  Definition~\ref{def:symdetectgraph}.
  The first property holds since~$v_i$, $i \in [\pm n]$, are the only
  nodes receiving a variable color from~$\variables$ by variable color
  compatibility of~$G_1,\dots,G_m$.
  The second property is satisfied, because the only edges between the
  distinguished nodes~$v_i$, $i \in [\pm n]$, are the
  edges~$\{v_i,v_{-i}\}$.
  Consequently, if~$\perm(v_i) = v_j$ for some~$j \in [\pm n]$,
  also~$\perm(v_{-i}) = v_{-j}$ holds.

  For verifying the last property, observe that each path in~$G$
  corresponds to a path in~$\perm(G)$.
  This in particular holds for the paths originating from the anchors
  of~$G_1,\dots,G_m$ to the remaining nodes of the respective graphs.
  Consequently, if~$\perm$ maps a
  non-distinguished node of~$G_k$ to a non-distinguished node of~$G_{k'}$
  for~$k,k' \in [m]$, it also maps the remaining non-distinguished
  nodes of~$G_k$ to nodes of~$G_{k'}$.
  As a consequence, $G_k$ and~$G_{k'}$ are isomorphic, because we can
  identify~$v_i$, $i \in [\pm n]$, with both~$v_i^k$ and~$v_i^{k'}$.
  By constraint compatibility of~$G_1,\dots,G_m$ we thus conclude~$g_k
  \equiv g_{k'}$.
  Property three for~$G$ then follows from the same property
  for~$G_1,\dots,G_m$ since we can mutually identify the distinguished
  nodes of~$G_1,\dots,G_m$ with the distinguished nodes of~$G$. \end{proof}
Note that all frameworks of Section~\ref{sec:existingdetection} are
(slight variations of) an instantiation of Theorem~\ref{thm:buildSDG}.
When ignoring colors, the framework of~\cite{Sakallah2021} for SAT problems
follows Theorem~\ref{thm:buildSDG}, and when reflecting variables
at the origin, the framework of~\cite{Bodi2013} for signed permutations for
MILPs can be derived from Theorem~\ref{thm:buildSDG}.
If one ignores reflections and has only distinguished nodes for
non-reflected variables, also the frameworks for MILP~\cite{Salvagnin2005}
and MINLP~\cite{Liberti2012a} follow from Theorem~\ref{thm:buildSDG}.

\subsection{Basic Framework for Detecting Reflection Symmetries in MINLP}
\label{sec:basicdetection}

We now turn the focus to the detection of reflection symmetries of an
instance~$P$ of~\eqref{eq:MINLP} with~$m$ constraints that are represented
by expression trees.
Due to Theorem~\ref{thm:buildSDG}, it is sufficient to find, for
each~$k \in [m]$, an anchored \SDG~$G_k$ for~$P_k$ such that~$G_1,\dots,G_m$
are variable color and constraint compatible.
To this end, we combine the ideas of~\cite{Bodi2013}
and~\cite{Liberti2012a}.
We will illustrate the concepts introduced in this section using the
example
\begin{equation}
  \label{eq:exampleSDG}
  \min\{0:
  4x_1 - 4x_2 + x_3 - x_4\leq 0,\;
  x_1,x_2 \in [-1,1],
  x_3 \in [1,3],
  x_4 \in [-2,0]
  \}.
\end{equation}

For deriving our \SDG, it will be convenient to apply two modifications to
an expression tree~$T$.
First, we transform~\eqref{eq:MINLP} such that all reflection
centers~$\dcenter_i$, $i \in [n]$, are at~$0$.
This can be achieved by replacing all variables~$x_i$ by~$x_i -
\dcenter_i$, as already indicated by the relative bounds of variable
type~$\type(x_i)$.
To still represent the original MINLP, we apply two modifications:
\begin{enumerate*}[label=(\roman*)]
\item  set the lower and upper bounds of~$x_i$ to~$\ell_i - \dcenter_i$
  and~$u_i - \dcenter_i$, respectively, and
\item replace in~$T$ every sub expression tree representing an
  expression~$\alpha \cdot x_i$ for some $\alpha \in \R$ by an expression
  tree for the expression~$\alpha \cdot x_i + \alpha \cdot \dcenter_i$.
\end{enumerate*}
Second, observe that reflecting the expression~$\alpha \cdot x_i$ for
some~$\alpha \in \R$, results in~$\alpha\cdot(2 \dcenter_i - x_i)$.
That is, the sign of the coefficient of~$x_i$ changes.
To easily model this in an \SDG, we encode variable coefficients as edge
colors.
More concretely, suppose~$T$ has an operator node~$v$ that corresponds to
a multiplication operation and that has exactly two children.
Whenever one child~$v_1$ is a variable and the other child~$v_2$ is a
numerical value, we remove~$v_2$ from~$T$ and assign its value to the arc
connecting~$v$ and~$v_1$, see Figure~\ref{fig:illustratePreprocessing} for an illustration of both
operations.

\begin{figure}[t]
  \centering

  \subfigure[Expression tree.]
  {
    \begin{tikzpicture}[scale=0.5, every node/.style={scale=0.7}]
      \tikzstyle{v} += [circle,inner sep=2pt,minimum size=2mm,draw=black];

      \node (r) at (0,0) [v,label=above:$+$] {};
      \node (p1) at (-2.5,-0.75) [v,label=left:$\times$] {};
      \node (p2) at (-1.0,-0.75) [v,label=left:$\times$] {};
      \node (p3) at (1.0,-0.75) [v,label=left:$\times$] {};
      \node (p4) at (2.5,-0.75) [v,label=right:$\times$] {};

      \draw[->] (r) -- (p1);
      \draw[->] (r) -- (p2);
      \draw[->] (r) -- (p3);
      \draw[->] (r) -- (p4);

      \node (p11) at (-3.0,-1.5) [v,label=below:$4$] {};
      \node (p12) at (-2.0,-1.5) [v,label=below:$x_1$] {};
      \node (p21) at (-1.5,-1.5) [v,label=below:$-4$] {};
      \node (p22) at (-0.5,-1.5) [v,label=below:$x_2$] {}; 
      \node (p31) at (0.5,-1.5) [v,label=below:$1$] {};
      \node (p32) at (1.5,-1.5) [v,label=below:$x_3$] {};
      \node (p41) at (2.0,-1.5) [v,label=below:$-1$] {};
      \node (p42) at (3.0,-1.5) [v,label=below:$x_4$] {};

      \draw[->] (p1) -- (p11);
      \draw[->] (p1) -- (p12);
      \draw[->] (p2) -- (p21);
      \draw[->] (p2) -- (p22);
      \draw[->] (p3) -- (p31);
      \draw[->] (p3) -- (p32);
      \draw[->] (p4) -- (p41);
      \draw[->] (p4) -- (p42);
    \end{tikzpicture}
  }
  \quad
  \subfigure[Tree with shifted domains.]
  {
    \begin{tikzpicture}[scale=0.5, every node/.style={scale=0.7}]
      \tikzstyle{v} += [circle,inner sep=2pt,minimum size=2mm,draw=black];

      \node (r) at (0,0) [v,label=above:$+$] {};
      \node (p1) at (-2.5,-0.75) [v,label=left:$\times$] {};
      \node (p2) at (-1.0,-0.75) [v,label=left:$\times$] {};
      \node (p3) at (1.0,-0.75) [v,label=left:$+$] {};
      \node (p4) at (2.5,-0.75) [v,label=right:$+$] {};

      \draw[->] (r) -- (p1);
      \draw[->] (r) -- (p2);
      \draw[->] (r) -- (p3);
      \draw[->] (r) -- (p4);

      \node (p11) at (-3.0,-1.5) [v,label=below:$4$] {};
      \node (p12) at (-2.0,-1.5) [v,label=below:$x_1$] {};
      \node (p21) at (-1.5,-1.5) [v,label=below:$-4$] {};
      \node (p22) at (-0.5,-1.5) [v,label=below:$x_2$] {}; 
      \node (p31) at (0.5,-1.5)  [v,label=below:$2$] {};
      \node (p32) at (1.5,-1.5)  [v,label=left:$\times$] {};
      \node (p41) at (2.0,-1.5)  [v,label=below:$1$] {};
      \node (p42) at (3.0,-1.5)  [v,label=right:$\times$] {};

      \draw[->] (p1) -- (p11);
      \draw[->] (p1) -- (p12);
      \draw[->] (p2) -- (p21);
      \draw[->] (p2) -- (p22);
      \draw[->] (p3) -- (p31);
      \draw[->] (p3) -- (p32);
      \draw[->] (p4) -- (p41);
      \draw[->] (p4) -- (p42);

      \node (q1) at (1.0,-2.5) [v,label=below:$1$] {};
      \node (q2) at (1.5,-2.5) [v,label=below:$x_3$] {};
      \node (q3) at (2.5,-2.5) [v,label=below:$-1$] {};
      \node (q4) at (3.0,-2.5) [v,label=below:$x_4$] {};

      \draw[->] (p32) -- (q1);
      \draw[->] (p32) -- (q2);
      \draw[->] (p42) -- (q3);
      \draw[->] (p42) -- (q4);
    \end{tikzpicture}
  }
  \quad
  \subfigure[Tree with labeled edges (indicated by colors).]
  {
    \begin{tikzpicture}[scale=0.5, every node/.style={scale=0.7}]
      \tikzstyle{v} += [circle,inner sep=2pt,minimum size=2mm,draw=black];
      \node (r) at (0,0) [v,label=above:$+$] {};
      \node (p1) at (-2.5,-0.75) [v,label=left:$x_1$] {};
      \node (p2) at (-1.0,-0.75) [v,label=left:$x_2$] {};
      \node (p3) at (1.0,-0.75) [v,label=left:$+$] {};
      \node (p4) at (2.5,-0.75) [v,label=right:$+$] {};

      \draw[->,thick,draw=red] (r) -- (p1);
      \draw[->,thick,draw=orange] (r) -- (p2);
      \draw[->] (r) -- (p3);
      \draw[->] (r) -- (p4);

      \node (p31) at (0.5,-1.5)  [v,label=below:$2$] {};
      \node (p32) at (1.5,-1.5)  [v,label=below:$x_3$] {};
      \node (p41) at (2.0,-1.5)  [v,label=below:$1$] {};
      \node (p42) at (3.0,-1.5)  [v,label=below:$x_4$] {};

      \draw[->] (p3) -- (p31);
      \draw[->,thick,draw=blue] (p3) -- (p32);
      \draw[->] (p4) -- (p41);
      \draw[->,thick,draw=cyan] (p4) -- (p42);
    \end{tikzpicture}
  } 
 
  \caption{Illustration of the two preprocessing modifications of
    expression tree for~\eqref{eq:exampleSDG}.}
  \label{fig:illustratePreprocessing}
\end{figure}
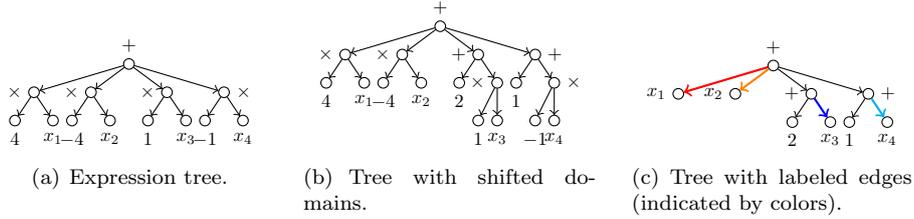

In the following, we assume that all expression trees incorporate the two
previously described modifications.
To construct colors for \SDGs, let~$\values'$ be the set of all
\emph{numerical values} used in the expression trees of MINLP~$P$, and
define~$\values = \{ \pm v : v \in \values'\}$.
Let~$\operators$ be the set of all \emph{operators} appearing in the expression
trees, and let~$\anchor$ be a symbol representing the anchor of a graph.
Moreover, recall that~$\variables$ denotes the set of all types of
variables in~$P$.
Then, $\colors = \values \cup \operators \cup \variables \cup \{\anchor\}$
contains all types of nodes needed to build an \SDG.
We refer to~$\colors$ as \emph{colors}, i.e., we associate a color to
each node based on its type.

Our construction of an \SDG closely follows~\cite{Liberti2012a}.
We derive an undirected copy of the expression tree of constraint~$k$,
denoted~$G'_k = (V'_k, E'_k)$.
Each node in~$V'_k$ receives the color from~$\colors$ corresponding to
its type.
If an arc has received a value due to preprocessing, we color the
corresponding edge in~$E'_k$ by the corresponding color in~$\colors$.
Each uncolored edge incident with a variable node receives
the color of the numerical value~1.
Finally, we add an anchor node~$a$ to~$G'_k$ and replace nodes corresponding
to variables by a gadget.
The anchor node~$a$ receives color~$\anchor$ and is connected via an edge
with the root node of the expression tree.
To represent variable~$x_i$, $i \in [n]$, we introduce two distinguished
nodes~$v^k_i$ and~$v^k_{-i}$ that are colored by~$\type(x_i)$
and~$\type(x_{-i})$, respectively.
Both nodes are connected by an edge~$\{v^k_i, v^k_{-i}\}$.
Furthermore, every edge~$\{v,v_i\} \in E'_k$ with color~$c \in \values$
that is incident to a node~$v_i$ representing variable~$x_i$ is
replaced by two edges~$\{v,v^k_i\}$ and~$\{v,v^k_{-i}\}$ with color~$c$
and~$-c$, respectively.
We denote the resulting graph by~$G_k$, see Figure~\ref{fig:SDG} for an
illustration.
Note that all variable nodes are equally colored as all shifted
variable domains are identical in our example.

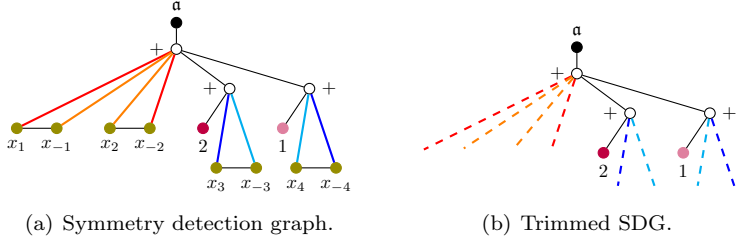
\begin{figure}[t]
  \centering
  
  \subfigure[Symmetry detection graph.\label{fig:SDG}]
  {
    \begin{tikzpicture}[scale=0.7, every node/.style={scale=0.7}]
      \tikzstyle{v} += [circle,inner sep=2pt,minimum size=2mm,draw=black];
      \node (a) at (0,0.5) [v,label=above:$\anchor$,fill=black] {};
      \node (r) at (0,0) [v,label=left:$+$] {};

      \node (p1a) at (-3,-1.5) [v,label=below:$x_1$,fill=olive,draw=olive] {};
      \node (p1b) at (-2.25,-1.5) [v,label=below:$x_{-1}$,fill=olive,draw=olive] {};
      \node (p2a) at (-1.25,-1.5) [v,label=below:$x_2$,fill=olive,draw=olive] {};
      \node (p2b) at (-0.5,-1.5) [v,label=below:$x_{-2}$,fill=olive,draw=olive] {};
      \node (p3) at (1.0,-0.75) [v,label=left:$+$] {};
      \node (p4) at (2.5,-0.75) [v,label=right:$+$] {};

      \draw[-] (r) -- (a);
      \draw[-,thick,draw=red] (r) -- (p1a);
      \draw[-,thick,draw=orange] (r) -- (p1b);
      \draw[-,thick,draw=red] (r) -- (p2b);
      \draw[-,thick,draw=orange] (r) -- (p2a);
      \draw[-] (r) -- (p3);
      \draw[-] (r) -- (p4);
      \draw[-] (p1a) -- (p1b);
      \draw[-] (p2a) -- (p2b);

      \node (p31) at (0.5,-1.5)  [v,label=below:$2$,fill=purple,draw=purple] {};
      \node (p41) at (2.0,-1.5)  [v,label=below:$1$,fill=purple!50,draw=purple!50] {};

      \node (p32a) at (0.75,-2.25) [v,label=below:$x_3$,fill=olive,draw=olive] {};
      \node (p32b) at (1.5,-2.25) [v,label=below:$x_{-3}$,fill=olive,draw=olive] {};
      \node (p42a) at (2.25,-2.25) [v,label=below:$x_4$,fill=olive,draw=olive] {};
      \node (p42b) at (3.0,-2.25) [v,label=below:$x_{-4}$,fill=olive,draw=olive] {};

      \draw[-] (p3) -- (p31);
      \draw[-] (p4) -- (p41);
      \draw[-,thick,draw=blue] (p3) -- (p32a);
      \draw[-,thick,draw=cyan] (p3) -- (p32b);
      \draw[-,thick,draw=cyan] (p4) -- (p42a);
      \draw[-,thick,draw=blue] (p4) -- (p42b);
      \draw[-] (p32a) -- (p32b);
      \draw[-] (p42a) -- (p42b);
    \end{tikzpicture}
  } 
  \quad
  \subfigure[Trimmed \SDG.\label{fig:SDGtrimmed}]
  {
    \begin{tikzpicture}[scale=0.7, every node/.style={scale=0.7}]
      \tikzstyle{v} += [circle,inner sep=2pt,minimum size=2mm,draw=black];
      \node (a) at (0,0.5) [v,label=above:$\anchor$,fill=black] {};
      \node (r) at (0,0) [v,label=left:$+$] {};

      \node (p1a) at (-3,-1.5)  {};
      \node (p1b) at (-2.25,-1.5)  {};
      \node (p2a) at (-1.25,-1.5)  {};
      \node (p2b) at (-0.5,-1.5)  {};
      \node (p3) at (1.0,-0.75) [v,label=left:$+$] {};
      \node (p4) at (2.5,-0.75) [v,label=right:$+$] {};

      \draw[-] (r) -- (a);
      \draw[-,thick,draw=red,dashed] (r) -- (p1a);
      \draw[-,thick,draw=orange,dashed] (r) -- (p1b);
      \draw[-,thick,draw=red,dashed] (r) -- (p2b);
      \draw[-,thick,draw=orange,dashed] (r) -- (p2a);
      \draw[-] (r) -- (p3);
      \draw[-] (r) -- (p4);

      \node (p31) at (0.5,-1.5)  [v,label=below:$2$,fill=purple,draw=purple] {};
      \node (p41) at (2.0,-1.5)  [v,label=below:$1$,fill=purple!50,draw=purple!50] {};

      \node (p32a) at (0.75,-2.25) {};
      \node (p32b) at (1.5,-2.25) {};
      \node (p42a) at (2.25,-2.25) {};
      \node (p42b) at (3.0,-2.25) {};

      \draw[-] (p3) -- (p31);
      \draw[-] (p4) -- (p41);
      \draw[-,thick,draw=blue,dashed] (p3) -- (p32a);
      \draw[-,thick,draw=cyan,dashed] (p3) -- (p32b);
      \draw[-,thick,draw=cyan,dashed] (p4) -- (p42a);
      \draw[-,thick,draw=blue,dashed] (p4) -- (p42b);
    \end{tikzpicture}
  } 

  \caption{Illustration of an \SDG and the corresponding trimmed \SDG for~\eqref{eq:exampleSDG}.}
\end{figure}
\begin{proposition}
  \label{prop:basicSDG}
  Let~$P$ be an MINLP with~$m$ constraints represented by expression
  trees~$T_1,\dots,T_m$ all of whose operators are commutative.
  For~$k \in [m]$, let~$G_k = (V_k,E_k)$ be the
  colored graph derived from~$T_k$ as described above.
  Then, for every~$k \in [m]$, $G_k$ is an anchored \SDG for~$P_k$.
  Moreover, $G_1,\dots,G_m$ are variable color and constraint compatible.
\end{proposition}
\begin{proof}
  For every~$k \in [m]$, graph~$G_k$ is anchored since each variable node
  of~$T_k$ is a leaf and the anchor node of~$G_k$ is connected with the
  root of~$T_k$.
  Moreover, $G_1,\dots,G_m$ are constraint compatible since they are
  essentially copies of the corresponding expression trees.
  The graphs~$G_1,\dots,G_m$ are also variable color compatible, because
  variable nodes are colored according to their type.
  It thus suffices to show that each graph~$G_k$, $k \in [m]$, is an \SDG.

  The first two properties of \SDGs follow immediately from coloring the
  distinguished nodes by the corresponding variables' types and the
  only edges connecting the distinguished nodes being~$\{v^k_i, v^k_{-i}\}$,
  $i \in [n]$.
  Regarding the third property, observe that expression
  tree~$T_k$ is almost identical to~$G_k$.
  The main differences are
  \begin{enumerate*}[label={(\roman*)}, ref={(\roman*)}]
  \item the anchor node in~$G_k$ has no counterpart in~$T_k$;
  \item\label{detectProofB} there are possibly multiple copies of variable
    nodes in~$T_k$, whereas these nodes are identified in~$G_k$ via the
    distinguished nodes;
  \item $G_k$ also contains nodes for~$x_{-1},\dots,x_{-n}$.
  \end{enumerate*}
  We now argue that these differences ensure that~$G_k$ is an \SDG.
  To this end, let~$\bar{T}_k$ be the graph arising from~$T_k$ by
  removing all variable nodes, see Figure~\ref{fig:SDGtrimmed}.
  As for expression trees, we can associate with~$\bar{T}_k$ an
  arithmetic expression in which the input of some operators is not fully
  specified.
  We refer to such an expression as a \emph{pre-function}.

  Let~$\perm$ be an automorphism of~$G_k$ and let~$\sperm$ be the signed
  permutation associated with~$\perm$ as defined in
  Definition~\ref{def:symdetectgraph}.
  Let~$\bar{\perm}$ be the restriction of~$\perm$ onto~$\bar{T}_k$.
  Then, $\bar{\perm}$ is well-defined since every node in~$\bar{T}_k$ has a
  unique counterpart in~$G_k$ (there are no variable nodes in~$\bar{T}_k$).
  We can thus interpret~$\bar{T}_k$ as an induced subgraph of~$G_k$.
  Since the anchor of~$G_k$ is the only node colored by~$\anchor$ and
  the anchor is only connected with~$T_k$'s root, $\bar{\perm}$ is an
  automorphism of~$\bar{T}_k$ that keeps the root invariant.
  The pre-functions associated with~$\perm(\bar{T}_k)$ and~$\bar{T}_k$ are
  hence the same.
  To conclude the proof, we show that inserting~$x$
  and~$\sperm(x)$ into the pre-functions for~$\bar{T}_k$
  and~$\perm(\bar{T}_k)$, respectively, yields the same function.

  Let~$v$ be an operator node of~$T_k$ that has variable~$x_i$, $i \in
  [n]$, as child, and let~$\alpha$ be the numerical value assigned to
  arc~$(v,x_i)$.
  That is, $\alpha \cdot x_i$ is input of the operator associated with~$v$.
  In~$G_k$, arc~$(v,x_i)$ corresponds to the edge~$\{v, v^k_i\}$ with
  color~$\alpha$.
  Due to the coloring of~$G_k$ and since~$\perm$ is an automorphism
  of~$G_k$, there exists~$j \in [\pm n]$ such that~$\perm(v^k_i) = v^k_j$.
  Because of the same reasons, $\perm(v)$ corresponds to an
  operator of the same type as~$v$ and~$\{\perm(v), v^k_j\}$ has
  color~$\alpha$.
  Thus, operator~$v$ in~$T_k$ has input~$\alpha \cdot x_i$ and
  operator~$\perm(v)$ in~$\perm(T_k)$ has input
  \[
    \alpha \cdot x_{\sign(i)\abs{i}}
    =
    \sign(i)\alpha \cdot x_{\abs{i}}
    =
    \alpha \cdot \sperm(x)_j,
  \]
  where the first equality holds due to assigning edges~$\{v,v^k_j\}$
  and~$\{v,v^k_{-j}\}$ negated values/colors.
  Consequently, since we assumed that all operators are commutative,
  assigning the pre-function of~$\bar{T}_k$ input~$x$ and the pre-function
  of~$\perm(\bar{T}_k)$ input~$\sperm(x)$ yields the same function.
  \end{proof}
\begin{remark}
  Proposition~\ref{prop:basicSDG} can easily be generalized to expression
  trees involving non-commu\-ta\-tive operators.
  As mentioned in Section~\ref{sec:existingdetection}, the arcs leaving an
  operator node of such an expression tree need to be labeled to indicate
  the order in which the input of an operator needs to be processed.
  These labels then need to be incorporated into the edge colors of an
  \SDG.
\end{remark}
We close this section by evaluating the the capabilities of the \SDG of
Proposition~\ref{prop:basicSDG} for
example~\eqref{eq:exampleSDG}.
The reflection symmetries of~\eqref{eq:exampleSDG} are the
signed permutations~$\sperm_1 = (1,-2)(2,-1)$, $\sperm_2
= (3,-4)(4,-3)$, $\sperm_3 = \sperm_2 \circ \sperm_1$, and~$\sperm_4 =
\id$.
The only signed permutations that can be derived from automorphisms of the
\SDG in Figure~\ref{fig:SDG}, however, are~$\sperm_1$ and~$\sperm_4$.
That is, even for linear expressions, the \SDG might not allow to detect all
reflection symmetries.
This undesired behavior will be investigated in more detail in the next
section, where we will modify the \SDG to detect all reflection symmetries
of linear expressions.
Moreover, the modifications will also allow to detect the natural
reflection symmetries of the disk packing problem from
Example~\ref{ex:geompack}.

\subsection{Enhancements of the Framework}
\label{sec:enhanceddetection}

Although \SDGs are defined in terms of relative variable domains, the \SDG
of Proposition~\ref{prop:basicSDG} for the exemplary
MINLP~\eqref{eq:exampleSDG} does not encode the
reflection symmetry between~$x_3$ and~$x_{-4}$.
The reason is that our proposed modifications of expression trees
explicitly encode the original reflection centers in the expression trees.
Variables with different reflection centers consequently
cannot be symmetric in the \SDG, cf.\ Figure~\ref{fig:SDG}.
In this section, we discuss small modifications of the proposed \SDGs,
which allow to detect reflection symmetries of sum expressions,
squared differences, bilinear products, and even operators.
This list of examples is, of course, not exhaustive and only serves as an
illustration of how the abstract concept of symmetry detection graphs can
be used to derive tailored \SDGs incorporating symmetry information of
particular types of constraints.

\paragraph{Sum Expressions.}
A simple way to resolve the issue for
MINLP~\eqref{eq:exampleSDG} is to not consider the summands of a sum
expression independently.
To make this precise, let~$I \subseteq [n]$ and~$S = \sum_{i \in
  I} \alpha_i \cdot x_i$, where~$\alpha_i \in \R$ for all~$i \in I$.
To center each variable at the origin, the last section proposed to modify
an expression tree by replacing each summand~$\alpha_i \cdot x_i$ by a sub
expression tree for~$\alpha_i\cdot x_i + \alpha_i \cdot \dcenter_i$.
Alternatively, we can compute the expression~$S' = \sum_{i \in I} \alpha_i \cdot
\dcenter_i + \sum_{i \in I} \alpha_i \cdot x_i$ first and create the
expression tree for~$S'$, see Figure~\ref{fig:enhancedSDG}.
Since both the modified expression tree~$T$ from the previous section and the
tree~$T'$ described here model the same function, automorphisms of the \SDG
corresponding to~$T'$ also correspond to reflection symmetries of~\eqref{eq:MINLP}.
In particular, because the \SDG for~$S'$ essentially corresponds to the
\SDGs from~\cite{Bodi2013} for linear constraints, all reflection
symmetries of~$S$ are encoded in the \SDG, see~\cite[Thm.~4]{Bodi2013}.

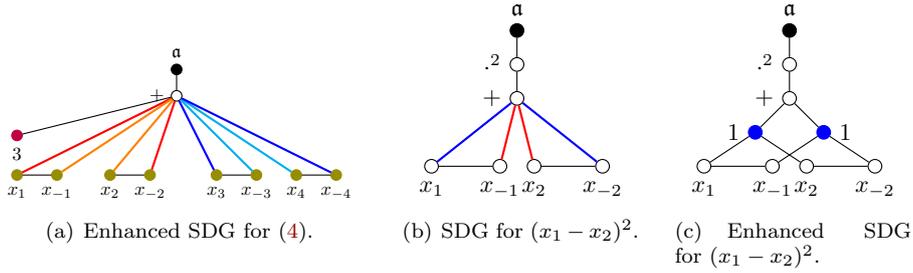
\begin{figure}[t]
  \centering
  
  \subfigure[Enhanced \SDG for~\eqref{eq:exampleSDG}.\label{fig:enhancedSDG}]
  {
    \begin{tikzpicture}[scale=0.7, every node/.style={scale=0.7}]
      \tikzstyle{v} += [circle,inner sep=2pt,minimum size=2mm,draw=black];
      \node (a) at (0,0.5) [v,label=above:$\anchor$,fill=black] {};
      \node (r) at (0,0) [v,label=left:$+$] {};

      \node (p1a) at (-3,-1.5) [v,label=below:$x_1$,fill=olive,draw=olive] {};
      \node (p1b) at (-2.25,-1.5) [v,label=below:$x_{-1}$,fill=olive,draw=olive] {};
      \node (p2a) at (-1.25,-1.5) [v,label=below:$x_2$,fill=olive,draw=olive] {};
      \node (p2b) at (-0.5,-1.5) [v,label=below:$x_{-2}$,fill=olive,draw=olive] {};
      \node (p3a) at (0.75,-1.5) [v,label=below:$x_3$,fill=olive,draw=olive] {};
      \node (p3b) at (1.5,-1.5) [v,label=below:$x_{-3}$,fill=olive,draw=olive] {};
      \node (p4a) at (2.25,-1.5) [v,label=below:$x_4$,fill=olive,draw=olive] {};
      \node (p4b) at (3.0,-1.5) [v,label=below:$x_{-4}$,fill=olive,draw=olive] {};
      \node (p5) at (-3,-0.75) [v,label=below:$3$,fill=purple,draw=purple] {};

      \draw[-] (r) -- (a);
      \draw[-,thick,draw=red] (r) -- (p1a);
      \draw[-,thick,draw=orange] (r) -- (p1b);
      \draw[-,thick,draw=red] (r) -- (p2b);
      \draw[-,thick,draw=orange] (r) -- (p2a);
      \draw[-,thick,draw=blue] (r) -- (p3a);
      \draw[-,thick,draw=cyan] (r) -- (p3b);
      \draw[-,thick,draw=cyan] (r) -- (p4a);
      \draw[-,thick,draw=blue] (r) -- (p4b);
      \draw[-] (r) -- (p5);

      \draw[-] (p1a) -- (p1b);
      \draw[-] (p2a) -- (p2b);
      \draw[-] (p3a) -- (p3b);
      \draw[-] (p4a) -- (p4b);
    \end{tikzpicture}
  } 
  \quad
  \subfigure[\SDG for~$(x_1 - x_2)^2$.\label{fig:sqdiff}]
  {
    \begin{tikzpicture}[scale=0.9, every node/.style={scale=0.9}]
      \tikzstyle{v} += [circle,inner sep=2pt,minimum size=2mm,draw=black];
      \node (a) at (0,0.5) [v,label=above:$\anchor$,fill=black] {};
      \node (r) at (0,0) [v,label=left:$\cdot^2$] {};

      \node (p) at (0,-0.5) [v,label=left:$+$] {};

      \draw[-] (a) -- (r);
      \draw[-] (r) -- (p);

      \node (x1) at (-1.25,-1.5) [v,label=below:$x_1$] {};
      \node (x1a) at (-0.25,-1.5) [v,label=below:$x_{-1}$] {};
      \node (x2) at (0.25,-1.5) [v,label=below:$x_2$] {};
      \node (x2a) at (1.25,-1.5) [v,label=below:$x_{-2}$] {};

      \draw[-,thick,draw=blue] (p) -- (x1);
      \draw[-,thick,draw=red] (p) -- (x1a);
      \draw[-,thick,draw=red] (p) -- (x2);
      \draw[-,thick,draw=blue] (p) -- (x2a);
      \draw[-] (x1) -- (x1a);
      \draw[-] (x2) -- (x2a);
    \end{tikzpicture}
  }
  \quad
  \subfigure[Enhanced \SDG for~$(x_1 - x_2)^2$.\label{fig:enhancesqdiff}]
  {
    \begin{tikzpicture}[scale=0.9, every node/.style={scale=0.9}]
      \tikzstyle{v} += [circle,inner sep=2pt,minimum size=2mm,draw=black];
      \node (a) at (0,0.5) [v,label=above:$\anchor$,fill=black] {};
      \node (r) at (0,0) [v,label=left:$\cdot^2$] {};

      \node (p) at (0,-0.5) [v,label=left:$+$] {};

      \draw[-] (a) -- (r);
      \draw[-] (r) -- (p);

      \node (aux1) at (-0.5,-1) [v,label=left:1,fill=blue,draw=blue] {};
      \node (aux2) at (0.5,-1) [v,label=right:1,fill=blue,draw=blue] {};

      \draw[-] (p) -- (aux1);
      \draw[-] (p) -- (aux2);

      \node (x1) at (-1.25,-1.5) [v,label=below:$x_1$] {};
      \node (x1a) at (-0.25,-1.5) [v,label=below:$x_{-1}$] {};
      \node (x2) at (0.25,-1.5) [v,label=below:$x_2$] {};
      \node (x2a) at (1.25,-1.5) [v,label=below:$x_{-2}$] {};

      \draw[-] (x1) -- (x1a);
      \draw[-] (x2) -- (x2a);
      \draw[-] (aux1) -- (x1);
      \draw[-] (aux1) -- (x2);
      \draw[-] (aux2) -- (x1a);
      \draw[-] (aux2) -- (x2a);
    \end{tikzpicture}
  }
  \caption{Illustration of a enhanced \SDGs.}
\end{figure}

\paragraph{Squared Differences.}
Consider Example~\ref{ex:geompack} for two disks.
Incorporating the previously discussed idea for sum expressions into
\SDGs allows to detect the symmetries~$(x_1,x_2) \mapsto
(-x_2,-x_1)$ and~$(y_1,y_2) \mapsto (-y_2,-y_1)$ even when the box is not
centered at the origin.
The reflection along the horizontal and vertical symmetry axes of the box,
however, cannot be detected via the \SDG.
The reason is that the \SDG does not encode that~$(x_1 - x_2)^2$ is the
same as~$(-x_1 - (-x_2))^2$ since the corresponding expression trees are
different, see Figure~\ref{fig:sqdiff} for an \SDG of this expression.

To incorporate this information into an \SDG, we replace the
subgraph in the \SDG corresponding to the sub expression tree for~$(x_1 -
x_2)^2$ by a gadget.
For this gadget, we introduce a binary \emph{squared difference}
operator~$\sqdiff$ modeling~$\sqdiff(x_1,x_2) = (x_1 - x_2)^2$.
The gadget then consists of an operator node~$v$ of type~$\sqdiff$, two
auxiliary value nodes~$a_1$ and~$a_2$ both receiving color~1, and the
distinguished nodes~$v_1, v_{-1}, v_2, v_{-2}$ for~$x_1,x_{-1},x_2,
x_{-2}$.
The gadget's edges all remain uncolored and are given by~$\{v,a_1\}$,
$\{v,a_2\}$, $\{a_1,v_1\}$, $\{a_1,v_2\}$, $\{a_2,v_{-1}\}$,
and~$\{a_2,v_{-2}\}$, see Figure~\ref{fig:enhancesqdiff}.

Since every \SDG adds the edges~$\{v_1,v_{-1}\}$ and~$\{v_2,v_{-2}\}$, note
that any automorphism of the gadget that exchanges~$v_1$ and~$v_2$ also
needs to exchange~$v_{-1}$ and~$v_{-2}$.
Hence, the automorphism group of this gadget is generated (provided all
variables have the same type) by the permutations~$\perm_1 =
(v_1,v_2)(v_{-1},v_{-2})$ and~$\perm_2 =
(a_1,a_2)(v_1,v_{-1})(v_2,v_{-2})$.
The former corresponds to the permutation symmetry of
Example~\ref{ex:geompack} that exchanges identical disks; the latter
corresponds to the reflection along the vertical symmetry axis of the box
(as we only consider~$x$-variables).
That is, all reflection symmetries of Example~\ref{ex:geompack} can be
detected via this gadget.

Note that the introduction of the new operator type indeed ensures that
every automorphism of the \SDG corresponds to a reflection symmetry of the
corresponding MINLP, since the squared difference operator serves as an
anchor of the gadget.
In the following, we will generalize this idea to other structures
frequently arising in MINLP.
We only need to ensure that each newly introduced gadget is anchored
and that there is a unique operator type identifying the gadget.
The former ensures that the entire \SDG remains anchored if a gadget
replaces a sub expression tree, whereas the latter guarantees constraint
compatibility among different \SDGs.

\paragraph{Bilinear Products.}
Let~$i,j \in [n]$ be distinct and consider the bilinear product~$x_i \cdot
x_j$.
If the reflection center of both~$x_i$ and~$x_j$ is at the origin and both
variables have the same type, the product admits two symmetries:
either one exchanges~$x_i$ and~$x_j$, or one simultaneously replaces~$x_i$
by~$x_{-i}$ and~$x_j$ by~$x_{-j}$.
These symmetries are exactly the same symmetries as for the case of squared
differences.
For this reason, the symmetries can be encoded using an analogous gadget;
the only difference is that the $\sqdiff$ operator node needs to be
replaced by the product operator.

\paragraph{Even Functions.}
Let~$i \in [n]$ and let~$f\colon \R \to \R$ be a univariate function such
that~$f(x_i)$ arises as a subexpression in~\eqref{eq:MINLP}.
If~$f$ is an even function, then~$x_i \mapsto -x_i$ is a potential symmetry
of~\eqref{eq:MINLP} as~$f(x_i) = f(-x_i)$.
This property of even functions can easily be incorporated into \SDGs.
Instead of connecting the operator node corresponding to~$f$ with the
distinguished nodes~$v_i$ and~$v_{-i}$ by two edges being colored by the
numerical values~$1$ and~$-1$, respectively, we assign both edges color~1.

The idea for even functions can be generalized, of course, if the input
of~$f$ is a more complicated expression~$E$.
In this case, nodes~$v_i$ and~$v_{-i}$ must be replaced by the root nodes
of graphs representing expression trees for~$E$ and~$-E$, respectively.
Depending on the size of~$E$, this could increase the size of the \SDGs
significantly though, and thus also increases the time needed to detect
symmetries.

\section{An Open-Source Implementation of the Abstract Framework}
\label{sec:implementation}

In this section, we describe our implementation of the abstract symmetry
detection framework of the previous section within the open-source solver
\scip.
Our framework is contained as a \code{C}-implementation in the release of
\scip~9.0~\cite{bolusani2024scip} and replaces \scip's previous symmetry
detection mechanism.
We start by explaining the design principles of our implementation
(Section~\ref{sec:design}), followed by an illustration of how to use our
framework within \scip (Section~\ref{sec:illustrateFramework}).
The section is concluded in Section~\ref{sec:technicalDetails} by providing
further technical details and a discussion of how to detect symmetries of
\SDGs.
Appendix~\ref{app:functions} provides an overview of the most
important functions needed to apply our framework.
\subsection{Design Principles}
\label{sec:design}
One of the design principles of \scip is that all major
components of the solver are organized as plug-ins.
This allows users to easily extend \scip by tailored techniques for
specific applications, e.g., cutting planes or heuristics.
The main plug-in type is a so-called constraint handler.
A constraint handler provides an abstract notion of a class of constraints
(e.g., general linear constraints, knapsack constraints, SOS1 constraints,
or general nonlinear constraints), and defines general rules to enforce
that a solution adheres to the corresponding constraints.
That is, if \scip solves an instance of~\eqref{eq:MINLP} and finds some
intermediate solution~$x$, every constraint~$g_k$, $k \in [m]$, will ask
its corresponding constraint handler for rules to check whether the
solution satisfies~$g_k(x) \leq 0$ and, if not, how this can be resolved.

With the release of \scip~5.0, a symmetry detection mechanism had been
implemented.
This mechanism, however, was only able to detect symmetries if all
constraints of a problem have a constraint handler being part of the \scip
release.
One of our motivations to introduce the abstract notion of \SDGs
was to overcome this limitation and to allow symmetry detection also in the
presence of custom constraints.
We realized symmetry detection via \SDGs in \scip by introducing a new
optional callback for constraint handlers.
If a constraint handler implements this callback, the callback needs to create
an \SDG from the constraint's data adhering to
Proposition~\ref{prop:basicSDG}.
The \SDGs for individual constraints are then combined to a global \SDG for
the entire problem via Theorem~\ref{thm:buildSDG}.
Symmetry detection in \scip will then check whether the constraint handlers
of all constraints present in the problem implement the new callback.
If this is the case, symmetries are computed based on the global \SDG;
otherwise, symmetry computation is disabled.
In the remainder of this section, we provide details of how \SDGs can be
encoded using our implementation, and how variable color and constraint
compatibility can be ensured.

\paragraph{Nodes and Edges of \SDGs.}
In the construction of \SDGs for nonlinear constraints represented by
expression trees, we used four different types of nodes: nodes
representing numerical values, mathematical operators, variables, and an
anchor.
Since we believe that the representation via expression trees is rather
generic, our implementation also makes use of four different types of
nodes that can hold different information:
\begin{description}
\item[\emph{value nodes}] store a floating-point number;
\item[\emph{operator nodes}] store an integer value serving as an
  identifier of an operator;
\item[\emph{variable nodes}] store the index of the corresponding
  (reflected) variable;
\item[\emph{constraint nodes}] store a pointer to a constraint as well as
  two floating-point numbers (referred to as left-hand side and right-hand
  side).
\end{description}
Due to the definition of \SDGs, our implementation automatically adds
variable nodes for all (reflected) variables to an \SDG.
The remaining types of nodes can be added to an \SDG by dedicated
functions, see Appendix~\ref{app:functions}.
Each node of an \SDG is identified by an integer index.

Edges~$\{u,v\}$ of \SDGs are identified by the indices~$u$ and~$v$,
and can optionally hold a floating-point value, cf.\ the assignment of
numerical values to edges in Section~\ref{sec:basicdetection}.
By calling a function, edges can be added to an \SDG.
Our implementation also makes sure that every \SDG contains the edges
connecting the distinguished nodes for a variable~$x_i$ and its
reflection~$x_{-i}$.

Note that neither for nodes nor for edges we allow to specify a color.
The colors will be determined separately.

\paragraph{Ensuring Compatibility.}
To ensure constructing correct \SDGs for optimization problems,
all \SDGs for individual constraints need to be anchored, and the
collection of \SDGs has to be variable color and constraint
compatible.
We therefore devised the following principles for constructing an \SDG that
should be used by symmetry detection callbacks of constraint
handlers.
\begin{enumerate}[left=0.1cm,label={(P\arabic*)},ref={P\arabic*}]
\item\label{P1} The callback must not add edges connecting two variable nodes.

\item\label{P2} Every \SDG for an individual constraint should have exactly one
  constraint node that serves as an anchor.

\item\label{P3} The \SDGs constructed by a constraint handler for two
  constraints~$c,c'$ are only isomorphic if~$c \equiv c'$, cf.\
  Section~\ref{sec:detection}.
\end{enumerate}
Regarding~\eqref{P1}, the only edges between variable nodes that
are required in   Section~\ref{sec:detection} are the edges connecting
variable nodes for pairs of reflected variables, which are present in an
\SDG by default in our implementation.
Regarding~\eqref{P2}, although operator and value nodes could serve as
anchors, too, constraint nodes easily allow to ensure constraint
compatibility, see below.

A key component of ensuring compatibility is that only equivalent types
of nodes and edges receive identical colors.
We therefore did not allow to specify the color of a node or edge during
creation, but instead assign properties to nodes and edges.
Once all \SDGs have been created, \scip computes pairwise disjoint sets of
colors for the edges and the four different types of nodes.
While the colors of edges as well as operator and value nodes can be easily
derived from the unique value assigned to the corresponding object,
deriving the colors of variable and constraint nodes is not immediate.
For constraint nodes, we derive a color based on the corresponding
left-hand side and right-hand side values as well as the constraint handler
of the associated constraint.
Variable nodes are assigned a color based on their types as explained in
Section~\ref{sec:detection}.
The latter guarantees variable color compatibility.

It remains to discuss constraint compatibility.
Let~$G$ be the constructed \SDG for an entire optimization problem, and
let~$G'$ and~$G''$ be the \SDGs for two constraints of the problem.
In the proof of Theorem~\ref{thm:buildSDG}, we noted that anchors
guarantee that, if an automorphism of~$G$ maps a node of~$G'$ onto a
node of~$G''$, then every node of~$G'$ needs to be mapped onto a node
of~$G''$.
Since the color of an anchor node depends on the corresponding constraint
handler by~\eqref{P2}, only \SDGs derived from the same constraint handler
can be isomorphic.
Principle~\eqref{P3} then guarantees constraint compatibility.
\begin{remark}
  Although~\eqref{P3} seems difficult to achieve at first glance, it is
  usually easy to implement in practice.
  The reason is that constraint handlers need to implement abstract rules
  for deriving an \SDG from a constraint.
  One such mechanism has been described for \SDGs derived from expression
  trees in Section~\ref{sec:basicdetection}, and
  Section~\ref{sec:illustrateFramework} will discuss another
  approach for particular constraints without an immediate expression tree
  representation.
\end{remark}

\subsection{Callbacks and Their Usage}
\label{sec:illustrateFramework}

In this section, we provide a more detailed description of our callbacks
and we illustrate how to use them to detect symmetries.

\paragraph{Callbacks}
In \scip~9.0, one can compute either
reflection symmetries or classical permutation symmetries.
To allow constraints to inform \scip about how to detect (reflection)
symmetries, their constraint handlers must implement the
\code{SCIP\_DECL\_CONSGETSIGNEDPERMSYMGRAPH} callback for reflection
symmetries or the \code{SCIP\_DECL\_CONSGETPERMSYMGRAPH} callback for
permutation symmetries.
Since the functionality of both callbacks is essentially the same, we refer
to both of them just as ``the callback'' in the following.
The only differences between the \SDGs for permutation symmetries and
reflection symmetries is that an \SDG for permutation symmetries does not
contain nodes for reflected variables and that the variable type is defined
according to the original variable bounds (i.e., not shifted).

The input of either callback are five pointers providing the following information:
\begin{description}
\item[\code{scip}] data structure with information about
  the problem and solving process;
\item[\code{conshdlr}] constraint handler for which the callback
  has been implemented;
\item[\code{cons}] constraint for which an \SDG shall be
  created;
\item[\code{graph}] \SDG to which the information of \code{cons} shall be added;
\item[\code{success}] Boolean to store whether the \SDG could be created successfully.
\end{description}
If the callback cannot create the \SDG, it should set \code{success} to
\code{FALSE} to inform \scip that not all constraints could provide
symmetry information.
Symmetry detection gets then disabled.

Constraints in \scip usually provide further information that fully
characterizes the constraint, so-called consdata.
Linear constraints, for example, store the variables and coefficients as
well as the left-hand side and right-hand side of the constraint in the
consdata.
The consdata can be accessed from \code{cons} and usually is the only
information that should be used to create the \SDG.

The \SDG to which nodes and edges shall be added is given by \code{graph}.
Note that we decided to not create a separate \SDG for every constraint.
Instead, \code{graph} corresponds to the \SDG for the entire problem.
This does not cause any conflicts as long as the callback only adds edges
between nodes that have been created by the callback for the same
constraint.
Since no new variable nodes can be created, there cannot be
multiple nodes modeling the same variable.

\paragraph{Using the Callback}

We illustrate how the callbacks can be used for the stable set problem.
Let~$H = (V,E)$ be an undirected graph with node weights~$w_v \in
\R$, $v \in V$.
A \emph{stable set} in~$H$ is a set~$S \subseteq V$ if~$\{u,v\} \notin E$
for all~$u,v \in S$.
The task is to find a stable set~$S$ with maximum total node weight.
A classical integer programming formulation is
\[
  \max\left\{ \sum_{v \in V} w_vx_v : x_u + x_v \leq 1 \text{ for all }
    \{u,v\} \in E \text{ and } x \in \B{V}\right\}.
\]
It is well-known that the linear programming (LP) relaxation of this
formulation is rather weak and that it can be enhanced by adding so-called
clique inequalities~$\sum_{v \in C} x_v \leq 1$, where~$C$ is a clique
in~$H$.
Among others, the following two possibilities exist to make use of clique
inequalities in \scip.
On the one hand, one could enumerate all clique inequalities explicitly and
add them as linear inequalities to the problem formulation.
In this case, \scip can automatically detect symmetries.
But since there might be exponentially many cliques, the LP relaxation
could become much harder to solve.
On the other hand, one could implement an abstract constraint handler that
gets~$H$ as input and decides on the fly whether a solution satisfies all
clique inequalities, and adds an inequality in case it is violated by the
solution.
Without our callback, however, \scip cannot know how the presence of the
stable set constraint handler impacts the symmetries of the problem.

To implement our callback, we observe that every automorphism of~$H$ that
exchanges nodes of the same weight is also a symmetry of the stable set
problem.
The \SDG for the stable set constraint handler thus essentially needs to
add a copy of~$H$ to the global \SDG.
We illustrate this in Figure~\ref{fig:illustrateCallback} and provide a
description next.

\begin{figure}[t!]
  \centering
  \includegraphics[scale=0.8]{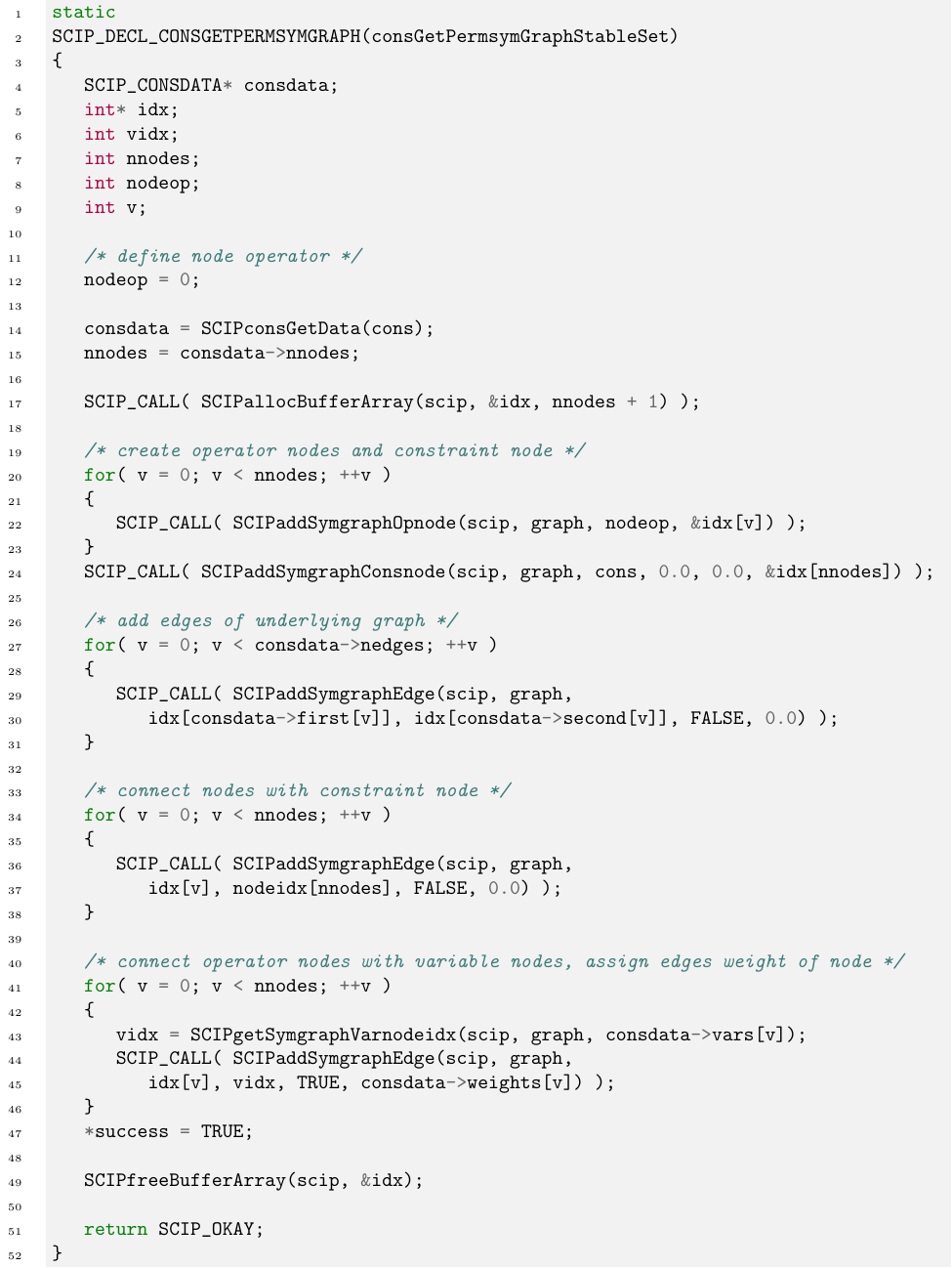}  
  \caption{Illustration of usage of symmetry detection callback.}
  \label{fig:illustrateCallback}
\end{figure}

We assume that the consdata of the stable set constraint contains the
following information:
\code{nnodes} provides the number of nodes in~$H$ and the nodes are
labeled~$0,\dots,\code{nnodes}-1$; \code{nedges} provides the number of
edges in~$H$ and edges are encoded via two arrays \code{first} and
\code{second} that contain the first and second nodes of all edges,
respectively;
\code{weights} is an array that assigns each node its weight.

First, we encode an operator type in Line~12 of
Figure~\ref{fig:illustrateCallback}, which represent nodes of the
graph~$H$.
Note that there might exist other constraint handlers that also define an
operator with the same index.
This is not an issue though as~\eqref{P2} will make sure that we cannot mix
constraints of the stable set constraint handler with constraints from
other constraint handlers.

Second, we extract some information about~$H$ and create an array
\code{idx} to store the indices of the newly created nodes of the \SDG
(Line~17).
Afterwards, we create for each node of~$H$ a copy in the \SDG by creating
a corresponding operator node (Line~22).
We also create an anchor of the \SDG (Line~24) by introducing a constraint node.
Since the abstract stable set constraint has neither a left-hand side nor a
right-hand side, we store the dummy values~0 and only assign the constraint
pointer \code{cons} to the anchor node.

Third, the edges of~$H$ are copied to the \SDG in Line~29.
Since edges of~$H$ are unweighted, also the corresponding edges of the \SDG
are unweighted (indicated by \code{FALSE}).
Fourth, we connect the anchor with the remaining nodes of the \SDG.
To preserve symmetries of~$H$, the anchor is connected with all copied nodes
of~$H$ (Line~36).
Finally, we add edges between the copies of nodes of~$H$ and the
corresponding variable nodes (Line~44).
These edges receive the weights of the corresponding nodes in~$H$, and we
set the \code{success} pointer to \code{TRUE} since the graph could
be created.

\subsection{Technical Details}
\label{sec:technicalDetails}

In this section, we provide further technical details of our
implementation.
First, we briefly discuss our data structures for \SDGs.
Second, we mention further details of the implementation to realize
the ideas of Section~\ref{sec:detection} in \scip.
Finally, we explain how automorphism of \SDGs can be detected.

\paragraph{Encoding Symmetry Detection Graphs}
We experimented with different data structures for encoding \SDGs.
Our first attempt was to use an object-oriented implementation containing
separate objects (or structs in \code{C}) for graphs, nodes, and edges.
But it turned out that, even for linear problems, the creation of \SDGs
was much slower than the previous implementation in \scip that was tailored
towards constraints known by \scip.
The main reason for the slower performance was an overhead in memory
allocation for the separate objects.
For the same reason, we discarded the idea of creating separate \SDGs for
individual constraints.
Instead, we only maintain a single \SDG that is extended by the callbacks,
which turned out to be more efficient.
The final data structure for \SDGs is a \code{C} struct using basic
\code{C} data types such as \code{int}, \code{double}, or pointers to
encode information about nodes and edges.

In the following, assume an \SDG has~$n$ nodes and~$m$ edges.
Nodes of \SDGs are identified by an index in~$N = \{0,\dots,n-1\}$ and several
arrays store information about the nodes.
Recall that we distinguish four different types of nodes (operator,
numerical value, variable, and constraint).
These types are encoded by an \code{enum} and several arrays hold the
information corresponding to the nodes of different types.
For example, array \code{vals} holds the values assigned to value nodes, and
array \code{lhs} stores the left-hand side values assigned to the
constraint nodes.
To access the information associated with node~$i \in N$, we use two
arrays.
Array \code{nodetypes} stores the type of the different nodes and array
\code{nodeinfopos} stores the index of each node in its corresponding
group.
That is, if~$j$ is the value stored at \code{nodeinfopos[i]}, then node~$i
\in N$ is the~$j$-th node within the group \code{nodetypes[i]}.
If node~$i$ was a constraint node, we thus could access its associated
left-hand side at \code{lhs[j]}.
Several functions are available to access node information or to create
nodes, see Appendix~\ref{app:functions}.
When creating a new node, we check whether the node still fits into the
allocated memory for the aforementioned arrays.
If not, we enlarge the arrays by memory reallocation; the new size of the
arrays is determined by \scip based on a growth factor to avoid frequent
memory reallocation.

Edges are identified by an index in~$\{0,\dots,m-1\}$.
Two arrays \code{edgefirst} and \code{edgesecond} store the indices of the
first and second nodes in the edges, respectively; array \code{edgevals}
stores the values associated with an edge.
If an edge has no assigned value, we use the placeholder value infinity in
\code{edgevals}.

Moreover, an \SDG struct contains \code{int} arrays for the colors
associated with nodes and edges.
Initially, these arrays are empty and only created if computing colors is
triggered by calling a function, see the next section for details.
In this case, an \SDG will lock itself, which means that is is no longer
possible to add nodes and edges.
This is a safety mechanism to ensure that the stored colors always
correspond to the \SDG.

\paragraph{Implementation Details}
We discuss two implementation details concerning numerical
inaccuracies and aggregation of variables.
The former is relevant for computing colors, whereas the second has
consequences for creating \SDGs.

To derive colors for nodes, we sort nodes first by their type (operator,
constraint etc.) and then based on the information associated with them.
For instance, numerical value nodes are sorted based on their associated
value, and variable nodes are sorted based on their variable's type.
Since \scip uses floating-point arithmetic when solving an optimization
problem, it is reasonable to also consider two nodes as identical if their
associated information does not deviate more than a small
quantity~$\varepsilon$ (by default~$10^{-9}$ in \scip).
We then partition the sorted list of nodes into blocks of consecutive
elements such that the first and last node per block do not deviate more
than~$\varepsilon$.
Every node within the same block is then assigned the same color, and nodes
from different blocks receive different colors.
Colors for edges are computed analogously.

In Definition~\ref{def:symdetectgraph}, we required \SDGs to contain a
distinguished node for every (reflected) variable of~\eqref{eq:MINLP}.
In practice, however, \scip removes variables from a problem, e.g., when
they get fixed or can be represented as a weighted sum of other
variables (variable aggregation).
To reduce the number of variable nodes in \SDGs, it thus makes sense to
only represent variables in an \SDG that are neither fixed nor aggregated.
Our implementation of \SDGs therefore also only supports unfixed and
non-aggregated variables.
When using the symmetry detection callbacks, every occurrence of a
fixed or aggregated variable needs to be replaced by the corresponding
constant or sum of variables, respectively.

\paragraph{Detecting Automorphisms}
To detect symmetries of \SDGs, we use external software packages for
computing graph automorphisms.
Our current implementation supports the packages
\solver{bliss}~\cite{bliss} and \solver{nauty}~\cite{nauty}.
Since \solver{bliss} and \solver{nauty} only support graphs with uncolored
edges, we post-process \SDGs:
every edge~$\{u,v\}$ of color~$c$ is replaced by an auxiliary node~$w$
of color~$c$ as well as the uncolored edges~$\{u,w\}$ and~$\{v,w\}$.
As noted in Section~\ref{sec:existingdetection}, the number of auxiliary
nodes can be reduced by identifying some of them with each other.
For the \SDGs~$G$ for permutation symmetries of MILP, see
Section~\ref{sec:existingdetection}, this can be achieved as
follows, cf.~\cite{Puget2005}.

Let~$v$ be a variable node of~$G$ and let~$E^v_c$ be the set of all edges of
color~$c$ that are incident with~$v$.
Instead of introducing an auxiliary node for all edges in~$E^v_c$, it is
sufficient to introduce one auxiliary node~$w$ as well as the
edges~$\{v,w\}$ and~$\{u,w\}$ for all~$\{u,v\} \in E^v_c$.
This way, the number of auxiliary nodes changes from~$\abs{E^v_c}$
to~1 and the number of edges changes from~$2 \cdot \abs{E^v_c}$ to~$1 +
\abs{E^v_c}$.
This mechanism of reducing the number of auxiliary nodes is called
\emph{grouping by constraints}~\cite{PfetschRehn2019}; if the mechanism is
applied to constraint nodes instead of variable nodes, it is called
\emph{grouping by variables}.

We have implemented an analogous grouping mechanism for our abstract \SDGs,
where we group edges being incident to either the same constraint or
variable node.
The former is used if there are less constraint nodes than variable nodes.

\section{Handling Reflection Symmetries}
\label{sec:handling}

In this section, we describe methods for handling (reflection) symmetries
in MINLP, which we illustrate using the disk packing problem of
Example~\ref{ex:geompack}.
In the disk packing problem, we distinguish three types of symmetries:
\begin{enumerate*}[label=(\roman*)]
\item permutation symmetries that exchange indices of disks;
\item reflection symmetries that reflect the~$x$- or~$y$-coordinates of all
  disks;
\item and, if the width and height of the box are identical, permutation
  symmetries that exchange both coordinates.
\end{enumerate*}
To visualize these symmetries, it is convenient to encode a solution of the
disk packing problem by a matrix~$M \in \R^{D \times 2}$, where the~$i$-th
row~$(M_{i1}, M_{i2})$ corresponds to the center point~$(x_i, y_i)$ of
the~$i$-th disk.
The first class of symmetries then corresponds to reordering the rows
of~$M$, the second class to changing the signs of all entries of a column
of~$M$, and the last class to an exchange of the two columns of~$M$, see
Figure~\ref{fig:matrixsym}.

\begin{figure}[t]
  \centering
  \subfigure[Original matrix.]
  {
    \begin{tikzpicture}[scale=0.5]
      \draw (0,0) grid (3,4);

      \node(a) at (-1,0) {};
      \node (1)  at (0.5,0.5) {1};
      \node (2)  at (1.5,0.5) {2};
      \node (3)  at (2.5,0.5) {3};
      \node (4)  at (0.5,1.5) {4};
      \node (5)  at (1.5,1.5) {5};
      \node (6)  at (2.5,1.5) {6};
      \node (7)  at (0.5,2.5) {7};
      \node (8)  at (1.5,2.5) {8};
      \node (9)  at (2.5,2.5) {9};
      \node (10) at (0.5,3.5) {10};
      \node (11) at (1.5,3.5) {11};
      \node (12) at (2.5,3.5) {12};
      \node(z) at (4,0) {};
    \end{tikzpicture}
  }
  \subfigure[Disk permutation.]
  {
    \begin{tikzpicture}[scale=0.5]
      \draw (0,0) grid (3,4);

      \node(a) at (-1,0) {};
      \node (1)  at (0.5,0.5) {4};
      \node (2)  at (1.5,0.5) {5};
      \node (3)  at (2.5,0.5) {6};
      \node (4)  at (0.5,1.5) {10};
      \node (5)  at (1.5,1.5) {11};
      \node (6)  at (2.5,1.5) {12};
      \node (7)  at (0.5,2.5) {1};
      \node (8)  at (1.5,2.5) {2};
      \node (9)  at (2.5,2.5) {3};
      \node (10) at (0.5,3.5) {7};
      \node (11) at (1.5,3.5) {8};
      \node (12) at (2.5,3.5) {9};
      \node(z) at (4,0) {};

      \draw[->] (0,3.25) to [bend right=90] (0,1.75);
      \draw[->] (0,1.25) to [bend right=60] (0,0.75);
      \draw[->] (0,0.25) to [bend left=90] (0,2.75);
      \draw[->] (0,2.25) to [bend left=90] (0,3.75);
    \end{tikzpicture}
  }
  \subfigure[Reflection.]
  {
    \begin{tikzpicture}[scale=0.5]
      \draw (0,0) grid (3,4);

      \node(a) at (-0.5,0) {};
      \node (1)  at (0.5,0.5) {1};
      \node (2)  at (1.5,0.5) {-2};
      \node (3)  at (2.5,0.5) {3};
      \node (4)  at (0.5,1.5) {4};
      \node (5)  at (1.5,1.5) {-5};
      \node (6)  at (2.5,1.5) {6};
      \node (7)  at (0.5,2.5) {7};
      \node (8)  at (1.5,2.5) {-8};
      \node (9)  at (2.5,2.5) {9};
      \node (10) at (0.5,3.5) {10};
      \node (11) at (1.5,3.5) {-11};
      \node (12) at (2.5,3.5) {12};
      \node(z) at (3.5,0) {};

      \draw[red,ultra thick] (1,0) -- (2,0) -- (2,4) -- (1,4) -- (1,0);
    \end{tikzpicture}
  }
  \subfigure[Coordinate symmetry.]
  {
    \begin{tikzpicture}[scale=0.5]
      \draw (0,0) grid (3,4);

      \node(a) at (-1.5,0) {};
      \node (1)  at (0.5,0.5) {1};
      \node (2)  at (1.5,0.5) {3};
      \node (3)  at (2.5,0.5) {2};
      \node (4)  at (0.5,1.5) {4};
      \node (5)  at (1.5,1.5) {6};
      \node (6)  at (2.5,1.5) {5};
      \node (7)  at (0.5,2.5) {7};
      \node (8)  at (1.5,2.5) {9};
      \node (9)  at (2.5,2.5) {8};
      \node (10) at (0.5,3.5) {10};
      \node (11) at (1.5,3.5) {12};
      \node (12) at (2.5,3.5) {11};
      \node(z) at (4.5,0) {};

      \draw[<->] (1.5,4) to [bend left=60] (2.5,4); 
    \end{tikzpicture}
  }
  \caption{Illustration of matrix symmetries arising in
    Example~\ref{ex:geompack}.}
  \label{fig:matrixsym}
\end{figure}
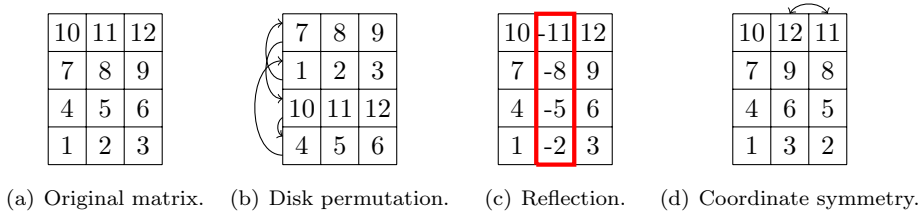

The described actions of the symmetry group of an MINLP on a matrix of
variables is rather generic and, as we will see in
Section~\ref{sec:numerics}, occurs in many applications.
We therefore discuss, after a short overview of symmetry handling
techniques for permutation symmetries, methods for row-exchanges,
column-reflections, and their combination with column-exchanges in more
detail.
We conclude by devising general techniques for handling reflection
symmetries.

\subsection{General Symmetries}

Consider an~\eqref{eq:MINLP} with variable vector~$x \in \R^n$
and reflection symmetry group~\mbox{$\sgroup \leq \ssym{n}$}.
A standard way for handling~$\sgroup$ is to define an
order of the variables, encoded by a permutation~$\perm \in \sym{n}$, and
to enforce that a solution~$x$ must be lexicographically maximal
in its~$\sgroup$-orbit.
Formally, one is looking for feasible solutions~$x$ of~\eqref{eq:MINLP}
satisfying~$\perm(x) \lexgeq
\perm(\sperm(x))$ for all~$\sperm \in \sgroup$, where~$\lexgeq$ denotes the
lexicographic comparison.
In general, enforcing the lexicographic order constraints in coNP-hard,
cf.~\cite{BabaiLuks1983}.
One therefore typically uses one of the following three alternatives.

First, one can handle the lexicographic order constraints for a
small subset of symmetries.
For binary variables, a constraint~$\perm(x) \lexgeq \perm(\sperm(x))$ can
be handled by adding linear inequalities to~\eqref{eq:MINLP}.
To fully encode the lexicographic comparison by linear inequalities,
however, one needs exponentially many inequalities or inequalities with
exponentially large coefficients~\cite{Hojny2020a}.
An exponentially large class with coefficients in~$\{0, \pm 1\}$ is
presented in~\cite{HojnyPfetsch2019}, whereas~\cite{Friedman2007} uses a
single inequality with exponentially large coefficients.
To overcome these drawbacks, an alternative is to propagate lexicographic
constraints during branch-and-bound.
That is, given the bounds on variables of a subproblem of branch-and-bound,
one derives further reductions of variable bounds that are satisfied by any
lexicographic maximal solution.
Such a propagation algorithm, \emph{lexicographic reduction}, is described
in~\cite{DoornmalenHojny2023} and runs in linear time.
In particular, the algorithm also works for non-binary variables.

Second, one can derive strong symmetry handling methods for
particular classes of symmetry groups.
Besides techniques for the row symmetries of the disk packing problem,
which we discuss in the next section, also efficient propagation algorithms
for particular cyclic groups exist~\cite{DoornmalenHojny2024}.

Third, there exist intermediate approaches that handle some but not
necessarily all group structure.
For example, \emph{orbital fixing}~\cite{OstrowskiEtAl2011} fixes
binary variables based on the branching decisions and permutation
symmetries of~\eqref{eq:MINLP}.
Recently, this method has been generalized to arbitrary variable
types~\cite{DoornmalenHojny2023}, so-called \emph{orbital reduction}.
Moreover, \cite{LibertiOstrowski2014,Salvagnin2018} discuss sparse linear
inequalities that are derived from
the Schreier-Sims table of a permutation group and that handle symmetries
based on variable orbits.

Recall that, by Definition~\ref{def:admissable}, symmetric variables can
have different variable domains due to translations.
For handling row and column symmetries, we assume in the following
that all symmetric variables have the same domain.

\subsection{Row Symmetries}
\label{sec:rowsymmetries}

A popular approach for handling row symmetries of a matrix~$M \in \R^{p
  \times q}$ is to lexicographically sort its rows,
i.e., if~$M_{i \cdot}$, $i \in [p]$, denotes the~$i$-th row of~$M$, one
enforces~$M_{1\cdot} \lexgeq M_{2\cdot} \lexgeq \dots \lexgeq
M_{p\cdot}$.
By interpreting~$M$ as a vector~$x \in \R^{pq}$ such that entry~$(i,j)$
of~$M$ is entry~$(i-1)q + j$ of~$x$, this can be achieved via the
lexicographic ordering constraints of the previous section using~$\perm =
\id$.
But since the group of all row permutations has size~$p!$,
it is impractical to add all ordering constraints.
Instead, it is folklore that sorting the rows can already be achieved by
the ordering constraints for the~$p-1$ permutations that exchange
consecutive rows, cf.~\cite{HojnyPfetsch2019}.
Methods such as lexicographic reduction can then be used to enforce the
ordering constraints.

Since these methods ignore the interplay of different row permutations,
they might not detect all reductions that can be derived from a matrix
with sorted rows.
An alternative is \emph{orbitopal reduction}~\cite{DoornmalenHojny2023}.
This propagation algorithm is called within a branch-and-bound algorithm
and receives bounds on all variables in~$M$ w.r.t.\ the current
subproblem.
The algorithm then identifies the tightest bounds for all
variables in~$M$ to which every sorted feasible solution of the subproblem
adheres to.
Originally, this algorithm has been derived for matrices of binary
variables~\cite{BendottiEtAl2021}, but the variant
of~\cite{DoornmalenHojny2023} is applicable for general integer and
continuous variables, too.
Moreover, \cite{KaibelEtAl2011} describe a variant of this algorithm if
all variables in~$M$ are binary and at most one (resp.\ exactly one)
variable per column is allowed to attain value~1.
A facet description of the convex hull of all such binary matrices~$M$ is
also known~\cite{KaibelPfetsch2008}; the facet-defining inequalities can
then be added to~\eqref{eq:MINLP} to handle row symmetries.

Above, we discussed a fixed assignment of the entries of~$M \in \R^{p \times q}$
to a vector~$x \in \R^{pq}$.
This requirement can be relaxed such that different assignments can be used
at different nodes of the branch-and-bound tree.
For example, \cite{BendottiEtAl2021} discusses a mechanism to change the
order of the columns of~$M$ before translating~$M$ into a vector~$x$.
A variant in which also the order of the rows can be exchanged is discussed
in~\cite{DoornmalenHojny2023}.
We refer to both variants as \emph{column-dynamic} and \emph{column-row-dynamic}
orbitopal fixing, respectively.
The variant without any reordering is referred to as \emph{static}
orbitopal fixing.

\subsection{Column Reflections}
\label{sec:colrefl}

Let~$M \in \R^{D \times 2}$ be the variable matrix for the disk packing
problem for~$D$ disks.
Then, $M$ admits row symmetries and reflection symmetries of
individual columns.
Let~$n_1 = \lceil \frac{D}{2} \rceil$ and~$n_2 = \lceil \frac{n_1}{2}
\rceil$.
To handle reflection symmetries in the regime of disk packing,
\cite{Khajavirad2017} enforces~$M_{i,1} \geq 0$ for all~$i \in
[n_1]$ and~$M_{i,2} \geq 0$ for all~$i \in [n_2]$.
Geometrically, this means that there are at least as many disks in the
right half of the box than in the left half; moreover, within the right
half, there are not less disks in the upper half than in the lower half.

Formally, we can derive these reductions as follows.
After possibly applying a reflection of the first column, one can guarantee
that the first column has at least as many non-negative entries as
non-positive entries.
Due to row symmetries, the first~$n_1$ entries can thus be enforced to be non-negative.
This argument can be repeated for the first~$n_1$ rows of the second
column, because enforcing the non-negativity structure on the first column does not allow
to exchange one of the first~$n_1$ rows of~$M$ with one of the last~$D -
n_1$ rows.

The latter argument also shows that one can sort the first~$n_1$ rows
of~$M$ and the last~$D - n_1$ rows lexicographically.
To partially enforce this, \cite{Khajavirad2017} suggests to add the
inequalities~$M_{i,1} \geq M_{i+1,1}$ for~$i \in [D - 1] \setminus
\{n_1\}$.
A similar sorting idea has been applied for the kissing
number problem in~\cite{Liberti2012}.

The idea of~\cite{Khajavirad2017} can be generalized to
matrices with more than two columns.
We provide the following result without a proof, see
Figure~\ref{fig:rowsymrefl} for an illustration.
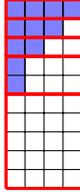
\begin{figure}[t]
  \centering
  \begin{tikzpicture}[scale=0.25]
    \draw[fill=blue!50] (0,5) -- (1,5) -- (1,10) -- (0,10) -- (0,5);
    \draw[fill=blue!50] (1,7) -- (2,7) -- (2,10) -- (1,10) -- (1,7);
    \draw[fill=blue!50] (2,8) -- (3,8) -- (3,10) -- (2,10) -- (2,8);
    \draw[fill=blue!50] (3,9) -- (4,9) -- (4,10) -- (3,10) -- (3,9);
    \draw (0,0) grid (4,10);
    \draw[-,red,very thick] (0,0) -- (4,0) -- (4,5) -- (0,5) -- (0,0);
    \draw[-,red,very thick] (0,5) -- (4,5) -- (4,7) -- (0,7) -- (0,5);
    \draw[-,red,very thick] (0,7) -- (4,7) -- (4,8) -- (0,8) -- (0,7);
    \draw[-,red,very thick] (0,8) -- (4,8) -- (4,9) -- (0,9) -- (0,8);
    \draw[-,red,very thick] (0,9) -- (4,9) -- (4,10) -- (0,10) -- (0,9);
  \end{tikzpicture}
  \caption{Row symmetries combined with reflection symmetries of columns.
    Rows within the red blocks can be exchanged arbitrarily, whereas blue
    cells can be restricted to the upper half of the domain.
  }
  \label{fig:rowsymrefl}
\end{figure}
\begin{proposition}
  \label{prop:handlerefl}
  Let~$M \in \R^{p \times q}$ be a variable matrix admitting row symmetries
  and reflection symmetries for individual columns.
  Suppose that, for all~$j \in [q]$, all variables in column~$j$
  of~$M$ have the same domain center~$\dcenter_j \in \R$.
  Let~$n_0 = p$ and, for~$j \in [q]$, let~$n_j = \lceil \frac{n_{j-1}}{2}
  \rceil$.
  Then, a valid symmetry handling approach is given by
  \begin{itemize}
  \item adding, for each~$j \in [q]$ and~$i \in [n_j]$, the
    inequality~$M_{i,j} \geq \dcenter_j$, and
  \item enforcing, for each~$j \in [q] \cup \{0\}$, that~$M_{n_j + 1,
      \cdot} \lexgeq \dots \lexgeq M_{n_{j-1}, \cdot}$.
  \end{itemize}
\end{proposition}
The lexicographic sorting can be enforced, e.g., via static orbitopal
reduction.
A dynamic variant of orbitopal reduction is not immediately applicable due
to the additional lower bound constraints derived from the reflection
symmetries.

\subsection{Row and Column Symmetries}

In contrast to row symmetries only, it is coNP-complete to decide whether a
matrix~$M \in \{0,1\}^{p \times q}$ is lexicographically maximal w.r.t.\
row and column symmetries~\cite{BessiereEtAl2004}.
There is consequently no efficient mechanism to simultaneously handle row and
column symmetries unless~$\text{P}=\text{coNP}$.
A possible remedy to handle at least some symmetries is to enforce that both the
rows and columns of a solution matrix~$M \in \R^{p \times q}$
of~\eqref{eq:MINLP} are sorted lexicographically~\cite{FlenerEtAl2002}.

Sorting the rows can be achieved by static orbitopal reduction as
described before.
Moreover, sorting of the columns can be achieved by applying static
orbitopal reduction to the transposed variable matrix.
This is indeed compatible as one checks via the lexicographic ordering
constraints:
Recall that we associated with~$M \in \R^{p \times q}$ the vector~$x \in
\R^{pq}$ such that~$M_{i,j}$ corresponds to~$x_{(i-1)p + j}$.
As mentioned above, the rows of~$M$ are sorted
lexicographically if and only if~$x \lexgeq \sperm(x)$ for all row
permutations~$\sperm$.
However, one can show that also the vector~$x'$ that identifies~$M_{i,j}$
with~$x'_{(j-1)q + i}$ has the same properties.
That is, since~$x'$ corresponds to the first mechanism for the transpose
of~$M$, orbitopal reduction for both~$M$ and its transpose can be combined.

Due to the interaction of row and column symmetries, we stress that a
dynamic variant of orbitopal reduction is not immediately applicable.
Moreover, if also reflection symmetries act on individual columns, handling
row symmetries can be replaced by the mechanism in
Proposition~\ref{prop:handlerefl}.

\subsection{Further Techniques for Handling Reflection Symmetries}
\label{sec:furthertechniques}

We describe two further techniques for handling reflection symmetries that
also apply if no row or column symmetries are present in the problem.
The first one is a simple inequality that handles some, but by far
not all symmetries.
Consider an instance of~\eqref{eq:MINLP} that contains the signed
permutation~$\sperm^\star$ that reflects \emph{all} variables simultaneously,
i.e., $\sperm^\star(i) = -i$ for all~$i \in [n]$.
Instead of lexicographically comparing a solution~$x$ of~\eqref{eq:MINLP}
with~$\sperm^\star(x)$, one can enforce that the aggregated weight
of~$x$ is larger than the weight of~$\sperm^\star(x)$, i.e.,
\begin{equation}
  \label{eq:simpleRefl}
  \sum_{i = 1}^n x_i
  \geq
  \sum_{i = 1}^n \sperm^\star(x)_i
  =
  \sum_{i = 1}^n (2\dcenter_i - x_i)
  \qquad
  \Leftrightarrow
  \qquad
  \sum_{i = 1}^n x_i
  \geq
  \sum_{i = 1}^n \dcenter_i.
\end{equation}
This inequality can, e.g., be used for the max-cut problem~\cite{KorteVygen2018}, where
the aim is to partition the node set of an undirected graph~$G = (V,E)$
into two sets such that the weight of edges between both sets is
maximized.
If a binary variable~$x_v$, $v \in V$, indicates whether~$v$ is contained
in the first set ($x_v = 1$) or not ($x_v = 0$), all variables are
symmetric w.r.t.\ the described reflection symmetry.
Inequality~\eqref{eq:simpleRefl} then models that at least as many nodes
are in the first set of the partition than in the second set.

Note that~\eqref{eq:simpleRefl} cannot necessarily be combined
with lexicographic order based techniques.
For example, vector~$e^1 = (1,0,\dots,0)$
satisfies~$e^1 \lexgeq \sperm^\star(e^1)$, but
violates~\eqref{eq:simpleRefl}.
We therefore also discuss a second approach: a generalization of
lexicographic reduction to reflection symmetries.

The lexicographic reduction algorithm of~\cite{DoornmalenHojny2023}
receives a permutation symmetry~$\perm$ of~\eqref{eq:MINLP} as well as
upper and lower bounds for all variables as input.
The algorithm then finds, for each variable, the tightest lower and upper
bounds such that any solution~$x$ satisfying the initial bounds and~$x \lexgeq \perm(x)$
adheres to the new bounds.
Their algorithm can be immediately generalized to reflection
symmetries~$\sperm \in \ssym{n}$.
Since the arguments for correctness are the same as
in~\cite{DoornmalenHojny2023}, we only provide the high-level ideas but no
formal proof.

The core of lexicographic reduction is the following observation:
Suppose there is~$i \in [n]$ such that~$x_j = \sperm(x)_j$ for all~$j \in
[i-1]$.
Then, $x \lexgeq \sperm(x)$ can only hold if~$x_i \geq \sperm(x)_i$.
In this case, the lower bound of variable~$x_i$ can possibly be
strengthened to the lower bound of~$\sperm(x)_i$ and, vice versa, the upper
bound of~$\sperm(x)_i$ can possibly be strengthened to the upper bound
of~$x_i$.

To exploit this observation for~$\sperm \in \ssym{n}$, lexicographic
reduction iterates over the entries~$x_i$, $i \in [n]$, of a solution~$x$
and checks whether~$x_j = \sperm(x)_j$ holds for all~$j \in [i-1]$ by
comparing the upper and lower bounds of the variables.
If this is the case, three cases are distinguished.
First, the variable bounds imply~${x_i < \sperm(x)_i}$.
Then, the algorithm reports infeasibility because the variable bounds
imply~$x \nlexgeq \sperm(x)$.
Second, the bounds imply~$x_i > \sperm(x)_i$.
Then the algorithm stops since every solution adhering to the bounds
satisfies~${x \lexgt \sperm(x)}$.
Third, the upper and lower bounds on the variables can possibly be
improved due to the observation and the algorithm continues with the next
iteration and strengthened bounds.

If the algorithm terminates and did not report infeasibility, lexicographic
reduction has possibly improved some variable bounds.
As discussed in~\cite{DoornmalenHojny2023}, these bounds are as tight as
possible for all variables except for the lower bound on~$x_i$ and the
upper bound on~$\sperm(x)_i$ of the last iteration~$i$.
In a post-processing step, lexicographic reduction can determine whether
these bounds can be improved further, see~\cite{DoornmalenHojny2023} for
details.
The running time of lexicographic reduction is~$O(n)$.

\section{Numerical Experience}
\label{sec:numerics}

In this section, we evaluate our symmetry detection framework when solving
MILPs and MINLPs.
Specifically, we aim to answer the following questions:
\begin{enumerate}[label=(Q\arabic*), ref=(Q\arabic*),leftmargin={1cm}]
\item\label{Q1} Does the framework of Sections~\ref{sec:detection}
  and~\ref{sec:implementation} reliably detect reflection symmetries?

\item\label{Q2} What is the most effective way to enforce the lexicographic order
  constraints in Proposition~\ref{prop:handlerefl}?

\item\label{Q3} How frequently arise reflection symmetries in benchmarking instances?

\item\label{Q4} Is there a computational benefit for benchmarking instances when
  handling reflection symmetries instead of permutation symmetries?
\end{enumerate}
To answer~\ref{Q1}, we apply our framework to specific applications in
which we know that reflection symmetries arise and we evaluate whether
these symmetries can be detected (Section~\ref{sec:numericsApplications}).
Many of these instances contain row and column symmetries, which will allow
us to answer~\ref{Q2}.
In Section~\ref{sec:numericsBenchmark}, we turn the focus to instances
from three benchmarking test sets, which allows us to answer~\ref{Q3}
and~\ref{Q4}.
Before answering the questions, we provide an overview of how symmetries
are detected and handled in Section~\ref{sec:numericsSetup}.

\subsection{Computational Setup}
\label{sec:numericsSetup}

As mentioned in Section~\ref{sec:implementation}, we included a
\code{C}-implementation of our detection framework into the solver \scip.
Our implementation is publicly available in \scip; detailed
instructions on how to reproduce our results are available at~\cite{supplement}.
For most of \scip's default constraint handlers, we have implemented the
two callbacks for detecting permutation and reflection symmetries.
\scip allows to compute symmetries of the original problem or of the
problem after presolving.
In our experiments, we considered symmetries of the presolved problem,
which is also \scip's default setting.
\scip then builds the corresponding \SDG, computes its automorphism group,
and returns a set~$\sgroup'$ of signed permutations (so-called
\emph{generators}) that generates a group~$\sgroup$ of symmetries of the
MINLP.
That is, we do not have explicit access to all symmetries in~$\sgroup$, but
we can write every~$\sperm \in \sgroup$ as a finite composition of
elements from~$\sgroup'$.

\scip offers a variety of state-of-the-art methods for handling symmetries.
To decide which method is used, \scip analyzes the symmetry group~$\sgroup
\leq \ssym{n}$ and checks whether it is the direct product of smaller
groups, i.e., $\sgroup = \bigotimes_{s = 1}^t \sgroup_s$ for some signed
permutation groups~$\sgroup_s$, $s \in [t]$.
Each factor~$\sgroup_s$ can then be handled independently from the others,
and \scip calls different heuristics for each factor to decide how the
symmetries are handled.
The heuristics aim to detect structured symmetry groups in the following
order:
\begin{enumerate}
\item row and column symmetries;
\item row symmetries;
\item unclassified symmetry groups.
\end{enumerate}
We briefly describe how we detect and handle these symmetries.
For the ease of notation, we assume that~$\sgroup$ has a single factor.

\paragraph{Row and Column Symmetries.}
The group of row symmetries
is generated by permutations that exchange pairs of consecutive rows of
a matrix, cf.\ Section~\ref{sec:rowsymmetries}.
If~$x,y \in \R^q$ denote such rows, the corresponding row
permutation is~$(x_1,y_1)(x_2,y_2) \dots (x_q,y_q)$, i.e., it is a
composition of cycles of length~2.
Analogously, the group of column symmetries is generated by exchanges of
consecutive columns.
We refer to the number of 2-cycles as the composition's length.

When detecting row and column symmetries, however, the underlying matrix is unknown.
Our heuristic therefore collects all generators of~$\sgroup$
corresponding to permutation symmetries.
If there is a permutation that is not a composition of~2-cycles,
the heuristic stops: no row/column symmetries are detected.
Otherwise, the permutations are partitioned based on the length
of their composition.
If this partition consists of two sets, the heuristic tries to construct a
matrix~$M \in \R^{p \times q}$ such that the permutations of the first and
second set correspond to row and column symmetries of the matrix, respectively.

If row and column symmetries are detected, one can handle symmetries by
enforcing that the rows and columns of the matrix are sorted
lexicographically, cf.\ Section~\ref{sec:handling}.
Moreover, the heuristic also checks whether there is a permutation that
corresponds to the reflection of a single column (or row).
Since all columns (or rows) are symmetric, this means that the
group~$\sgroup$ also contains the remaining column (or row) reflections.
In this case, we can strengthen the symmetry handling approach by
restricting the variable domain of some variables according to
Proposition~\ref{prop:handlerefl}.

To enforce sorted rows and columns in~$M$ in the absence of reflections, we
explicitly add the inequalities~$M_{i,1} \geq M_{i+1,1}$ for all~$i \in [p-1]$ and~$M_{1,j}
\geq M_{1,j+1}$ for all~$j \in [q-1]$, respectively, to the MINLP.
Moreover, we use static orbitopal reduction to sort the rows and columns.
Note that the inequalities are in principle also enforced by orbitopal
reduction.
Preliminary experiments showed, however, that the dual bounds can
improve when adding the inequalities explicitly.
This is in one line with an observation by~\cite{CostaHansenLiberti2013}
that symmetry handling inequalities can strengthen relaxations for
MINLPs.
In the presence of reflection symmetries for columns, we also add
inequalities and use orbitopal reduction to sort the rows, but we restrict
both techniques to the blocks of symmetric rows as indicated by
Proposition~\ref{prop:handlerefl}.
Additionally, we restrict some variable domains to their upper half as
described in Proposition~\ref{prop:handlerefl}.
If there are row reflections instead of column reflections, we proceed
analogously.

\paragraph{Row Symmetries.}
To detect row symmetries, we use an analogous mechanism as for row and
column symmetries.
The only difference is that the heuristic can only be successful if all
permutations have the same number of~2-cycles.

In the presence of column reflections, row symmetries are handled
analogously to the case of row and column symmetries.
If no column reflections are detected, we only handle row
symmetries if the number of generators in~$\sgroup'$ being permutation
symmetries is greater than~\SI{80}{\percent}.
Our reasoning is that if the percentage of permutation symmetries is too
low, we potentially miss many symmetry reductions being implied by
proper reflection symmetries.
In this case, we proceed with general techniques as explained in the
next paragraph.

If row symmetries for a matrix~$M \in \R^{p \times q}$ are handled, we use
dynamic orbitopal fixing besides the following exceptions:
\begin{itemize}
\item If there is only one column, we add the
  inequalities~$M_{1,1}\geq \dots \geq M_{p,1}$ to fully
  handle the symmetries.

\item If there are exactly two rows, we use lexicographic reduction to
  handle the symmetries of the corresponding permutation exchanging the
  two rows.

\item Suppose there are at least three rows, and there are at least three columns
  containing only binary variables such that other problem constraints
  enforce that the column sums are exactly or at most~1.
  Then, we use orbitopal fixing for packing and partitioning
  orbitopes~\cite{KaibelEtAl2011} for the submatrix containing the
  cardinality restricted columns.
\end{itemize}
These are also the default settings in \scip to handle row symmetries.
The motivation for the first, second, and third point is to add strong
symmetry handling inequalities, more flexible symmetry handling techniques,
or techniques exploiting additional problem structure, respectively.
The hope is that these techniques are more powerful than dynamic
orbitopal reduction.

\paragraph{General Symmetries.}
We test two different approaches to handle reflection symmetries.
The first approach uses orbital reduction and lexicographic reduction;
the second approach just adds~\eqref{eq:simpleRefl} if it is
applicable and no other symmetry handling method is used.

To apply the first approach, let~$\Pi'$ be the generators in~$\sgroup'$
that are permutation symmetries.
The group generated by~$\Pi'$ is then handled by orbital reduction
as implemented in \scip.
Moreover, for each generator~$\sperm \in \sgroup'$, we enforce a
lexicographic order constraint~$\perm(x) \lexgeq \sperm(\perm(x))$ by
lexicographic reduction, cf.\ Section~\ref{sec:furthertechniques}.
The reordering~$\perm$ can differ at each node of the
branch-and-bound tree and is defined based on the branching history of the
respective node.
This allows to apply lexicographic reduction and orbital reduction
simultaneously~\cite{DoornmalenHojny2023}.

\paragraph{Hardware and Software Specifications.}
All of the following experiments have been conducted on a Linux cluster
with Intel Xeon E5-1620 v4 \SI{3.5}{\GHz} quad core processors and
\SI{32}{\giga\byte} memory.
The code was executed using a single thread and the time limit for all
computations was \SI{2}{\hour} per instance.
A memory limit of \SI{27}{\giga\byte} has been used per instance.

We use a developers version of the branch-and-bound framework \scip~10.0
(githash a93d088d).
All LP relaxations are solved with \code{Soplex}~8.0 (githash 89ab43a)
and nonlinear problems are solved with
\code{Ipopt}~3.12~\cite{WachterBiegler2006}.
To detect symmetries of \SDGs, we use the software
\code{sassy}~1.1~\cite{AndersEtAl2023}
to preprocess \SDGs; symmetries of the preprocessed \SDGs are computed
using \code{bliss}~0.77~\cite{JunttilaKaski2011}.

\medskip
\noindent
To compare different settings in our experiments, we use, among others,
mean numbers.
All mean numbers for quantities~$t_1,\dots,t_n$ are reported in shifted
geometric mean~$\prod_{i = 1}^n (t_i + s)^{\frac{1}{n}} - s$ to reduce the
impact of outliers.
For mean running times, a shift of~$s = 1$ is used; otherwise, we use a
shift of~$s = 0$, i.e., the classical geometric mean.

\subsection{Results for Structured Instances}
\label{sec:numericsApplications}

To answer Question~\ref{Q1}, we consider four different classes of
instances with different types of nonlinearities that all admit reflection
symmetries.
Moreover, three of them admit row and column symmetries.
The four considered problems, which we discuss next, are a geometric
packing problem, the kissing number problem, an energy minimization
problem, and the max-cut problem.

\paragraph{Geometric Packing Problem.}
Let~$d$ and~$n$ be positive integer numbers.
We aim to find the largest real number~$r$ such that~$n$
non-overlapping $d$-dimensional~$\ell_1$-balls of radius~$r$ can be packed
into~$[-1,1]^d$.
This problem can be modeled as
\begin{align*}
  \max r &&&\\
  \sum_{i = 1}^d \abs{x^s_i - x^t_i} &\geq 2r, && s,t \in [n],\; s < t,\\
  -1 + r \leq x^s_i &\leq 1 - r, && s \in [n],\; i \in [d],
\end{align*}
where~$r$ models the radius of the balls and~$x^s_i$, $(i,s) \in [d] \times
[n]$, is the~$i$-th coordinate of the~$s$-th ball's center.
Interpreting~$x^s_i$ as the entries of an~$n \times d$ matrix, this
formulation admits row and column symmetries as well as reflection
symmetries of individual columns.
Although this model can easily be linearized, we do not do so to check
whether symmetry detection works for constraints involving absolute
values.

\paragraph{Kissing Number Problem.}
Let~$d$ and~$n$ be positive integers and let~$S$ be the unit sphere in~$\R^d$
w.r.t.\ the~$\ell_2$-norm that is centered at the origin.
The kissing number problem asks whether~$n$ unit spheres can be placed
in~$\R^n$ such that each of the~$n$ spheres touches~$S$ and neither pair of
the~$n$ spheres intersects.
An optimization variant of this problem is looking for the maximum distance
of the center points of the~$n$ spheres~\cite{Liberti2012}:
\begin{align*}
  \max \alpha &&&\\
  \sum_{i = 1}^d (x^s_i)^2 &= 4, && s \in [n],\\
  8 - 2 \sum_{i = 1}^d x^s_ix^t_i &\geq 4\alpha, && s,t \in [n],\; s < t,\\
  -2 \leq x^s_i &\leq 2, && s \in [n],\; i \in [d],\\
  0 \leq \alpha &\leq 1, &&
\end{align*}
where the interpretation of~$x^s_i$ is the same as for the geometric
packing problem.
The second constraint indeed measures the distance
between two center points as~$\|x^s - x^t\|_2^2 = 8 - 2\sum_{i = 1}^d
x^s_ix^t_i$.
As before, this problem admits row and column symmetries with column reflections.

\paragraph{Energy Minimization Problem.}
In the energy minimization problem, we consider~$n$ particles that need to
be distributed on a~$d$-dimensional unit sphere~\cite{SaffKuijlaars1997}.
If~$x^s_i$ models the~$i$-th coordinate of the~$s$-th particle, the
energy~$\frac{1}{\sum_{i = 1}^d (x^s_i - x^t_i)^2}$ is needed to keep
particles~$s, t \in [n]$ at their distance.
The goal is to distribute the particles such that their total energy is
minimized, i.e., to solve
\begin{align*}
  \min \sum_{s = 1}^n \sum_{t = s+1}^n \frac{1}{\sum_{i = 1}^d (x^s_i - x^t_i)^2} &&&\\
  \sum_{i = 1}^d (x^s_i)^2 &= 1, && s \in [n],\\
  -1 \leq x^s_i &\leq 1, && s \in [n],\; i \in [d].
\end{align*}
Also this problem admits row and column symmetries with column reflections.

\paragraph{Max-Cut.}
Given an undirected graph~$G = (V,E)$, the maximum cut problem asks for a
partition of the nodes~$V$ into two sets such that the number of edges
between the two sets is maximized.
This problem can be modeled as
\begin{align*}
  \max \sum_{e \in E} y_e &&&\\
  x_u + x_v + y_{\{u,v\}} &\leq 2, && u,v \in V,\\
  -x_u - x_v + y_{\{u,v\}} &\leq 0, && u,v \in V,\\
  x_v &\in \{0,1\}, && v \in V,\\
  y_e &\in \{0,1\}, && e \in E,
\end{align*}
where~$x_v$ indicates to which of the two sets node~$v \in V$ belongs to
and~$y_e$ encodes whether edge~$e \in E$ has its endpoints in different
sets of the partition.
This problem admits the reflection that simultaneously maps~$x_v
\mapsto 1 - x_v$ for all~$v \in V$.
Moreover, every automorphism~$\perm$ of~$G$ gives rise to a permutation
symmetry that relabels the indices of~$x$ and~$y$ according to~$\perm$.

\paragraph{Test Sets.}
We consider four test sets corresponding to the aforementioned problems.
We refer to the test sets as ``packing'', ``kissing'', ``energy'', and
``maxcut''.
The packing, kissing, and energy test set contain all instances with~$n$
objects in dimension~$d$, where~$(n,d) \in \{3,\dots,14\} \times \{2,3\}$.
The maxcut test set contains an instance for each graph from the Color02
symposium\footnote{Obtained
  from~\url{https://mat.tepper.cmu.edu/COLOR02/}. We removed the graph
  \texttt{games120} because the different settings hit the memory limit
  while solving the instance.}.

\paragraph{Discussion Question~\ref{Q1}.}
To answer \ref{Q1}, we have computed symmetries for all instances after
presolving.
Table~\ref{tab:statistics} summarizes our results, where column ``total''
contains the total number of instances per test set and ``sym.'' the
number of instances for which symmetries are detected.
Column group ``generators'' lists the number of generators being signed
(``sig.'') and unsigned (``unsig.'') permutations.
Column groups ``row + column'' and ``row/column'' provide the number
of instances containing row and column symmetries, and row or column
symmetries, respectively.
We distinguish whether the rows/columns can be reflected within the columns
(``sig.'' and ``unsig.'').
Finally, column ``simple'' reports on how many instances allow to handle
symmetries via Inequality~\eqref{eq:simpleRefl}.

As Table~\ref{tab:statistics} shows, we detect both permutation
symmetries and signed permutation symmetries for all instances in the
packing, kissing, and energy test set.
Except for one instance per test set, row and column symmetries combined
with column reflections are detected for all instances.
The only exception are instances with specification~$(n,d) = (3,3)$,
since both row and column symmetries have the same number of 2-cycles in
this case.
Our heuristic for detecting row and column symmetries thus cannot
distinguish whether a permutation corresponds to a row or column
permutation.
Moreover, for each maxcut instance, we detect signed
permutation symmetries.
In particular, for~25 instances whose graphs are asymmetric, we also detect
the applicability of Inequality~\eqref{eq:simpleRefl}.
We conclude that our symmetry detection framework indeed reliably detects
reflection symmetries.

\begin{table}[t]
  \caption{Statistics on how many instances allow for a particular symmetry handling method.}
  \label{tab:statistics}
  \centering
  \begin{tabular*}{\textwidth}{@{}l@{\;\;\extracolsep{\fill}}rrrrrrrrr@{}}
    \toprule
    & \multicolumn{2}{c}{\# instances} & \multicolumn{2}{c}{generators} & \multicolumn{2}{c}{row + column} & \multicolumn{2}{c}{row/column} & \\
    \cmidrule{2-3} \cmidrule{4-5} \cmidrule{6-7} \cmidrule{8-9}
    test set & total & sym. & sig. & unsig. & sig. & unsig. & sig. & unsig. & simple\\
    \midrule
    \multicolumn{10}{@{}l}{structured instances:}\\
               packing &   24 &   24 &   24 &   24 &   23 &    0 &    0 &    0 &    0\\
               kissing &   24 &   24 &   24 &   24 &   23 &    0 &    0 &    0 &    0\\
                energy &   24 &   24 &   24 &   24 &   23 &    0 &    0 &    0 &    0\\
                maxcut &  119 &  119 &  119 &   94 &    0 &    0 &    0 &    0 &   25\\
    \midrule
    \multicolumn{10}{@{}l}{benchmarking instances:}\\
            miplib2017 & 1055 &  504 &   60 &  485 &    0 &    1 &    0 &  362 &    2\\
              minlplib &  486 &  111 &    6 &  109 &    4 &    0 &    0 &  100 &    2\\
               sat2002 &  819 &  292 &  282 &  125 &    0 &    0 &    9 &   78 &   54\\
    \bottomrule
  \end{tabular*}
\end{table}

\paragraph{Discussion Question~\ref{Q2}.}
To answer~\ref{Q2}, we have conducted experiments on the packing, kissing,
and energy test set.
We compare different ad-hoc approaches, which add symmetry handling
inequalities directly to the problem formulation, with the more
sophisticated techniques for row and column symmetries as described in
Section~\ref{sec:numericsSetup}.
We refer to the latter setting as the ``automatic'' setting that is applied
by \scip.
The remaining settings are inspired by the inequalities
in~\cite{Khajavirad2017} for the disk packing problem, also compare
Proposition~\ref{prop:handlerefl}.
Assuming that the corresponding variable matrix is~$M \in \R^{p
  \times q}$, these settings are:
\begin{description}
\item[sym0:] no symmetry handling inequalities are added;
\item[sym1:] add the inequalities~$M_{1,1} \geq M_{1,2} \geq \dots \geq M_{1,q}$;
\item[sym2:] as sym1, additionally add~$M_{1,1},\dots,M_{1,q} \geq 0$;
\item[sym3:] as sym2, additionally add inequalities~$M_{1,1} \geq M_{2,1}
  \geq \dots \geq M_{p,1}$;
\item[sym4:] using the notation of Proposition~\ref{prop:handlerefl},
  add~$M_{i,j} \geq 0$ for~$j \in [q]$ and~$i \in [n_j]$;
\item[sym5:] as sym4, also add~$M_{n_j + 1,1} \geq M_{n_j + 2} \geq
  \dots \geq M_{n_{j-1},1}$ for~$j \in [q] \cup \{0\}$;
\item[sym6:] combination of sym1 and sym5.
\end{description}
Setting sym0 thus serves as a baseline to compare the different symmetry
handling methods with an approach that does not handle symmetries.
Moreover, note that settings sym1--sym3 and sym4--sym6 form two groups of
settings, in which a higher number indicates to use more symmetry handling
inequalities.

Tables~\ref{tab:packing}--\ref{tab:energy} summarize our results, where
columns ``\# solved'', ``time'', and ``primal-dual'' provide the number of
solved instances, the mean running time, and the mean primal-dual integral
per test set, respectively.
The primal-dual integral is a proxy for the speed of convergence of a
branch-and-bound algorithm (roughly speaking, it measures the area between the
primal and dual bounds plotted against time), see~\cite{Berthold2013} for details.
We included this measure, because we observed a very slow convergence of
the dual bounds for some settings.
Since optimal solutions are found usually rather quickly for these
instances, a long running time but small integral indicates that
almost optimal dual bounds are found early, but closing the gap completely
is challenging.
We rate an instance as solved when the primal-dual gap is
below~$\SI{0.1}{\percent}$.

\begin{table}
  \caption{Comparison of running times and primal-dual integrals for packing test set.}
  \label{tab:packing}
  \centering
  \begin{tabular*}{\textwidth}{@{}l@{\;\;\extracolsep{\fill}}rrrrrr@{}}
    \toprule
    & \multicolumn{3}{c}{dimension 2} & \multicolumn{3}{c}{dimension 3}\\
    \cmidrule{2-4} \cmidrule{5-7}
    setting & \# solved & time & primal-dual & \# solved & time & primal-dual\\
    \midrule
                sym0 &   4 &  605.02 &     14591.6 &   3 & 1928.92 &     43786.1\\
                sym1 &   4 &  561.59 &     12891.3 &   3 & 1633.94 &     34028.4\\
                sym2 &   6 &  457.49 &      8910.5 &   4 & 1308.92 &     21674.9\\
                sym3 &   8 &   89.90 &      1238.4 &   5 &  639.00 &     10042.2\\
                sym4 &   6 &  205.96 &      2150.2 &   5 &  815.80 &      8748.7\\
                sym5 &   9 &   64.74 &       476.3 &   5 &  655.24 &      6845.1\\
                sym6 &   9 &   58.60 &       584.9 &   5 &  508.59 &      5680.2\\
           automatic &   9 &   51.86 &       475.9 &   5 &  516.43 &      5266.5\\
    \bottomrule
  \end{tabular*}
\end{table}

The results show that symmetry handling is an important aspect for solving
the three classes of problems more effectively.
As expected, throughout all test sets (except for 3-dimensional kissing
number problems, which are very easy), we observe that adding more symmetry
handling inequalities to the problem formulation results in better running
times.
Methods sym3 and sym6 are thus the most effective settings within their
groups.
There is, however, no clear trend whether sym3 or sym6 is more effective.
For example, for the kissing number problem and 2-dimensional energy problems, sym3 is faster,
whereas sym6 dominates sym3 for the packing problem and
3-dimensional energy problems.
In contrast, the automatic symmetry handling setting consistently dominates
the other settings on 2-dimensional problems, and on the 3-dimensional
packing problem it is almost the fastest approach.
Compared to the best competitor, the automatic symmetry handling setting is
\SI{10.1}{\percent} (2-dimensional packing), \SI{52.0}{\percent} (2-dimensional kissing), and
\SI{2.4}{\percent} (2-dimensional energy) faster.
On 3-dimensional energy problems, however, the automatic setting is slower
than sym3 and sym6.

\begin{table}
  \caption{Comparison of running times and primal-dual integrals for kissing test set.}
  \label{tab:kissing}
  \centering
  \begin{tabular*}{\textwidth}{@{}l@{\;\;\extracolsep{\fill}}rrrrrr@{}}
    \toprule
    & \multicolumn{3}{c}{dimension 2} & \multicolumn{3}{c}{dimension 3}\\
    \cmidrule{2-4} \cmidrule{5-7}
    setting & \# solved & time & primal-dual & \# solved & time & primal-dual\\
    \midrule
                sym0 &   4 &  381.16 &      7814.8 &  10 &    3.60 &        24.5\\
                sym1 &   4 &  400.52 &     10813.1 &  10 &    3.87 &        26.8\\
                sym2 &   5 &  342.43 &      9631.0 &  10 &    3.76 &        32.6\\
                sym3 &  12 &   22.00 &       361.3 &  10 &    3.53 &        17.8\\
                sym4 &   7 &  150.71 &      1885.1 &  10 &    3.69 &        36.0\\
                sym5 &  11 &   32.96 &       546.5 &  10 &    4.79 &        44.5\\
                sym6 &  11 &   24.88 &       330.7 &  10 &    3.80 &        21.9\\
           automatic &  12 &   10.57 &       170.0 &  10 &    3.77 &        21.6\\
    \bottomrule
  \end{tabular*}
\end{table}
\begin{table}
  \caption{Comparison of running times and primal-dual integrals for energy test set.}
  \label{tab:energy}
  \centering
  \begin{tabular*}{\textwidth}{@{}l@{\;\;\extracolsep{\fill}}rrrrrr@{}}
    \toprule
    & \multicolumn{3}{c}{dimension 2} & \multicolumn{3}{c}{dimension 3}\\
    \cmidrule{2-4} \cmidrule{5-7}
    setting & \# solved & time & primal-dual & \# solved & time & primal-dual\\
    \midrule
                sym0 &   3 & 1955.90 &     25054.0 &   0 & 7200.00 &     96744.4\\
                sym1 &   3 & 1686.17 &     21190.0 &   1 & 6444.70 &     77109.6\\
                sym2 &   3 & 1337.50 &     16811.4 &   1 & 5294.92 &     51906.2\\
                sym3 &   6 &  346.94 &      2470.0 &   2 & 4398.29 &     33854.0\\
                sym4 &   5 &  695.68 &      5615.2 &   1 & 5608.33 &     45764.8\\
                sym5 &   7 &  402.21 &      2499.0 &   1 & 5607.37 &     40648.6\\
                sym6 &   6 &  372.11 &      2384.9 &   2 & 3933.79 &     26150.2\\
           automatic &   6 &  338.69 &      1721.1 &   2 & 5275.07 &     32872.1\\
    \bottomrule
  \end{tabular*}
\end{table}

Investigating running times on a per-instance basis, see
Appendix~\ref{app:moretimes} for detailed numbers, reveals that the
automatic setting achieves the best results for the four aforementioned test
sets for instances with a medium number of objects.
For 2-dimensional packing problems with~9 and~10 balls it works
substantially better than the competitors, whereas it works best for the
3-dimensional packing problem with~6 balls, the 2-dimensional kissing
number problem with at least~9 spheres, and~9 points in the 2-dimensional
energy problem.
For instances with less objects, the automatic setting performs slightly
worse than its competitors.
Taking these observations into account might provide an explanation why
the automatic setting performs worse than sym3 and sym6 on 3-dimensional
energy instances:
These instances are very challenging and the automatic setting as well as
sym3 and sym6 only solve the instances with~3 and~4 points within the time
limit of two hours.
Extrapolating the results from the easier test sets thus could indicate
that the automatic setting becomes more effective on instances with more
points.
To verify this hypothesis, we wanted to run the experiments with a higher
time limit.
However, all settings hit the memory limit before solving the more
difficult instances.
We thus could neither refute nor verify the hypothesis.

Besides the comparison of running times, note that, for the four easier
test sets, the automatic setting has consistently a much smaller
primal-dual integral than the best competitor.
Thus, even if the running times are comparable for 3-dimensional packing
and 2-dimensional energy problems, the corresponding primal-dual integral
values are respectively \SI{7.2}{\percent} and~\SI{27.8}{\percent} smaller
for the automatic setting.
The automatic setting hence tends to converge faster than the remaining
methods.

In summary, the automatic setting is a competitive method that performs
better than its competitors sym3 and sym6.
While sym3 and sym6 can easily be incorporated into an optimization model
by adding inequalities to a problem formulation, the automatic setting
requires to implement the more sophisticated technique orbitopal reduction.
The latter might be a technical burden preventing inexperienced practitioners
to make use of this technique.
Due to our symmetry detection framework for reflection symmetries
combined with the heuristics for finding row and column symmetries,
however, a solver can automatically detect that the sophisticated symmetry
handling techniques are applicable.
Users can thus benefit from powerful symmetry handling techniques that are
built in a solver, and are not relying on handling reflection
symmetries on their own.

\paragraph{Discussion of Results for Max-Cut.}
Table~\ref{tab:symmetricresults} summarizes our experiments for, i.a., the
maxcut test set.
For each test set, it mentions in parentheses the total number of instances
of the test set as well as the number of instances that are solved by at
least one setting.
To compare different symmetry handling approaches, we use the mechanism
described in Section~\ref{sec:numericsSetup} and possibly disable some
techniques.
The first four columns of Table~\ref{tab:symmetricresults} indicate which
methods of this mechanism are used.
A cross in column ``sym.'', ``row+col'', ``refl.'', and ``simpl.'' respectively
indicates that symmetry handling is active, row and column symmetries
are detected, reflection symmetries are computed (if not, only permutation
symmetries are computed), and Inequality~\eqref{eq:simpleRefl} is applied when
possible.
To realize setting~$(\times,\times,\times,\times)$, we disable
lexicographic reduction to prevent that reflection symmetries are handled
both by lexicographic reduction and~\eqref{eq:simpleRefl}.
The remaining settings make use of lexicographic reduction if symmetry
handling is enabled.
Column groups ``time'' and ``gap'' report on the mean running times
and average gap for all instances and the solvable instances.
In contrast to the previous experiments, we rate an instance as solved if
its gap is zero (up to numerical tolerances).

\begin{table}
  \caption{Comparison of running times and gaps for symmetric instances.}
  \label{tab:symmetricresults}
  \centering
  \begin{tabular*}{\textwidth}{@{}c@{\;\;\extracolsep{\fill}}cccrrrrr@{}}
    \toprule
    \multicolumn{4}{c}{setting} & & \multicolumn{2}{c}{time} & \multicolumn{2}{c}{gap}\\
    \cmidrule{1-4} \cmidrule{6-7} \cmidrule{8-9}
    sym. & row+col & refl. & simp. & \# solved & all & solved & all & solved\\
    \midrule
    \multicolumn{7}{@{}l}{maxcut (118/41):}\\
             &          &          &          &  40 & 1220.19 &   42.61 &    35.37 &     0.05\\
    $\times$ &          &          &          &  40 & 1059.66 &   28.07 &    34.95 &     0.05\\
    $\times$ & $\times$ &          &          &  40 & 1058.43 &   27.97 &    34.94 &     0.05\\
    $\times$ & $\times$ & $\times$ &          &  41 & 1008.31 &   24.20 &    34.87 &     0.00\\
    $\times$ & $\times$ & $\times$ & $\times$ &  40 & 1070.91 &   28.96 &    35.33 &     0.08\\
    \midrule
    \multicolumn{7}{@{}l}{miplib2017 (59/17):}\\
             &          &          &          &  15 & 2753.85 &  255.53 &  4416.48 &   588.24\\
    $\times$ &          &          &          &  14 & 2881.57 &  299.22 &  4588.75 &   590.82\\
    $\times$ & $\times$ &          &          &  14 & 2883.01 &  299.74 &  4588.75 &   590.82\\
    $\times$ & $\times$ & $\times$ &          &  17 & 2883.41 &  299.89 &  4414.99 &     0.00\\
    $\times$ & $\times$ & $\times$ & $\times$ &  16 & 2782.88 &  265.03 &  4414.41 &     1.95\\
    \midrule
    \multicolumn{7}{@{}l}{minlplib (6/1):}\\
             &          &          &          &   1 & 1640.40 &    0.01 &  3380.72 &     0.00\\
    $\times$ &          &          &          &   1 & 1643.10 &    0.02 &  3379.81 &     0.00\\
    $\times$ & $\times$ &          &          &   1 & 1640.40 &    0.01 &  3375.20 &     0.00\\
    $\times$ & $\times$ & $\times$ &          &   1 & 1643.10 &    0.02 &  3370.42 &     0.00\\
    $\times$ & $\times$ & $\times$ & $\times$ &   1 & 1643.10 &    0.02 &  3370.42 &     0.00\\
    \midrule
    \multicolumn{7}{@{}l}{sat2002 (282/172):}\\
             &          &          &          & 150 &  371.53 &   55.05 &  4680.85 &  1279.07\\
    $\times$ &          &          &          & 152 &  328.78 &   44.90 &  4609.93 &  1162.79\\
    $\times$ & $\times$ &          &          & 153 &  330.29 &   45.24 &  4574.47 &  1104.65\\
    $\times$ & $\times$ & $\times$ &          & 169 &  270.32 &   32.33 &  4007.09 &   174.42\\
    $\times$ & $\times$ & $\times$ & $\times$ & 153 &  354.20 &   50.84 &  4574.47 &  1104.65\\
    \bottomrule
  \end{tabular*}
\end{table}

Comparing the running times of the different settings, we see that handling
reflection symmetries using $(\times,\times,\times,\phantom{\times})$ is
most effective.
It solves one instances more than the remaining settings and reduces the
running time in comparison to not detecting and handling reflection
symmetries by \SI{4.7}{\percent} on all instances and \SI{13.5}{\percent}
on the solvable instances.
The setting $(\times,\times,\times,\times)$, which disables lexicographic
reduction and instead applies~\eqref{eq:simpleRefl} when it is applicable,
still improves on the setting in which no symmetries are handled.
But in comparison to the other symmetry handling approaches, it is the least
performant.
Since the motivation of this setting was to compare the effect
of~\eqref{eq:simpleRefl} with lexicographic reduction, we also separately
considered the~25 instances for which~\eqref{eq:simpleRefl} is applicable
to see whether it has a positive effect there.
None of these instances could be solved within the time limit though.

\subsection{Results for Benchmarking Instances}
\label{sec:numericsBenchmark}

Besides the structured instances discussed in the previous section, we also
conducted experiments on general benchmarking instances.
The test sets that we considered are all instances from MIPLIB2017~\cite{MIPLIB2017} and
MINLPLIB~\cite{MINLPLIB}, as well as the submitted instances of the SAT 2002 Competition~\cite{SAT2002}.
To evaluate the impact of handling reflection symmetries, we removed all
instances from these test sets for which no reflection symmetries could be
detected.
We refer to the corresponding test sets as miplib2017, minlplib, and
sat2002, respectively.

In contrast to the structured instances, we cannot evaluate whether our
framework reliably detects reflection symmetries for benchmarking instances.
Our expectation was that reflection symmetries are rare for linear problems
(miplib2017) and arise frequently for nonlinear problems (minlplib) and SAT
problems.
Indeed, as Table~\ref{tab:statistics} shows, for~282 of the 819 instances
from sat2002, we could detect reflection symmetries, whereas we could find
only~60 instances from miplib2017 admitting reflection symmetries.
Among the~486 instances from minlplib, however, our framework could only
detect~6 instances admitting reflection symmetries.
This came as surprise to us, since MINLPLIB also contains instances corresponding to
geometric packing problems (instances whose names start with ``kall\_'').
Inspecting these instances revealed two explanations for not detecting
the reflection symmetries.
On the one hand, these instances already contain symmetry handling
inequalities.
On the other hand, in contrast to Example~\ref{ex:geompack}, the box in
which the objects need to be placed is not fixed.
Instead, one is looking for a box of minimal dimensions that can fit all
objects.
This is modeled asymmetrically by fixing the lower coordinate value and
introducing a variable to model the upper coordinate value of each
dimension.
That is, although the real world problem admits reflection symmetries, the
corresponding MINLP model is asymmetric.

In the following, we will therefore focus on the miplib2017 and sat2002
instances containing reflection symmetries, since the minlplib test set is
too small to draw reliable conclusions.
The running times are summarized in Table~\ref{tab:symmetricresults}.
Note that the table reports only on~59 instances although
Table~\ref{tab:statistics} shows that there are~60 instances with
reflection symmetries.
To ensure a fair comparison of the different methods, however, we removed
the instance ``tokyometro'' since all but one setting reached the memory
limit.

\paragraph{Discussion of MIPLIB2017}
For miplib2017, we observe that the
$(\times,\times,\times,\phantom{\times})$ setting performs best w.r.t.\ the
number of solved instances.
It can solve~17 instances, while just handling permutation symmetries can
only solve~14 instances, and handling no symmetries at all solves~15
instances.
Regarding the running time, however,
$(\times,\times,\times,\phantom{\times})$ and the settings only handling
permutation symmetries perform equally and are on all instances \SI{4.7}{\percent}
(on the solvable instances~\SI{17.1}{\percent}) slower than not handling
symmetries.
It thus seems that handling reflection symmetries can help
solving more instances, on average, however, it slows down the solving
process.
As such, it is not a surprise that the mean running time of
$(\times,\times,\times,\times)$ is better than the one of
$(\times,\times,\times,\phantom{\times})$.

To understand why not handling symmetries performs better than handling
symmetries, we compared the results for the~17 solvable instances for the
setting in which no symmetries are handled
and~$(\times,\times,\times,\phantom{\times})$.
The following three observations could be made:
\begin{enumerate*}[label=(\roman*)]
\item some instances are rather easy such that an improvement in running
  time is negligible;
\item for the two instances that cannot be solved when not handling
  symmetries, also $(\times, \times, \times, \phantom{\times})$ needed
  about~\SI{5900}{\second} and~\SI{6400}{\second}, respectively.
  That is, also when handling symmetries, the instances remain hard.
\item The dual bound after presolving is (almost) optimal, i.e.,
  it is sufficient to find an optimal solution.
  While the power of symmetry handling lies in pruning symmetric
  subproblems, which allows to more quickly improve the dual bound, it
  seems to hinder \scip in finding feasible or optimal solutions.
\end{enumerate*}
We conclude that, although handling symmetries on benchmarking instances
has a positive effect in general~\cite{PfetschRehn2019}, the characteristics
of instances from MIPLIB2017 that admit reflection symmetries make symmetry
handling less suited to enhance branch-and-bound for these instances.

The second question that arises is why the
setting~$(\times,\times,\times,\phantom{\times})$ has the same mean running
time as~$(\times,\times,\phantom{\times},\phantom{\times})$ although it
solves three more instances.
Inspecting the symmetries that are found by the two different settings, we
observed that the number of generators varies a lot between only detecting
permutation symmetries and also reflection symmetries.
For example, although the detected symmetry group for the instance
\texttt{neos-3004026-krka} is larger when detecting reflection symmetries
($\approx 10^{91.5}$ group elements in comparison to $\approx 10^{90.9}$
for permutation symmetries), the number of generators we get from
\code{bliss} is~35 for reflection symmetries and~64 for permutation symmetries.
When handling symmetries via lexicographic reduction, we thus lose a lot of
potential reductions when computing reflection symmetries.
Moreover, for the instance \texttt{neos-780889}, we obtain the same number
of generators corresponding to permutation symmetries; when handling
reflection symmetries, however, we detect less column/row symmetries.
That is, we miss the potential of specialized algorithms for column/row
symmetries.

For the three additionally solved instances when
handling reflection symmetries, we either find more generators (instance
\texttt{icir07\_tension}) or we detect more row/column symmetries
(instances \texttt{lectsched-1} and \texttt{tanglegram4}).
The explanation for the same mean running time thus indeed seems to be the
variability in the generators returned by \code{bliss}.

\paragraph{Discussion of SAT2002}

On the sat2002 test set, the most effective setting is
$(\times,\times,\times,\phantom{\times})$.
It solves~169 instances, and thus almost all solvable instances, within the
time limit and improves upon only handling permutation symmetries
by~\SI{17.8}{\percent}.
Taking Table~\ref{tab:statistics} into account, this behavior is not
surprising as at most~125 of the~292 reflection symmetric sat2002 instances
contain permutation symmetries.
That is, if reflection symmetries are not handled, a lot of instances
become asymmetric.

\begin{table}
  \caption{Comparison of running times and number of detected row/column symmetries for solvable sat2002 instances containing permutation symmetries.}
  \label{tab:symmetricresultsSAT}
  \centering
  \begin{tabular*}{\textwidth}{@{}c@{\;\;\extracolsep{\fill}}cccrrr@{}}
    \toprule
    \multicolumn{4}{c}{setting} & & &\\
    \cmidrule{1-4}
    sym. & row+col & refl. & simp. & \# solved & time & \# row/column symmetries\\
    \midrule
    \multicolumn{7}{@{}l}{all instances (76):}\\
             &          &          &          &  72 &   63.10 &    1.00\\
    $\times$ &          &          &          &  74 &   39.67 &   14.87\\
    $\times$ & $\times$ &          &          &  74 &   39.80 &   14.87\\
    $\times$ & $\times$ & $\times$ &          &  74 &   50.57 &    7.01\\
    $\times$ & $\times$ & $\times$ & $\times$ &  74 &   53.28 &    7.01\\
    \midrule
    \multicolumn{7}{@{}l}{feasible instances (27):}\\
             &          &          &          &  27 &   24.06 &    1.00\\
    $\times$ &          &          &          &  25 &   27.90 &    4.59\\
    $\times$ & $\times$ &          &          &  25 &   27.93 &    4.59\\
    $\times$ & $\times$ & $\times$ &          &  27 &   26.88 &    0.44\\
    $\times$ & $\times$ & $\times$ & $\times$ &  27 &   26.93 &    0.44\\
    \midrule
    \multicolumn{7}{@{}l}{infeasible instances (49):}\\
             &          &          &          &  45 &  106.55 &    1.00\\
    $\times$ &          &          &          &  49 &   48.09 &   20.53\\
    $\times$ & $\times$ &          &          &  49 &   48.31 &   20.53\\
    $\times$ & $\times$ & $\times$ &          &  47 &   71.38 &   10.63\\
    $\times$ & $\times$ & $\times$ & $\times$ &  47 &   77.27 &   10.63\\
    \bottomrule
  \end{tabular*}
\end{table}

To allow for a fair comparison between the different symmetry handling
settings, we therefore also considered the subset of all solvable sat2002
instances that contain proper permutation symmetries.
This results in~76 instances and the corresponding results are summarized
in Table~\ref{tab:symmetricresultsSAT}.
On these instances, we observe that handling reflection symmetries on top
of permutation symmetries decreases the performance by~\SI{27.5}{\percent},
and this effect is even more pronounced on the infeasible instances, for
which the running time increases by~\SI{48.4}{\percent}.
A possible explanation for this unexpected behavior is again the variance
in the generators of the symmetry groups reported by \code{bliss}.
While the mean number of row/column symmetries that are detected per instance
are about~20.5 when only detecting permutation symmetries, the number of
row/column symmetries drops to~10.6 when detecting reflection symmetries.
That is, when detecting reflection symmetries, the potential of handling
row/column symmetries by dedicated techniques cannot be exploited.

\subsection{Conclusion and Outlook}

In the introduction, we have formulated four main
goals~\eqref{goal1}--\eqref{goal4}, which also could be achieved in this
article:
Our abstract framework of symmetry detection graphs turned out to be a
flexible mechanism for detecting reflection symmetries in
MINLP and beyond, cf.\ Goal~\eqref{goal1}.
Our open-source implementation could be used to detect reflection
symmetries in many applications, cf.\ Goal~\eqref{goal3}, and the numerical
experiments showed that handling reflection symmetries can be crucial to
accelerate branch-and-bound for specific applications, cf.\
Goal~\eqref{goal4}.
Although we devised methods for handling reflection symmetries, cf.\
Goal~\eqref{goal2}, we noted that the performance improvement due to
handling symmetries heavily depends on the structure of the detected symmetry
groups.
Handling reflection symmetries thus might slow
down the solving process if this prevents our heuristics to detect
row/column symmetries.

This latter observation opens directions for future research.
As we noted in our experiments, the generators of symmetry groups returned
by symmetry detection tools such as \code{bliss} heavily depend on the
structure of the symmetry detection graphs.
Thus, based on the returned generators, our heuristics can fail to detect
row and column symmetries.
To circumvent this issue, it might be promising to develop alternative
approaches for detecting row and column symmetries that depend less on the
structure of generators.
Ideally, one would use an exact mechanism to detect row/column symmetries,
but detecting such symmetries is as hard as the graph isomorphism
problem~\cite{BertholdPfetsch2009}.
A possible future direction could thus be to exploit the specific structure
of the symmetry detection graphs to solve the graph isomorphism problem.

Moreover, for MIPLIB2017, we noted that some problems benefit from not
handling symmetries, because symmetry handling can hinder heuristics to
find feasible solutions.
A naive strategy for feasibility problems is thus to completely disable
symmetry handling.
For infeasible instances, however, this arguably slows down the solving
process, since a larger search space needs to be explored until
infeasibility is detected.
Instead, it could be interesting to investigate means to benefit from
handling symmetries in branch-and-bound, while removing the symmetry-based
restrictions in heuristics.

\medskip
\noindent
\emph{Acknowledgements}
The author thanks Marc E.\ Pfetsch for very valuable discussions on the
choice of the data structure for encoding symmetry detection graphs in
\scip as well as a thorough code review.
This publication is part of the project
``Local Symmetries for Global  Success'' with project number
OCENW.M.21.299 which is financed by the Dutch Research Council (NWO).

\bibliographystyle{spmpsci}      

\appendix

\section{Overview of Important Functions to Apply Our Symmetry Detection Framework}
\label{app:functions}

This appendix provides an overview of the most important functions needed
to extend an SDG within a \scip symmetry detection callback.
Since our implementation of SDGs allows for four different types of nodes,
we have different functions for adding these nodes:

\medskip
\begin{tabular*}{\textwidth}{@{}lL{6cm}@{}}
  SCIPaddSymgraphOpnode() & adds an operator node to an SDG;\\
  SCIPaddSymgraphValnode() & adds a numerical value node to an SDG;\\
  SCIPaddSymgraphConsnode() & adds a constraint node to an SDG.
\end{tabular*}

\medskip
\noindent
Recall that we do not allow to add variable nodes to an SDG, because \scip
ensures that every SDG contains all necessary variable nodes.
Instead, the indices of variable nodes can be accessed via the functions

\medskip
\begin{tabular*}{\textwidth}{@{}lL{6cm}@{}}
  SCIPgetSymgraphVarnodeidx() & returns the index of the node
  corresponding to a given variable;\\
  SCIPgetSymgraphNegatedVarnodeidx() & returns the index of the node
  corresponding to a negated/reflected variable.
\end{tabular*}

\medskip
\noindent
To add edges to a graph, the function

\medskip
\begin{tabular*}{\textwidth}{@{}lL{6cm}@{}}
  SCIPaddSymgraphEdge() & adds an edge between two existing nodes of an SDG
\end{tabular*}

\medskip
\noindent
can be used.

To simplify the usage of SDGs, we also provide two functions that add gadgets
for certain variable structures to an SDG:

\medskip
\begin{tabular*}{\textwidth}{@{}lL{5cm}@{}}
  SCIPextendPermsymDetectionGraphLinear() & adds a gadget for a linear
                                            expression~$\sprod{a}{x} + b$
                                            to an SDG;\\
  SCIPaddSymgraphVarAggregation() & adds a gadget for aggregated variables
                                    to an SDG.
\end{tabular*}

\medskip
\noindent
The second function has been introduced, since we require that no
aggregated or fixed variables are present in an SDG.

\clearpage
\section{Detailed Numerical Results}
\label{app:moretimes}

In this appendix, we provide detailed numerical results for the tested
problem classes.
Tables~\ref{tab:detailspacking2}--\ref{tab:detailsenergy3} report on
the running times and primal-dual integrals for each instance of the 2- and
3-dimensional packing, kissing number, and energy problems that we
discussed in Section~\ref{sec:numericsApplications}.
The number of items corresponds to the number of balls, spheres, and points
in these respective problems, whereas the settings refer to the settings
sym0--sym6 and the automatic setting as described in
Section~\ref{sec:numericsApplications}.

\begin{table}[h]
  \caption{Running times and primal-dual integrals for packing test set and dimension 2.}
  \label{tab:detailspacking2}
  \centering
  \scriptsize
  \begin{tabular*}{\textwidth}{@{}l@{\;\;\extracolsep{\fill}}rrrrrrrr@{}}
    \toprule
    & \multicolumn{8}{c}{setting}\\
    \cmidrule{2-9}
    \# items & sym0 & sym1 & sym2 & sym3 & sym4 & sym5 & sym6 & auto.\\
    \midrule
    \multicolumn{9}{@{}l}{running time in seconds:}\\
     3 & \num{   0.12} & \num{   0.12} & \num{   0.11} & \num{   0.05} & \num{   0.09} & \num{   0.07} & \num{   0.09} & \num{   0.09}\\
     4 & \num{   2.84} & \num{   1.94} & \num{   0.67} & \num{   0.30} & \num{   0.44} & \num{   0.38} & \num{   0.29} & \num{   0.28}\\
     5 & \num{   0.79} & \num{   0.59} & \num{   0.38} & \num{   0.17} & \num{   0.30} & \num{   0.31} & \num{   0.22} & \num{   0.21}\\
     6 & \num{  43.09} & \num{  25.56} & \num{   7.41} & \num{   0.61} & \num{   0.50} & \num{   0.40} & \num{   0.70} & \num{   0.68}\\
     7 & \num{7200.00} & \num{7200.00} & \num{6606.63} & \num{  22.81} & \num{ 201.83} & \num{  16.16} & \num{  11.01} & \num{  16.23}\\
     8 & \num{7200.00} & \num{7200.00} & \num{4352.59} & \num{  10.14} & \num{  70.32} & \num{  22.82} & \num{   6.68} & \num{  12.14}\\
     9 & \num{7200.00} & \num{7200.00} & \num{7200.00} & \num{ 644.78} & \num{7200.00} & \num{ 249.99} & \num{ 365.62} & \num{ 103.95}\\
    10 & \num{7200.00} & \num{7200.00} & \num{7200.00} & \num{ 267.77} & \num{7200.00} & \num{ 173.67} & \num{ 153.78} & \num{  61.42}\\
    11 & \num{7200.00} & \num{7200.00} & \num{7200.00} & \num{7200.00} & \num{7200.00} & \num{7200.00} & \num{7200.00} & \num{7200.00}\\
    12 & \num{7200.00} & \num{7200.00} & \num{7200.00} & \num{7200.00} & \num{7200.00} & \num{7200.00} & \num{7200.00} & \num{7200.00}\\
    13 & \num{7200.00} & \num{7200.00} & \num{7200.00} & \num{7200.00} & \num{7200.00} & \num{ 358.37} & \num{ 351.27} & \num{ 301.84}\\
    14 & \num{7200.00} & \num{7200.00} & \num{7200.00} & \num{7200.00} & \num{7200.00} & \num{7200.00} & \num{7200.00} & \num{7200.00}\\
    \midrule
    \multicolumn{9}{@{}l}{primal-dual integral:}\\
     3 & \num{      5} & \num{      5} & \num{      6} & \num{      2} & \num{      3} & \num{      2} & \num{      3} & \num{      4}\\
     4 & \num{     35} & \num{     28} & \num{     16} & \num{     11} & \num{      9} & \num{      9} & \num{      8} & \num{      8}\\
     5 & \num{     29} & \num{     23} & \num{     19} & \num{     10} & \num{     14} & \num{     14} & \num{     12} & \num{     11}\\
     6 & \num{    802} & \num{    527} & \num{    151} & \num{     22} & \num{     12} & \num{      9} & \num{     26} & \num{     24}\\
     7 & \num{ 120899} & \num{ 109585} & \num{  31757} & \num{    210} & \num{    941} & \num{    109} & \num{    133} & \num{    123}\\
     8 & \num{ 145839} & \num{ 118007} & \num{  52138} & \num{    285} & \num{   1489} & \num{    227} & \num{     90} & \num{    120}\\
     9 & \num{ 234515} & \num{ 269194} & \num{ 152641} & \num{   6449} & \num{  67037} & \num{   1741} & \num{   2826} & \num{    667}\\
    10 & \num{ 266727} & \num{ 224203} & \num{ 229408} & \num{   5454} & \num{  73878} & \num{   2567} & \num{   1027} & \num{    754}\\
    11 & \num{ 355432} & \num{ 344742} & \num{ 334154} & \num{ 154772} & \num{ 102571} & \num{  66800} & \num{  74434} & \num{  50360}\\
    12 & \num{ 361376} & \num{ 353135} & \num{ 351864} & \num{ 199890} & \num{ 128913} & \num{  20578} & \num{  77572} & \num{  68018}\\
    13 & \num{ 368337} & \num{ 354376} & \num{ 353221} & \num{ 190939} & \num{  98344} & \num{   7071} & \num{   8769} & \num{   6855}\\
    14 & \num{ 408953} & \num{ 408172} & \num{ 405794} & \num{ 277031} & \num{ 200833} & \num{  46888} & \num{ 128374} & \num{ 105141}\\
    \bottomrule
  \end{tabular*}
\end{table}
\begin{table}
  \caption{Running times and primal-dual integrals for packing test set and dimension 3.}
  \label{tab:detailspacking3}
  \centering
  \scriptsize
  \begin{tabular*}{\textwidth}{@{}l@{\;\;\extracolsep{\fill}}rrrrrrrr@{}}
    \toprule
    & \multicolumn{8}{c}{setting}\\
    \cmidrule{2-9}
    \# items & sym0 & sym1 & sym2 & sym3 & sym4 & sym5 & sym6 & auto.\\
    \midrule
    \multicolumn{9}{@{}l}{running time in seconds:}\\
     3 & \num{   2.09} & \num{   0.93} & \num{   0.35} & \num{   0.29} & \num{   0.22} & \num{   0.25} & \num{   0.27} & \num{   0.41}\\
     4 & \num{   1.49} & \num{   0.85} & \num{   0.29} & \num{   0.29} & \num{   0.31} & \num{   0.31} & \num{   0.29} & \num{   0.28}\\
     5 & \num{7200.00} & \num{7200.00} & \num{5921.24} & \num{ 380.15} & \num{ 626.00} & \num{ 414.19} & \num{  67.69} & \num{  73.69}\\
     6 & \num{6663.82} & \num{1961.23} & \num{ 341.28} & \num{  10.73} & \num{  34.86} & \num{  13.46} & \num{   8.40} & \num{   5.87}\\
     7 & \num{7200.00} & \num{7200.00} & \num{7200.00} & \num{ 631.20} & \num{2443.29} & \num{ 645.28} & \num{ 287.70} & \num{ 395.11}\\
     8 & \num{7200.00} & \num{7200.00} & \num{7200.00} & \num{7200.00} & \num{7200.00} & \num{7200.00} & \num{7200.00} & \num{7200.00}\\
     9 & \num{7200.00} & \num{7200.00} & \num{7200.00} & \num{7200.00} & \num{7200.00} & \num{7200.00} & \num{7200.00} & \num{7200.00}\\
    10 & \num{7200.00} & \num{7200.00} & \num{7200.00} & \num{7200.00} & \num{7200.00} & \num{7200.00} & \num{7200.00} & \num{7200.00}\\
    11 & \num{7200.00} & \num{7200.00} & \num{7200.00} & \num{7200.00} & \num{7200.00} & \num{7200.00} & \num{7200.00} & \num{7200.00}\\
    12 & \num{7200.00} & \num{7200.00} & \num{7200.00} & \num{7200.00} & \num{7200.00} & \num{7200.00} & \num{7200.00} & \num{7200.00}\\
    13 & \num{7200.00} & \num{7200.00} & \num{7200.00} & \num{7200.00} & \num{7200.00} & \num{7200.00} & \num{7200.00} & \num{7200.00}\\
    14 & \num{7200.00} & \num{7200.00} & \num{7200.00} & \num{7200.00} & \num{7200.00} & \num{7200.00} & \num{7200.00} & \num{7200.00}\\
    \midrule
    \multicolumn{9}{@{}l}{primal-dual integral:}\\
     3 & \num{     17} & \num{     11} & \num{      9} & \num{      6} & \num{      5} & \num{      6} & \num{      7} & \num{      9}\\
     4 & \num{     19} & \num{     12} & \num{      5} & \num{     11} & \num{     10} & \num{     10} & \num{      9} & \num{      8}\\
     5 & \num{  52376} & \num{  22897} & \num{   5366} & \num{    437} & \num{    566} & \num{    421} & \num{     86} & \num{     88}\\
     6 & \num{  48670} & \num{  15778} & \num{   2794} & \num{    210} & \num{    474} & \num{    128} & \num{    126} & \num{     94}\\
     7 & \num{ 154538} & \num{ 164221} & \num{ 110088} & \num{   6448} & \num{  10830} & \num{   3605} & \num{   2901} & \num{   2755}\\
     8 & \num{ 259241} & \num{ 231341} & \num{ 176139} & \num{ 133512} & \num{ 126275} & \num{ 120202} & \num{ 106350} & \num{ 106084}\\
     9 & \num{ 281270} & \num{ 271892} & \num{ 224430} & \num{ 154967} & \num{  92668} & \num{  66342} & \num{  79984} & \num{  69743}\\
    10 & \num{ 319452} & \num{ 304708} & \num{ 285611} & \num{ 206888} & \num{ 154403} & \num{ 158242} & \num{ 138025} & \num{ 127942}\\
    11 & \num{ 342112} & \num{ 334613} & \num{ 314555} & \num{ 251442} & \num{ 144333} & \num{ 128962} & \num{ 117638} & \num{ 115709}\\
    12 & \num{ 361785} & \num{ 345513} & \num{ 334208} & \num{ 267405} & \num{ 171725} & \num{ 163406} & \num{ 166023} & \num{ 145259}\\
    13 & \num{ 362833} & \num{ 355450} & \num{ 341441} & \num{ 277927} & \num{ 143289} & \num{ 173997} & \num{ 127984} & \num{ 115555}\\
    14 & \num{ 373245} & \num{ 388226} & \num{ 343286} & \num{ 299399} & \num{ 227500} & \num{ 207009} & \num{ 208462} & \num{ 153976}\\
    \bottomrule
  \end{tabular*}
\end{table}

\begin{table}
  \caption{Running times and primal-dual integrals for kissing test set and dimension 2.}
  \label{tab:detailskissing2}
  \centering
  \scriptsize
  \begin{tabular*}{\textwidth}{@{}l@{\;\;\extracolsep{\fill}}rrrrrrrr@{}}
    \toprule
    & \multicolumn{8}{c}{setting}\\
    \cmidrule{2-9}
    \# items & sym0 & sym1 & sym2 & sym3 & sym4 & sym5 & sym6 & auto.\\
    \midrule
    \multicolumn{9}{@{}l}{running time in seconds:}\\
     3 & \num{   0.02} & \num{   0.03} & \num{   0.08} & \num{   0.01} & \num{   0.06} & \num{   0.02} & \num{   0.04} & \num{   0.04}\\
     4 & \num{   0.07} & \num{   0.18} & \num{   0.16} & \num{   0.14} & \num{   0.10} & \num{   0.15} & \num{   0.16} & \num{   0.17}\\
     5 & \num{   0.06} & \num{   0.35} & \num{   0.28} & \num{   0.17} & \num{   0.39} & \num{   0.29} & \num{   0.17} & \num{   0.17}\\
     6 & \num{   0.16} & \num{   0.48} & \num{   0.47} & \num{   0.23} & \num{   0.31} & \num{   0.61} & \num{   0.61} & \num{   0.26}\\
     7 & \num{7200.00} & \num{7200.00} & \num{1136.23} & \num{   2.36} & \num{  22.42} & \num{   3.62} & \num{   1.75} & \num{   2.72}\\
     8 & \num{7200.00} & \num{7200.00} & \num{7200.00} & \num{   5.89} & \num{ 340.06} & \num{  18.33} & \num{   6.34} & \num{   6.11}\\
     9 & \num{7200.00} & \num{7200.00} & \num{7200.00} & \num{  11.85} & \num{ 451.55} & \num{   9.87} & \num{   7.22} & \num{   4.50}\\
    10 & \num{7200.00} & \num{7200.00} & \num{7200.00} & \num{  51.76} & \num{7200.00} & \num{  86.73} & \num{  54.52} & \num{  17.89}\\
    11 & \num{7200.00} & \num{7200.00} & \num{7200.00} & \num{ 175.53} & \num{7200.00} & \num{ 355.42} & \num{ 126.20} & \num{  29.58}\\
    12 & \num{7200.00} & \num{7200.00} & \num{7200.00} & \num{ 448.92} & \num{7200.00} & \num{1408.36} & \num{1578.97} & \num{ 380.60}\\
    13 & \num{7200.00} & \num{7200.00} & \num{7200.00} & \num{1928.51} & \num{7200.00} & \num{3137.21} & \num{2976.38} & \num{ 116.21}\\
    14 & \num{7200.00} & \num{7200.00} & \num{7200.00} & \num{5501.44} & \num{7200.00} & \num{7200.00} & \num{7200.00} & \num{ 849.15}\\
    \midrule
    \multicolumn{9}{@{}l}{primal-dual integral:}\\
     3 & \num{      2} & \num{      3} & \num{      7} & \num{      1} & \num{      6} & \num{      2} & \num{      2} & \num{      3}\\
     4 & \num{      5} & \num{     14} & \num{     16} & \num{      8} & \num{      9} & \num{     11} & \num{      9} & \num{     10}\\
     5 & \num{      6} & \num{     30} & \num{     24} & \num{      9} & \num{     31} & \num{     29} & \num{     14} & \num{     14}\\
     6 & \num{      9} & \num{     42} & \num{     45} & \num{     20} & \num{     25} & \num{     47} & \num{     33} & \num{     17}\\
     7 & \num{  38404} & \num{  21079} & \num{   5254} & \num{     99} & \num{    209} & \num{    159} & \num{     54} & \num{     74}\\
     8 & \num{ 245267} & \num{ 213439} & \num{ 101851} & \num{    264} & \num{   1873} & \num{    651} & \num{    149} & \num{    201}\\
     9 & \num{ 313339} & \num{ 306493} & \num{ 286624} & \num{    342} & \num{   1978} & \num{    193} & \num{    159} & \num{    205}\\
    10 & \num{ 392407} & \num{ 414595} & \num{ 404572} & \num{   1598} & \num{  94434} & \num{   1503} & \num{   1095} & \num{    406}\\
    11 & \num{ 491427} & \num{ 491463} & \num{ 484365} & \num{   4463} & \num{ 104650} & \num{   5516} & \num{   1994} & \num{    682}\\
    12 & \num{ 527099} & \num{ 527103} & \num{ 524806} & \num{   9572} & \num{ 239255} & \num{  13014} & \num{  18173} & \num{   5704}\\
    13 & \num{ 555106} & \num{ 555107} & \num{ 555094} & \num{  45741} & \num{ 152492} & \num{  34012} & \num{  29414} & \num{   1369}\\
    14 & \num{ 577444} & \num{ 578765} & \num{ 577441} & \num{ 130808} & \num{ 177820} & \num{ 328974} & \num{ 121905} & \num{  10812}\\
    \bottomrule
  \end{tabular*}
\end{table}
\begin{table}
  \caption{Running times and primal-dual integrals for kissing test set and dimension 3.}
  \label{tab:detailskissing3}
  \centering
  \scriptsize
  \begin{tabular*}{\textwidth}{@{}l@{\;\;\extracolsep{\fill}}rrrrrrrr@{}}
    \toprule
    & \multicolumn{8}{c}{setting}\\
    \cmidrule{2-9}
    \# items & sym0 & sym1 & sym2 & sym3 & sym4 & sym5 & sym6 & auto.\\
    \midrule
    \multicolumn{9}{@{}l}{running time in seconds:}\\
     3 & \num{   0.02} & \num{   0.02} & \num{   0.03} & \num{   0.02} & \num{   0.05} & \num{   0.02} & \num{   0.01} & \num{   0.02}\\
     4 & \num{   0.02} & \num{   0.03} & \num{   0.04} & \num{   0.03} & \num{   0.05} & \num{   0.03} & \num{   0.02} & \num{   0.02}\\
     5 & \num{   0.04} & \num{   0.04} & \num{   0.04} & \num{   0.01} & \num{   0.07} & \num{   0.02} & \num{   0.02} & \num{   0.02}\\
     6 & \num{   0.04} & \num{   0.03} & \num{   0.03} & \num{   0.03} & \num{   0.07} & \num{   0.02} & \num{   0.03} & \num{   0.02}\\
     7 & \num{   0.05} & \num{   0.05} & \num{   0.06} & \num{   0.04} & \num{   0.08} & \num{   0.04} & \num{   0.03} & \num{   0.04}\\
     8 & \num{   0.04} & \num{   0.04} & \num{   0.07} & \num{   0.03} & \num{   0.06} & \num{   0.03} & \num{   0.05} & \num{   0.08}\\
     9 & \num{   0.07} & \num{   0.08} & \num{   0.07} & \num{   0.04} & \num{   0.08} & \num{   1.37} & \num{   0.07} & \num{   0.03}\\
    10 & \num{   0.09} & \num{   0.09} & \num{   0.08} & \num{   0.05} & \num{   0.13} & \num{   0.07} & \num{   0.73} & \num{   0.05}\\
    11 & \num{   0.10} & \num{   0.13} & \num{   0.54} & \num{   0.07} & \num{   0.09} & \num{   1.97} & \num{   0.12} & \num{   0.92}\\
    12 & \num{   0.10} & \num{   1.09} & \num{   0.14} & \num{   0.06} & \num{   0.15} & \num{   2.08} & \num{   0.20} & \num{   0.07}\\
    13 & \num{7200.00} & \num{7200.00} & \num{7200.00} & \num{7200.00} & \num{7200.00} & \num{7200.00} & \num{7200.00} & \num{7200.00}\\
    14 & \num{7200.00} & \num{7200.00} & \num{7200.00} & \num{7200.00} & \num{7200.00} & \num{7200.00} & \num{7200.00} & \num{7200.00}\\
    \midrule
    \multicolumn{9}{@{}l}{primal-dual integral:}\\
     3 & \num{      2} & \num{      2} & \num{      3} & \num{      2} & \num{      5} & \num{      2} & \num{      1} & \num{      2}\\
     4 & \num{      2} & \num{      3} & \num{      4} & \num{      3} & \num{      5} & \num{      3} & \num{      2} & \num{      2}\\
     5 & \num{      4} & \num{      4} & \num{      4} & \num{      1} & \num{      7} & \num{      2} & \num{      2} & \num{      2}\\
     6 & \num{      4} & \num{      3} & \num{      3} & \num{      3} & \num{      7} & \num{      2} & \num{      3} & \num{      2}\\
     7 & \num{      5} & \num{      4} & \num{      6} & \num{      4} & \num{      8} & \num{      4} & \num{      3} & \num{      4}\\
     8 & \num{      4} & \num{      4} & \num{      7} & \num{      3} & \num{      6} & \num{      3} & \num{      5} & \num{      8}\\
     9 & \num{      7} & \num{      8} & \num{      7} & \num{      4} & \num{      8} & \num{    137} & \num{      7} & \num{      3}\\
    10 & \num{      9} & \num{      9} & \num{      8} & \num{      5} & \num{     13} & \num{      7} & \num{     18} & \num{      5}\\
    11 & \num{     10} & \num{     13} & \num{     54} & \num{      7} & \num{      9} & \num{    197} & \num{      9} & \num{     33}\\
    12 & \num{     10} & \num{     22} & \num{     14} & \num{      6} & \num{     15} & \num{    193} & \num{     11} & \num{      7}\\
    13 & \num{  61672} & \num{  61411} & \num{  61412} & \num{  61568} & \num{  61612} & \num{  61583} & \num{  62150} & \num{  62493}\\
    14 & \num{  92105} & \num{  92101} & \num{  92102} & \num{  93511} & \num{  92106} & \num{  92116} & \num{  93969} & \num{  92138}\\
    \bottomrule
  \end{tabular*}
\end{table}

\begin{table}
  \caption{Running times and primal-dual integrals for energy test set and dimension 2.}
  \label{tab:detailsenergy2}
  \centering
  \scriptsize
  \begin{tabular*}{\textwidth}{@{}l@{\;\;\extracolsep{\fill}}rrrrrrrr@{}}
    \toprule
    & \multicolumn{8}{c}{setting}\\
    \cmidrule{2-9}
    \# items & sym0 & sym1 & sym2 & sym3 & sym4 & sym5 & sym6 & auto.\\
    \midrule
    \multicolumn{9}{@{}l}{running time in seconds:}\\
     3 & \num{   2.17} & \num{   0.96} & \num{   0.37} & \num{   0.14} & \num{   0.31} & \num{   0.30} & \num{   0.20} & \num{   0.14}\\
     4 & \num{  28.92} & \num{  14.75} & \num{   4.69} & \num{   0.90} & \num{   3.64} & \num{   1.87} & \num{   1.36} & \num{   1.35}\\
     5 & \num{ 637.64} & \num{ 330.00} & \num{  80.48} & \num{   3.96} & \num{   6.50} & \num{   3.74} & \num{   2.06} & \num{   2.00}\\
     6 & \num{7200.00} & \num{7200.00} & \num{7200.00} & \num{  20.58} & \num{ 169.95} & \num{  36.29} & \num{  19.36} & \num{  19.48}\\
     7 & \num{7200.00} & \num{7200.00} & \num{7200.00} & \num{ 131.63} & \num{1669.89} & \num{ 171.30} & \num{  95.21} & \num{  80.70}\\
     8 & \num{7200.00} & \num{7200.00} & \num{7200.00} & \num{ 733.38} & \num{7200.00} & \num{2358.56} & \num{7200.00} & \num{7200.00}\\
     9 & \num{7200.00} & \num{7200.00} & \num{7200.00} & \num{7200.00} & \num{7200.00} & \num{3555.87} & \num{3074.32} & \num{1258.01}\\
    10 & \num{7200.00} & \num{7200.00} & \num{7200.00} & \num{7200.00} & \num{7200.00} & \num{7200.00} & \num{7200.00} & \num{7200.00}\\
    11 & \num{7200.00} & \num{7200.00} & \num{7200.00} & \num{7200.00} & \num{7200.00} & \num{7200.00} & \num{7200.00} & \num{7200.00}\\
    12 & \num{7200.00} & \num{7200.00} & \num{7200.00} & \num{7200.00} & \num{7200.00} & \num{7200.00} & \num{7200.00} & \num{7200.00}\\
    13 & \num{7200.00} & \num{7200.00} & \num{7200.00} & \num{7200.00} & \num{7200.00} & \num{7200.00} & \num{7200.00} & \num{7200.00}\\
    14 & \num{7200.00} & \num{7200.00} & \num{7200.00} & \num{7200.00} & \num{7200.00} & \num{7200.00} & \num{7200.00} & \num{7200.00}\\
    \midrule
    \multicolumn{9}{@{}l}{primal-dual integral:}\\
     3 & \num{      6} & \num{      4} & \num{      3} & \num{      3} & \num{      2} & \num{      3} & \num{      5} & \num{      3}\\
     4 & \num{     59} & \num{     30} & \num{     22} & \num{      5} & \num{     14} & \num{     12} & \num{     11} & \num{     11}\\
     5 & \num{   1048} & \num{    733} & \num{    175} & \num{     38} & \num{     25} & \num{     21} & \num{     17} & \num{     13}\\
     6 & \num{  27577} & \num{  16940} & \num{  15275} & \num{    105} & \num{    339} & \num{     95} & \num{     67} & \num{     60}\\
     7 & \num{ 103265} & \num{  91918} & \num{  71028} & \num{    558} & \num{   3115} & \num{    460} & \num{    300} & \num{    221}\\
     8 & \num{ 165382} & \num{ 159741} & \num{ 138685} & \num{   3004} & \num{  61183} & \num{   5533} & \num{   8552} & \num{   4083}\\
     9 & \num{ 209730} & \num{ 204553} & \num{ 183465} & \num{  13291} & \num{  78012} & \num{   9411} & \num{   7622} & \num{   3704}\\
    10 & \num{ 248606} & \num{ 244166} & \num{ 225948} & \num{  56374} & \num{ 138046} & \num{  64303} & \num{  60746} & \num{  38110}\\
    11 & \num{ 268921} & \num{ 265619} & \num{ 246456} & \num{ 101186} & \num{ 161925} & \num{  98961} & \num{  98287} & \num{  76904}\\
    12 & \num{ 284270} & \num{ 286673} & \num{ 276693} & \num{ 153360} & \num{ 201235} & \num{ 146922} & \num{ 145253} & \num{ 126222}\\
    13 & \num{ 303207} & \num{ 301486} & \num{ 290631} & \num{ 178830} & \num{ 210680} & \num{ 163230} & \num{ 162411} & \num{ 123375}\\
    14 & \num{ 313677} & \num{ 311820} & \num{ 308313} & \num{ 214843} & \num{ 243370} & \num{ 206209} & \num{ 200767} & \num{ 165437}\\
    \bottomrule
  \end{tabular*}
\end{table}
\begin{table}
  \caption{Running times and primal-dual integrals for energy test set and dimension 3.}
  \label{tab:detailsenergy3}
  \centering
  \scriptsize
  \begin{tabular*}{\textwidth}{@{}l@{\;\;\extracolsep{\fill}}rrrrrrrr@{}}
    \toprule
    & \multicolumn{8}{c}{setting}\\
    \cmidrule{2-9}
    \# items & sym0 & sym1 & sym2 & sym3 & sym4 & sym5 & sym6 & auto.\\
    \midrule
    \multicolumn{9}{@{}l}{running time in seconds:}\\
     3 & \num{7200.00} & \num{1904.11} & \num{ 179.29} & \num{  73.14} & \num{ 358.43} & \num{ 357.69} & \num{  24.35} & \num{ 772.85}\\
     4 & \num{7200.00} & \num{7200.00} & \num{7200.00} & \num{1889.64} & \num{7200.00} & \num{7200.00} & \num{1448.27} & \num{1602.78}\\
     5 & \num{7200.00} & \num{7200.00} & \num{7200.00} & \num{7200.00} & \num{7200.00} & \num{7200.00} & \num{7200.00} & \num{7200.00}\\
     6 & \num{7200.00} & \num{7200.00} & \num{7200.00} & \num{7200.00} & \num{7200.00} & \num{7200.00} & \num{7200.00} & \num{7200.00}\\
     7 & \num{7200.00} & \num{7200.00} & \num{7200.00} & \num{7200.00} & \num{7200.00} & \num{7200.00} & \num{7200.00} & \num{7200.00}\\
     8 & \num{7200.00} & \num{7200.00} & \num{7200.00} & \num{7200.00} & \num{7200.00} & \num{7200.00} & \num{7200.00} & \num{7200.00}\\
     9 & \num{7200.00} & \num{7200.00} & \num{7200.00} & \num{7200.00} & \num{7200.00} & \num{7200.00} & \num{7200.00} & \num{7200.00}\\
    10 & \num{7200.00} & \num{7200.00} & \num{7200.00} & \num{7200.00} & \num{7200.00} & \num{7200.00} & \num{7200.00} & \num{7200.00}\\
    11 & \num{7200.00} & \num{7200.00} & \num{7200.00} & \num{7200.00} & \num{7200.00} & \num{7200.00} & \num{7200.00} & \num{7200.00}\\
    12 & \num{7200.00} & \num{7200.00} & \num{7200.00} & \num{7200.00} & \num{7200.00} & \num{7200.00} & \num{7200.00} & \num{7200.00}\\
    13 & \num{7200.00} & \num{7200.00} & \num{7200.00} & \num{7200.00} & \num{7200.00} & \num{7200.00} & \num{7200.00} & \num{7200.00}\\
    14 & \num{7200.00} & \num{7200.00} & \num{7200.00} & \num{7200.00} & \num{7200.00} & \num{7200.00} & \num{7200.00} & \num{7200.00}\\
    \midrule
    \multicolumn{9}{@{}l}{primal-dual integral:}\\
     3 & \num{   1528} & \num{    447} & \num{     44} & \num{     23} & \num{     80} & \num{     80} & \num{     11} & \num{    186}\\
     4 & \num{  14737} & \num{   6781} & \num{   2051} & \num{    491} & \num{   2658} & \num{   1855} & \num{    385} & \num{    408}\\
     5 & \num{  77151} & \num{  59093} & \num{  36746} & \num{  12087} & \num{  16408} & \num{  12026} & \num{   5531} & \num{   5184}\\
     6 & \num{ 121289} & \num{ 104487} & \num{  84450} & \num{  41473} & \num{  61835} & \num{  46654} & \num{  31565} & \num{  31476}\\
     7 & \num{ 157158} & \num{ 143920} & \num{ 126967} & \num{  86664} & \num{ 100511} & \num{  88191} & \num{  72522} & \num{  71813}\\
     8 & \num{ 178762} & \num{ 168522} & \num{ 156057} & \num{ 120629} & \num{ 133847} & \num{ 120473} & \num{ 109227} & \num{ 108159}\\
     9 & \num{ 195888} & \num{ 186304} & \num{ 172755} & \num{ 146157} & \num{ 134636} & \num{ 128682} & \num{ 116954} & \num{ 115462}\\
    10 & \num{ 208562} & \num{ 199569} & \num{ 187568} & \num{ 164420} & \num{ 158820} & \num{ 151279} & \num{ 141834} & \num{ 140935}\\
    11 & \num{ 219253} & \num{ 213879} & \num{ 201490} & \num{ 185224} & \num{ 174839} & \num{ 168161} & \num{ 159821} & \num{ 159343}\\
    12 & \num{ 225733} & \num{ 223219} & \num{ 210340} & \num{ 192819} & \num{ 188728} & \num{ 182468} & \num{ 175968} & \num{ 174244}\\
    13 & \num{ 233884} & \num{ 229828} & \num{ 221233} & \num{ 205331} & \num{ 197518} & \num{ 190162} & \num{ 185491} & \num{ 184832}\\
    14 & \num{ 240182} & \num{ 238349} & \num{ 227048} & \num{ 216176} & \num{ 209512} & \num{ 202497} & \num{ 198844} & \num{ 198483}\\
    \bottomrule
  \end{tabular*}
\end{table}
\clearpage


\end{document}